\documentclass[12pt,leqno]{amsart}

\usepackage{color}
\usepackage{amssymb,amsmath,amsthm,amscd,amsfonts}
\usepackage{enumerate}
\usepackage[all]{xy}
\usepackage{hyperref,cleveref,graphics,mathrsfs}

\numberwithin{equation}{section}
\def\hangbox to #1 #2{\vskip3pt\hangindent #1\noindent \hbox to #1{#2}$\!\!$}

\usepackage{hyperref,cleveref,amssymb,mathtools}

\theoremstyle{plain}
\newtheorem{theorem}{Theorem}[section]
\newtheorem{proposition}[theorem]{Proposition}
\newtheorem{corollary}[theorem]{Corollary}
\newtheorem{lemma}[theorem]{Lemma}

\theoremstyle{definition}
\newtheorem{remark}[theorem]{Remark}

\newcommand{\biggnorm}[1]{\biggl\lVert#1\biggr\rVert}

\DeclareSymbolFont{bbold}{U}{bbold}{m}{n}
\DeclareSymbolFontAlphabet{\mathbbold}{bbold}

\def\sfrac#1#2{\kern.1em\raise.5ex\hbox{$#1$}
        \kern-.1em/\kern-.05em\lower.25ex\hbox{$#2$}}

\newcommand{\RE}{\mathrm{Re}}

\newcommand{\fw}{\text{\fw}}

\title[Global well-posedness for the gDNLS]{Global well-posedness for the generalized derivative nonlinear Schr\"odinger equation}
\author{Ben Pineau}
\address{Courant Institute for Mathematical Sciences\\
New York University
} \email{brp305@nyu.edu}
\author{Mitchell~A.\ Taylor}
\address{Department of Mathematics\\
ETH Z\"urich, Ramistrasse 101, 8092 Z\"urich, Switzerland.
} \email{mitchell.taylor@math.ethz.ch}

\begin{document}

\maketitle
\allowdisplaybreaks

 \begin{abstract}
In this article we study the well-posedness of the generalized derivative nonlinear Schr\"odinger equation (gDNLS)
$$iu_t+u_{xx}=i|u|^{2\sigma}u_x,$$
for small powers $\sigma$. We analyze this equation at both low and high regularity, and are able to establish global well-posedness in $H^s$ when $s\in [1,4\sigma)$ and $\sigma \in (\frac{\sqrt{3}}{2},1)$.  Our result when $s=1$ is particularly relevant because it corresponds to the regularity of the energy for this problem. Moreover, a theorem of Liu, Simpson and Sulem (\textit{J.~Nonlinear Sci.}~2013) establishes the orbital stability of the gDNLS solitons, provided that there is a suitable $H^1$ well-posedness theory. 

To our knowledge, this is the first low regularity well-posedness result for a quasilinear dispersive model where the nonlinearity is both rough and lacks the decay necessary for global smoothing type estimates. These two features pose considerable difficulty when trying to apply standard tools for closing low-regularity estimates. While the tools developed in this article are used to study gDNLS, we believe that they should be applicable in the study of local well-posedness for other dispersive equations of a similar character. It should also be noted that the high regularity well-posedness presents a novel issue, as the roughness of the nonlinearity limits the potential regularity of solutions. Our high regularity well-posedness threshold $s<4\sigma$ is twice as high as one might na\"ively expect, given that the function $z\mapsto |z|^{2\sigma}$ is only $C^{1,2\sigma-1}$ H\"older continuous. Moreover, although we cannot prove $H^1$ well-posedness when $\sigma\leq \frac{\sqrt{3}}{2}$, we are able to establish $H^s$ well-posedness in the high regularity regime $s\in (2-\sigma,4\sigma)$ for the full range of $\sigma\in (\frac{1}{2},1)$. This considerably improves the known local results, which had only been established in either $H^2$ or in weighted Sobolev spaces.

\end{abstract} 
\footnotetext{Mitchell A.~Taylor is the corresponding author.}

 \tableofcontents

  \section{Introduction}
  In this article we consider the \textit{generalized derivative nonlinear Schr\"odinger equation}:
  \begin{equation}\label{gDNLS}\tag{gDNLS}
  \begin{cases}
    &i\partial_t u+\partial_x^2u=i|u|^{2\sigma}\partial_xu,
    \\  
    &u(0)=u_0,
    \end{cases}
\end{equation}
  where $u:\mathbb{R}\times \mathbb{R}\to \mathbb{C}$ and $\sigma>0$. We will be particularly interested in the case $\sigma<1$, as this is where  $H^s$ local well-posedness is most difficult. We begin with a brief history of this family of equations, and some of its closely related analogues.\\
 
The \eqref{gDNLS} equations originate from the study of the so-called  \textit{derivative nonlinear Schr\"odinger equation}:
 \begin{equation}\label{DNLS}\tag{DNLS}
 \begin{cases}
    &i\partial_t u+\partial_x^2u=i|u|^{2}\partial_xu,
\\  
&u(0)=u_0,
    \end{cases}
\end{equation}
which corresponds to \eqref{gDNLS} with $\sigma=1$. Physically, \eqref{DNLS} derives from the one-dimensional compressible magneto-hydrodynamic equation in the presence of the Hall effect, and the propagation of circular polarized nonlinear Alfv\'en waves in magnetized plasmas \cite{mio1976modified, mjolhus1976modulational, passot1993multidimensional}.  It also appears as a model for ultrashort optical pulses  \cite{agrawal2000nonlinear, moses2007self}, as well as in various other physical scenarios \cite{champeaux1999remarks, jenkins2019derivative, sanchez2010quasicollapse}. Mathematically, \eqref{DNLS} also has many interesting features. For example, like the 1D cubic NLS, it is completely integrable \cite{kaup1978exact}. However, it scales like the 1D quintic NLS, which makes it $L^2$ critical. Moreover, although at first glance \eqref{DNLS} looks to be semilinear, it is known that uniform continuity of the solution map fails in $H^s$ as long as $s<\frac{1}{2}$ (see \cite{MR1837253, MR1836810}). Therefore, this PDE has a clear quasilinear flavour. 
  \\
  
In recent years, the \eqref{gDNLS} family of equations has seen increasing interest, stemming from the 2013 article of Liu, Simpson and Sulem \cite{MR3079669}. One of the original motivations of \cite{MR3079669} was to shed light on the global well-posedness of \eqref{DNLS} in the energy space $H^1$, which was an important open problem. However, in an interesting turn of events, Bahouri and Perelman \cite{GWPDNLS} managed to prove global well-posedness for the \eqref{DNLS} equation \textit{before} the global well-posedness of \eqref{gDNLS} could be established for any $\sigma\neq 1$. In this article we make progress towards resolving one half of the program of Liu, Simpson and Sulem by proving that \eqref{gDNLS} is globally well-posed in $H^1$ for $\sigma\in (\frac{\sqrt{3}}{2},1)$. Note that, although completed shortly after each other, our result for $\sigma<1$ and the $\sigma=1$ result of \cite{GWPDNLS} are completely independent, and the methods used differ quite dramatically. Indeed, for $\sigma=1$, local well-posedness in $H^1$ has been known for a long time \cite{hayashi1993initial}, and can be established by employing a suitable gauge transformation, and standard Strichartz estimates. In fact, the smoothing properties of the equation are suitable to lower the well-posedness threshold to $H^{\frac{1}{2}}$ as in \cite{takaoka1999well}. Global well-posedness, however, is considerably harder, as the problem is $L^2$ critical. For this reason, Bahouri and Perelman (as well as Harrop-Griffiths, Killip, Ntekoume and Vi\c san \cite{harrop2022global,H16,killip2021well} in their subsequent work) crucially rely on the complete integrability of \eqref{DNLS}. In the case $\sigma<1$, the main difficulties are reversed. Establishing local well-posedness is difficult because of the lack of decay and roughness of the nonlinearity. On the other hand, one expects to be able to easily propagate any reasonable $H^1$ local well-posedness theory in time to obtain a global result. This is because when $\sigma<1$ the problem becomes $L^2$ subcritical, and one expects to be able to use the conserved energy and mass of the problem to control the $H^1$ norm of a solution. 
  \\
  
  Another motivation for \eqref{gDNLS} is the rich family of soliton solutions, which is actually where the majority of \cite{MR3079669}'s efforts were focused. Assuming a suitable $H^1$ well-posedness theory, the authors of \cite{MR3079669} were able to use the abstract theory of Grillakis, Shatah and Strauss \cite{grillakis1987stability, grillakis1990stability} to investigate the orbital stability of the solitons. However, an $H^1$ well-posedness theory for $\sigma<1$ had not been known until now.
  \\
 
 When $\sigma<1$, one can view \eqref{gDNLS} as a prototypical model for a quasilinear dispersive equation with a rough, low power nonlinearity (see \cite{linares2019class} for a KdV analogue). Such nonlinearities in the context of semilinear NLS type equations are becoming increasingly well-understood \cite{cazenave2017non, uchizono2012well}, and at modest regularity  local well-posedness can usually be proven by a combination of regularization and perturbative arguments. However, the combination of derivative and low power coefficient in the nonlinearity of \eqref{gDNLS} causes many interesting technical issues, several of which are yet to be fully understood. One issue for low regularity well-posedness is that the coefficient $|u|^{2\sigma}$ in the nonlinearity is less than quadratic in order. Because of this, the smoothing properties of the linear part of the Schr\"odinger equation are seemingly not strong enough to directly compensate for the apparent derivative loss which occurs because of the $u_x$ term in the nonlinearity. Another tool to avoid derivative loss - which has been successfully employed in the case $\sigma>1$ in  \cite{hao2008well, hayashi2016well} - is a gauge transformation. This technique allows one to re-normalize the equation to effectively remove the worst interactions in the derivative nonlinearity. However, again, it seems one can only directly apply this method when $\sigma\geq 1$ (i.e. $|u|^{2\sigma}$ is of quadratic order or higher), as in the case $\sigma<1$ negative powers of $|u|$ eventually appear in the analysis. This is related to the roughness of the nonlinearity, and will be elaborated on further when we outline the proof of our results. 
 \\
 \\
 To contrast, the Benjamin-Ono equation,
 \begin{equation}\label{BO}
 \begin{cases}
 &u_t+Hu_{xx}=uu_x,
 \\
 &u(0)=u_0,
 \end{cases}
 \end{equation}
 has a similar low power derivative nonlinearity $uu_x$, and as with \eqref{gDNLS}, the linear part of the equation does not have strong enough smoothing properties to directly compensate for the derivative loss in the nonlinearity. Nevertheless, $H^1$ well-posedness for this equation was established several years ago in \cite{tao2004global}. One should note, however, that the Benjamin-Ono nonlinearity has a much nicer algebraic structure than that of \eqref{gDNLS} (it is smooth and multilinear), which makes the equation more amenable to normal form type techniques (such as cubic corrections or a gauge transformation). Moreover, Christ \cite{christ2003illposedness} showed that Schr\"odinger's equation with Benjamin-Ono's nonlinearity is ill-posed in any reasonable sense, so the analogies between these equations are at best heuristic. For \eqref{gDNLS}, our solution to the above difficulties will be to introduce a family of \emph{partial} gauge transformation adapted to each dyadic frequency scale and the corresponding paradifferential flow - which  removes the portion of the nonlinearity which is large in a pointwise sense, on a scale which is balanced against the corresponding frequency localization scale of the nonlinearity. This will then be combined with smoothing and maximal function type arguments to attain the $H_x^1$ well-posedness threshold.
  \\
  
Another novel issue in the study of \eqref{gDNLS} when compared with other standard quasilinear Schr\"odinger models \cite{MR4331023,MR4830552} is that the nonlinearity has only a finite degree of H\"older regularity, and so one does not expect to be able to construct smooth solutions from regular data. In our case, the nonlinearity is only $C^{1,2\sigma-1}$ H\"older continuous. We expect therefore to only be able to differentiate the equation with respect to some parameter ``$2\sigma$ times" to obtain estimates. To maximize the potential regularity of solutions, we note that the scaling of the Schr\"odinger equation suggests that we can convert $L_x^2$ based estimates for one time derivative of a solution to estimates for two spatial derivatives. Therefore, it is advantageous to differentiate \eqref{gDNLS} in time rather than in space, and then convert time derivative estimates into estimates for spatial derivatives of a solution. After a single time differentiation, we are left with $2\sigma-1$ degrees of regularity on the nonlinearity. By working with fractional space derivatives, one expects to be able to prove an energy estimate for the $H_x^{1+2\sigma}$ norm of a solution.  However, working with fractional time derivatives (after suitably localizing in time), one expects to improve this further, and prove well-posedness in $H_x^s$ up to $s=2\cdot 1+2\cdot(2\sigma-1)=4\sigma$. A similar heuristic argument applies to any dispersion generalized equation with rough nonlinearity, where one can convert time derivative estimates into estimates on a certain number of spatial derivatives, perhaps modulo some perturbative terms coming from the nonlinearity. In general, we expect this heuristic to give rather sharp results, but this is not even known for semilinear NLS equations with rough nonlinearities \cite{cazenave2017non, uchizono2012well}, and is essentially unexplored in the quasilinear setting. 
\\

 Finally, let us  recall some basic symmetry properties of \eqref{gDNLS} as well as some conservation laws, which we will use to propagate our local well-posedness result to a global one. First, we have the scaling transformation
  $$u(t,x)\mapsto u_\lambda (t,x):=\lambda^\frac{1}{2\sigma}u(\lambda^2 t,\lambda x), \ \ \ \lambda>0,$$
  which makes the critical Sobolev index $s_c=\frac{1}{2}-\frac{1}{2\sigma}$. In particular, the problem is $L^2$ subcritical when $\sigma<1$. Moreover, \eqref{gDNLS} admits the following conserved quantities:
  
 \begin{equation}
     M(u)=\frac{1}{2}\int_\mathbb{R} |u|^2dx,
 \end{equation}
 \begin{equation}
     P(u)=\frac{1}{2}Re\int_\mathbb{R} i\overline{u}u_xdx,
 \end{equation}
  \begin{equation}
      E(u)=\frac{1}{2}\int_\mathbb{R} |u_x|^2dx+\frac{1}{2(\sigma+1)}Re\int_\mathbb{R}i|u|^{2\sigma}\overline{u}u_xdx,
  \end{equation}
which are the mass, momentum and energy, respectively. Unlike the standard NLS, \eqref{DNLS} doesn't enjoy the Galilean invariance nor the pseudo-conformal invariance symmetries, the latter being relevant for avoiding blowup.  We also note that a simple change of variables allows us to change the sign of the nonlinearity in \eqref{gDNLS} and arrive at
\begin{equation}
    i\partial_tu+\partial_x^2u+i|u|^{2\sigma}\partial_x u=0.
\end{equation}
This latter equation is more common in the study of the solitary waves of \eqref{gDNLS}.

   \subsection{Results}   

The main result of this article is global well-posedness of \eqref{gDNLS} in $H^s(\mathbb{R})$ when $\frac{\sqrt{3}}{2}<\sigma<1$ and $s\in [1,4\sigma)$. However, we divide this theorem into a ``low-regularity" part and a ``high-regularity" part, to maximize the range of $\sigma$. The high-regularity result is as follows:
\begin{theorem}(High-Regularity)\label{high reg theorem}
Let $\frac{1}{2}<\sigma<1$ and let $2-\sigma<s<4\sigma$. Then \eqref{gDNLS} is locally well-posed in $H^s(\mathbb{R})$.
\end{theorem}
As mentioned, for a restricted range of $\sigma$, we can lower the well-posedness threshold down to $H^1$, where the conserved energy also gives global well-posedness:
 \begin{theorem}\label{low reg theorem}
 Let $\frac{\sqrt{3}}{2}<\sigma<1$ and let $1\leq s<4\sigma$. Then \eqref{gDNLS} is globally well-posed in $H^s(\mathbb{R})$. 
\end{theorem}
\begin{remark}
As a special case, \Cref{high reg theorem} shows in particular that we have local well-posedness in $H^{s}$ for $\frac{3}{2}\leq s\leq 2$. Therefore, we recover the only previously known local well-posedness results for \eqref{gDNLS} when $\sigma<1$; namely, we recover the $H^2$ result of \cite{hayashi2016well} and improve the result of \cite{santos2015existence}, which used weighted Sobolev spaces. 
\end{remark}
\begin{remark}
In both \Cref{high reg theorem} and \Cref{low reg theorem},  well-posedness is to be interpreted in the usual quasilinear fashion, including existence, uniqueness and continuous dependence on the data. More specifically, given an appropriate Sobolev index $s$ and time $T>0$, we first build a function space $X^s_T$ that continuously embeds into $C([-T,T];H^s_x)$. We then show that for each $u_0\in H^s_x$ there exists a unique solution $u$ to \eqref{gDNLS} that lies in $X^s_T$ and satisfies $u(t=0)=u_0$. Finally, we show that the data to solution map is continuous, even  as a map from $H^s_x$ to the stronger topology $X^s_T$.
\end{remark}

\begin{remark}
Since \eqref{DNLS} is known to be globally well-posed in $H^\frac{1}{2}$, one may wonder why we only consider $H^s$ well-posedness when $s\geq 1$. This is, in fact, not necessary. For each $\sigma \in (\frac{\sqrt{3}}{2},1)$, we expect that technical modifications of our proof should establish $H^s$ well-posedness of \eqref{gDNLS} in a range $s\in [l(\sigma),4\sigma)$ with $l(\sigma)<1$ and $l(\sigma)\to \frac{1}{2}$ as $\sigma\to 1$. We avoid doing this for the sake of simplicity. It remains an open problem to prove well/ill-posedness in $H^\frac{1}{2}$ for any $\frac{1}{2}<\sigma<1$, and to find the smallest $\sigma\in (0,1)$ such that \eqref{gDNLS} is well-posed in $H^1.$
\end{remark}

  \subsection{History on well-posedness and solitons}  There is a vast literature devoted to the well-posedness of \eqref{DNLS}, as it took several decades for the regularity to approach current thresholds, and for global results to emerge. We begin our review with the work of Tsutsumi and Fukuda \cite{tsutsumi1980solutions, tsutsumi1981solutions} who studied the well-posedness in $H^s(\mathbb{R})$ for $s>\frac{3}{2}$ by classical energy methods and parabolic regularization. The well-posedness in $H^1(\mathbb{R})$ was reached by Hayashi \cite{hayashi1993initial} by applying a gauge transformation to overcome the derivative loss, and Strichartz estimates to close a-priori estimates. The $H^1(\mathbb{R})$-solution was shown to be global by Hayashi and Ozawa \cite{hayashi1992derivative}, as long as the initial data satisfies $\|u_0\|^2_{L^2}<2\pi.$ Later, Wu \cite{wu2015global} improved this global result by relaxing the smallness condition to $\|u_0\|^2_{L^2}<4\pi$, which is natural in view of the soliton structure.
  \\
  
  Below the energy space, there are also many results for \eqref{DNLS}. Takaoka \cite{takaoka1999well} proved local well-posedness in $H^s(\mathbb{R})$ when $s\geq \frac{1}{2}$ by the Fourier restriction method. This was complemented by a result of Biagioni and Linares \cite{MR1837253} which notes  that the solution map from $H^s(\mathbb{R})$ to $C([-T,T];H^s(\mathbb{R}))$ cannot be locally uniformly continuous when $s<\frac{1}{2}$. By using the I-method, Colliander, Keel, Staffilani, Takaoka and Tao \cite{colliander2001global, colliander2002refined} proved that the $H^s(\mathbb{R})$-solution is global if $s>\frac{1}{2}$ and   $\|u_0\|^2_{L^2}<2\pi.$ Guo and Wu \cite{guo2017global} were later able to strengthen this result by proving that $H^{\frac{1}{2}}(\mathbb{R})$-solutions are global if $\|u_0\|^2_{L^2}<4\pi.$ For an incomplete list of well-posedness results for \eqref{DNLS} on the torus, see \cite{hayashi2019studies, nahmod2012invariant} and references therein.
  \\
  
  There are also many works that use the complete integrability of the \eqref{DNLS} equation. The breakthrough result is \cite{GWPDNLS}, which establishes global well-posedness in $H^\frac{1}{2}(\mathbb{R})$. However, \cite{GWPDNLS} was preceded by many results - see, e.g., \cite{jenkins2020global, shimabukuro2017derivative, pelinovsky2018existence} - highlights of which include a global well-posedness result in the weighted Sobolev space $H^{2,2}(\mathbb{R})$, and progress towards the soliton resolution conjecture. Moreover, although $H^\frac{1}{2}$ regularity is necessary for uniform continuity of the solution map, \cite{harrop2022global,H16,killip2021well} are able to lower the global well-posedness threshold all the way to the critical Sobolev space $L^2$, definitively resolving the well-posedness theory for  \eqref{DNLS} on the line. On the other hand, blowup for \eqref{DNLS} on non-standard domains (for example, the half-line with the Dirichlet boundary condition) is known to be possible \cite{tan2004blow, wu2014global}.
  \\

For \eqref{gDNLS}, the literature on well-posedness is also quite large, though the results are far less definitive. As mentioned, \eqref{gDNLS} was popularized by \cite{MR3079669}, though well-posedness was not considered in that article. Possibly the first well-posedness result was by Hao, who in \cite{hao2008well} was able to prove local well-posedness in $H^\frac{1}{2}(\mathbb{R})$ intersected with an appropriate Strichartz space for $\sigma\geq \frac{5}{2}$. Ambrose and Simpson \cite{ambrose2015local} proved the existence and uniqueness of solutions $u\in C([0,T);H^2(\mathbb{T}))$ and the existence of solutions $u\in L^\infty([0,T),H^1(\mathbb{T}))$ for $\sigma\geq 1$. The uniqueness of $H^1(\mathbb{T})$-solutions was left unresolved, as the proof uses a compactness argument. For an up-to-date study of \eqref{gDNLS} on the torus we refer the reader to \cite{HOV2025}. On the real line, existence and uniqueness in $H^\frac{1}{2}(\mathbb{R})$ was proven by Santos in \cite{santos2015existence} for $\sigma>1$, by utilizing global smoothing and maximal function estimates. A result in weighted Sobolev spaces was also proved in \cite{santos2015existence} for the case $\frac{1}{2}<\sigma<1$, as adding weights helps compensate for the low power in the nonlinearity. In terms of $H^s(\mathbb{R})$ spaces, \cite{hayashi2016well} proves local well-posedness in $H^2$ when $\sigma\geq \frac{1}{2}$, local well-posedness and small data global existence in $H^1$ when $\sigma\geq 1$, existence of weak solutions when $\sigma<1$, and certain unconditional uniqueness results at high regularity. See \cite{mosincat2018unconditional} for more on unconditional uniqueness. The \eqref{gDNLS} equation with extremely rough nonlinearities $0<\sigma<\frac{1}{2}$ is studied in \cite{linares2019classII, MR3952703}, but not in standard Sobolev spaces $H^s$.
\\

 We now turn to the history on stability of solitons. This is also a vast subject, and \eqref{gDNLS} is not the only generalization of \eqref{DNLS} whose solitons have been considered. For the sake of unification, therefore, let us  consider the equation
\begin{equation}\label{DNLSb}
    i\partial_t u+\partial_x^2u+i|u|^{2\sigma}\partial_xu+b|u|^{4\sigma}u=0,\  x\in \mathbb{R},
\end{equation}
which is just a Schr\"odinger equation with a scale-invariant combination of derivative and power nonlinearities. Direct calculation verifies that the soliton solutions of \eqref{DNLSb} are given by
\begin{equation*}
    u_{\omega,c}(t,x)=e^{i\omega t}\phi_{\omega,c}(x-ct)
\end{equation*}
where
\begin{equation*}
 \phi_{\omega,c}(x)=\Phi_{\omega,c}(x)e^{i\theta_{\omega,c}(x)}, \ \ \ \ \ \ 
\theta_{\omega,c}(x)=\frac{c}{2}x-\frac{1}{2\sigma+2}\int_{-\infty}^x\Phi_{\omega,c}(y)^{2\sigma}dy,
\end{equation*}
and, using the notation $\gamma=1+\frac{(2\sigma+2)^2}{2\sigma+1}b$, the real valued function $\Phi_{\omega,c}$ is given by

\begin{equation*}
\Phi_{\omega,c}(x)^{2\sigma}=\left\{
\begin{aligned}
& \frac{(\sigma+1)(4\omega-c^2)}{\sqrt{c^2+\gamma(4\omega-c^2)}\cosh{(\sigma \sqrt{4\omega-c^2}x)-c}} \ \ \ \ \ \gamma>0, \ \ -2\sqrt{\omega}<c<2\sqrt{\omega},\\
	& \frac{2(\sigma+1)c}{(\sigma cx)^2+\gamma }\ \ \ \ \ \ \  \ \ \ \ \ \ \ \ \ \ \ \ \ \ \ \ \ 
	 \ \ \ \ \ \ \ \ \ \ \ \ \ \ \ \ \ \ \ \ \  \gamma>0, \ \ \ \ \   c=2\sqrt{\omega}, \\
	& \frac{(\sigma+1)(4\omega-c^2)}{\sqrt{c^2+\gamma(4\omega-c^2)}\cosh{(\sigma \sqrt{4\omega-c^2}x)-c}} \ \ \ \ \ \ \gamma\leq 0,\ \  -2\sqrt{\omega}<c<-2\sqrt{-\gamma/(1-\gamma)}\sqrt{\omega}.\\
\end{aligned}	\right.
\end{equation*}
These solitons are, of course, related to the Hamiltonian structure of \eqref{DNLSb}, as well as to the conservation of mass, energy and momentum, which we leave to the reader to compute.
\\
  
As expected, the story on soliton stability for \eqref{DNLSb} begins with the \eqref{DNLS} equation. Indeed, in \cite{guo1995orbital}, Guo and Wu proved that the soliton solutions of \eqref{DNLS} are orbitally stable when $\omega>\frac{c^2}{4}$ and $c<0$ by applying the abstract theory of Grillakis, Shatah, and Strauss \cite{grillakis1987stability, grillakis1990stability}. Colin and Ohta \cite{colin2006stability} removed the condition $c<0$ and proved that $u_{\omega,c}$ is orbitally stable when $\omega>\frac{c^2}{4}$ by applying the variational characterization of solitons as in Shatah \cite{shatah1983stable}. The endpoint case $c=2\sqrt{\omega}$ is only partially resolved; progress was made by Kwon and Wu in \cite{kwon2018orbital}, but with certain caveats, such as a non-standard definition of orbital stability. For the study of periodic travelling waves, we refer to \cite{chen2020modulational, hakkaev2020all,  hayashi2019long, hayashi2019studies} and references therein. 
\\

For \eqref{gDNLS}, the story on soliton stability is much richer. In \cite{MR3079669} it was shown that the solitary waves $u_{\omega,c}$ are orbitally stable if $-2\sqrt{\omega}<c<2z_0\sqrt{\omega}$, and orbitally unstable if $2z_0\sqrt{\omega}<c<2\sqrt{\omega}$ when $1<\sigma<2$. Here the constant $z_0=z_0(\sigma)\in (-1,1)$ is the solution to
$$F_\sigma(z):=(\sigma-1)^2\left(\int_0^\infty (\cosh y-z)^{-\frac{1}{\sigma}}dy\right)^2-\left(\int_0^\infty (\cosh y-z)^{-\frac{1}{\sigma}-1}(z\cosh y-1)dy\right)^2=0.$$
Moreover, \cite{MR3079669} proves that all solitary waves with $\omega>\frac{c^2}{4}$ are orbitally unstable when $\sigma\geq 2$ and orbitally stable when $0<\sigma<1$. As mentioned previously, these results are conditional on an appropriate well-posedness theory; there is also a minor numerical portion to the proof. In the borderline case when $c=2z_0\sqrt{\omega}$ and $1<\sigma<2$, Fukaya (\cite{fukaya2017instability}, see also \cite{guo2018instability}) proved orbital instability of the solitons. This completes the study of orbital stability of the solitons of \eqref{gDNLS}, except in the case of the algebraic soliton, which requires special attention \cite{guo2017orbital, li2018instability}.
  \\
  
  In the case $\sigma=1$, $b\neq 0$, there are also many works on soliton stability for \eqref{DNLSb}, e.g. \cite{colin2006stability, fukaya2021instability, hayashi2019studies, hayashi2020stability,  hayashi2020potential, ning2020instability, ning2017instability, ohta2014instability}. On the other hand, there are no results in the case $\sigma\neq 1$, $b\neq 0$, as it seems the explicit formulas for the solitons were not previously known.  We also mention that from the point of view of low regularity well-posedness, the additional term $b|u|^{4\sigma}u$ in \eqref{DNLSb} is both perturbative and maintains scaling, so in our usual range $\frac{\sqrt{3}}{2}< \sigma<1$ our proof can easily be modified to establish global well-posedness in $H^1$, regardless of the size or sign of $b$. 
  To contrast, recall that the known proof of global well-posedness in the case $\sigma=1$, $b=0$ is rather delicate;  global well-posedness could, in principle, fail to persist once the effect of the focusing NLS is added. For state of the art global results when $\sigma=1$, $b\neq 0$  we mention \cite{hayashi2019studies}, which establishes global well-posedness below the soliton thresholds. In particular, \eqref{DNLSb} in the case $\sigma=1$, $b\leq -\frac{3}{16}$ has been known to be globally well-posed for some time now, as at this point the energy becomes coercive, after a suitable gauge transformation.

\subsection{Outline of the proofs}
Here, we outline the key ideas in the proof of \Cref{high reg theorem} and \Cref{low reg theorem}.  We begin with a discussion of the low-regularity argument. Before describing the proof, however, it is instructive to discuss why the gauge transformation used in \cite{hayashi2016well} combined with standard Strichartz estimates will not work. The following discussion is mostly heuristic and for the purpose of motivation only.
\\
\\
Firstly, by a standard energy estimate, one obtains for (regular enough) solutions to \eqref{gDNLS},
\begin{equation}\label{Basic energy estimate}
\|u\|_{L^{\infty}_TH_x^1}\lesssim \|u_0\|_{H_x^1}\exp\left(\int_{0}^{T}\|u\|_{L_x^{\infty}}^{2\sigma-1}\|u_x\|_{L_x^{\infty}}\right).   
\end{equation}
Therefore, one expects to be able to prove suitable $H^1$ bounds for solutions to \eqref{gDNLS} as long as one can estimate the Strichartz norm, $\|u_x\|_{L_T^1L_x^{\infty}}$. However, applying Strichartz estimates directly to \eqref{gDNLS} leads to a loss of a derivative. Therefore, one might na\"ively try to do some sort of gauge transformation to remove the $|u|^{2\sigma}u_x$ term in the equation, which is responsible for this loss. Indeed, if one (formally) defines
\begin{equation}
\begin{split}
\Phi(t,x)=-\frac{1}{2}\int_{-\infty}^{x}|u|^{2\sigma}dy
\end{split}
 \end{equation}
 and then
 \begin{equation}
 w=e^{i\Phi}u,
 \end{equation}
this leads to an equation for $w$ of the form
\begin{equation}
iw_t+\partial_x^2w=(-\partial_t\Phi+i\partial_x^2\Phi-(\partial_x\Phi)^2)w. 
\end{equation}
At first glance, it looks like one can prove Strichartz estimates for $w_x$ without losing derivatives, to obtain the corresponding bound for $\|u_x\|_{L_T^1L_x^{\infty}}$. Unfortunately, if we expand $\partial_t\Phi$, we get
\begin{equation}
\begin{split}
\partial_t\Phi&=-\sigma\int_{-\infty}^{x}\RE(|u|^{2\sigma-2}\overline{u}u_t)dy
\\
&=-\sigma\int_{-\infty}^{x}\RE(|u|^{2\sigma-2}\overline{u}i\partial_x^2u)dy-\sigma\int_{-\infty}^{x}\RE(|u|^{4\sigma-2}\overline{u}u_x)dy.
\end{split}
\end{equation}
The first term above is problematic. To avoid losing derivatives, we are forced to integrate by parts off one derivative. However, since $|u|^{2\sigma-2}\overline{u}$ is not $C^1$ when $\sigma<1$, this will inevitably introduce negative powers of $u$, so this approach will not work.
\\

While the above calculations are not particularly useful for closing low-regularity estimates, they do clearly identify the main enemies in trying to close Strichartz estimates for the gauge transformed equation. That is, the portion of $u$ which is small or vanishes will prevent us from closing Strichartz estimates for $w$. Therefore, it is natural to try to somehow perform a gauge transformation which only removes some portion of the derivative nonlinearity $|u|^{2\sigma}u_x$, which corresponds to a part of $u$ for which $u$ is  bounded away from zero. Doing this is somewhat subtle. We can't simply fix a universal constant $\epsilon>0$, and remove the portion of the nonlinearity for which $|u|>\epsilon$. This is because when the $u_x$ factor in $|u|^{2\sigma}u_x$ is at very high frequency (compared to $\epsilon$), we will still lose derivatives in the Strichartz estimate. To work around this issue, we perform a paradifferential expansion of the equation. That is, for each $j>0$, we project onto frequencies of size $\sim 2^j$ and obtain
\begin{equation}
(i\partial_t+\partial_x^2)P_ju=iP_{<j-4}|u|^{2\sigma}P_ju_x+g_j
\end{equation}
where $g_j$ is a perturbative term. The idea now is to split the coefficient $P_{<j-4}|u|^{2\sigma}=P_{<j-4}|u_{s}|^{2\sigma}+P_{<j-4}|u_{l}|^{2\sigma}$, where $u_{l}$ corresponds to the portion of $u$ which is bounded away from zero (where the lower bound depends on the frequency parameter $j$), and $u_{s}$ is the remaining portion of $u$ which is bounded above by some small $j$ dependent parameter. We then try to do a gauge transformation by defining
\begin{equation}
\Phi_j=-\frac{1}{2}\int_{-\infty}^{x}P_{<j-4}|u_l|^{2\sigma}dy
\end{equation}
and 
\begin{equation}
w_j=e^{i\Phi_j}P_ju.
\end{equation}
This leads to an equation for $w_j$ of the form,
\begin{equation}
(i\partial_t+\partial_x^2)w_j=(-\partial_t\Phi_j+i\partial_x^2\Phi_j-(\partial_x\Phi_j)^2)w_j+e^{i\Phi_j}g_j+ie^{i\Phi_j}P_{<j-4}|u_{s}|^{2\sigma}P_ju_x.
\end{equation}
The point now is that the negative powers of $u$ that arise in the $\partial_t\Phi_j$ term are bounded above by some parameter depending on the frequency scale $2^j$. To avoid derivative loss, we would like this parameter to be as small as possible (i.e. $u_l$ should be bounded below by a ($j$ dependent) constant which is as large as possible). However, we still have to contend with the remainder of the original derivative nonlinearity, $ie^{i\Phi_j}P_{<j-4}|u_{s}|^{2\sigma}P_ju_x$, which is expected to cause derivative loss unless $u_s$ is sufficiently small (depending on $j$). Therefore, we have to compromise between potential losses incurred by the $\partial_t\Phi_j$ term, and the remaining derivative nonlinearity. Unfortunately, by optimizing the appropriate splitting of $u$, it turns out that we will still lose $1-\sigma$ derivatives in estimating the Strichartz norm $\|u_x\|_{L_T^1L_x^{\infty}}$, and therefore, one only expects to be able to control $\|u_x\|_{L_T^1L_x^{\infty}}$ by $\|u\|_{L_T^{\infty}H_x^{2-\sigma}}$. As mentioned, while this is certainly an improvement over previous results  \cite{hayashi2016well,santos2015existence}, this method is not quite robust enough to get well-posedness down to the energy space.  
\\
\\
To get $H^1$ well-posedness, we combine this modified gauge transformation (and Strichartz estimates) with smoothing and maximal function type estimates, as in \Cref{Hom max,Inhom Max}. However, we modify these Strichartz and maximal function norms (see the definition of $Y^s_T$ below) to reflect the loss of $1-\sigma$ derivatives compared to the $L_T^{\infty}H_x^1$ norm, as mentioned above. That is, we build this deficiency into the function spaces where we construct solutions. In particular, the Strichartz $(L_T^1L_x^{\infty})$ component of the norm involves no more than $\sigma$ derivatives. Therefore, the energy estimate \eqref{Basic energy estimate} described above is no longer appropriate to close a priori estimates in $H^1$. Hence, the energy estimate has to be modified accordingly so that the control parameter (i.e.~the Strichartz component) does not lead to a loss of derivatives (in excess of the $H^1$ norm) in the Strichartz/maximal function component of the estimate. It is actually this part of the argument that leads to the restriction on $\sigma$, which we will elaborate on later.
\\
 
Next, we outline the proof of the high regularity well-posedness. As mentioned previously, the $C^{1,2\sigma-1}$ H\"older regularity of the function $z\mapsto |z|^{2\sigma}$ effectively limits the number of times one can differentiate the equation to obtain $H^s$ estimates. A direct energy estimate, which involves differentiating the equation $s$ times in the spatial variable (i.e. applying $D_x^s$ to the equation) limits the range for which one can obtain estimates to $s\leq 2\sigma$. In \cite{hayashi2016well}, the authors managed to bypass this issue in the case $s=2$ by instead obtaining an $L_x^2$ energy estimate for the time derivative $\partial_tu$. The point is that doing this only requires one to differentiate the nonlinearity a single time. Once an appropriate $L_x^2$ estimate is obtained, $H_x^2$ energy estimates for the solution can then be obtained by observing that up to an error of size $O(\|u\|_{L_T^{\infty}H_x^1}^{2\sigma+1})$, the equation gives,  
\begin{equation}
\|(\partial_x^2u)(t)\|_{L_x^2}\sim\|(\partial_tu)(t)\|_{L_x^2}.
\end{equation}
In this article, we generalize this approach to derivatives of fractional order. It turns out that (after suitably localizing a solution in time), one can morally obtain an estimate (up to a suitable error term) essentially of the form 
\begin{equation}
\|D_t^{\frac{s}{2}}u\|_{L_T^{\infty}L_x^2}\sim\|D_x^{s}u\|_{L_T^{\infty}L_x^2}
\end{equation}
 where $1<s<4\sigma$. The main idea for proving this estimate is a modulation type analysis. Namely, when the space-time Fourier transform of a solution $u$ (after suitably localizing in time) is supported close to the characteristic hypersurface (or in the low modulation region), $\tau=-\xi^2$, one expects to be able to directly compare $D_t^{\frac{s}{2}}u$ and $D_x^{s}u$. On the other hand, when the space-time Fourier transform is supported far away from the hypersurface (or in the high modulation region), one expects to be able to control $D_t^{\frac{s}{2}}u$ and $D_x^{s}u$ in $L_x^2$ by a lower order error term stemming from the nonlinearity of the equation. This latter high modulation control can be loosely thought of as a space-time elliptic estimate. 
\\
\\
With a method for suitably comparing space and time derivatives of a solution in hand, it then essentially suffices to obtain an energy estimate for $D_t^{\frac{s}{2}}u$ when $u$ is localized near the characteristic hypersurface (which is precisely where one expects to be able to compare $D_t^{\frac{s}{2}}u$ to $D_x^su$). Therefore, in light of the  $C^{1,2\sigma-1}$ regularity of the nonlinearity, we should be able to obtain $H_x^{s}$ estimates for a solution as long as $\frac{s}{2}<2\sigma$. This explains the upper threshold of $4\sigma$ for our result. As hinted at earlier, the lower threshold of $2-\sigma$ is explained by the fact that such an energy estimate closes as long as one can control $\|u_x\|_{L_T^1L_x^{\infty}}$. Our low regularity estimates allow us to control this term by the $L_T^{\infty}H_x^{s}$ norm of $u$, as long as $s>2-\sigma$, where $\sigma$ lies in the full range $(\frac{1}{2},1)$. This should be contrasted with the $H^1$ case where we employ a more complicated functional setting and  only deal with a  restricted range of $\sigma$. For clarity, we have chosen to present our high regularity results in the simplest possible  functional setting, which is why the lower bound of $2-\sigma$ appears in \Cref{high reg theorem}, as it comes naturally from our previous estimates. Since $2-\sigma<\frac{3}{2}$ when $\sigma>\frac{1}{2},$ this is a reasonable lower threshold for the high regularity result (as it encompasses the range for which $\|u_x\|_{L_T^{1}L_x^{\infty}}$ can be controlled by Sobolev embedding). Nonetheless, we emphasize that the main novelty in \Cref{high reg theorem} is the upper threshold $s<4\sigma$.

\subsection{Acknowledgements}
We thank Daniel Tataru for several useful suggestions. This material is based upon work supported by the National Science Foundation under Grant No. DMS-1928930 while the authors participated in the program \textit{Mathematical problems in fluid dynamics} hosted by the Mathematical Sciences Research Institute in Berkeley, California, during
the Spring 2021 semester and appears as part of the authors' theses \cite{MR4820290,MR4675424}.

\section{Preliminaries} In this section we settle notation and recall some standard tools. 
\subsection{Littlewood-Paley decomposition}\label{LWP}
First, we recall the standard Littlewood-Paley decomposition. For this, let $\phi_0$ be a radial function in $C_0^{\infty}(\mathbb{R})$ that satisfies
$$0\leq \phi_0\leq 1, \ \ \phi_0(\xi)=1 \ \text{for}\  |\xi|\leq 1, \ \ \phi_0(\xi)=0 \ \text{for}\ |\xi|\geq \frac{7}{6}.$$
Let $\phi(\xi):=\phi_0(\xi)-\phi_0(2\xi)$. For $j\in \mathbb{Z}$, define
$$\widehat{P_{\leq j}f}(\xi)=\phi_0(2^{-j}\xi)\widehat{f}(\xi),$$
$$\widehat{P_jf}(\xi)=\phi(2^{-j}\xi)\widehat{f}(\xi).$$
We will denote $P_{>j}=I-P_{\leq j}$, where $I$ is the identity. Similarly, we define $P_{[a,b]}=\sum_{a\leq j\leq b}P_j$. We will also use the notation $\tilde{P}_j$, $\tilde{P}_{<j}$, $\tilde{P}_{>j}$ to denote a slightly enlarged or shrunken frequency localization. For example, we may denote $P_{[j-3,j+3]}$ by $\tilde{P_j}$. 
\\
\\
Next, we recall a useful bookkeeping device. Following \cite{ifrim2017well,tao2001global}, we denote by $L(\phi_1,\dots,\phi_n)$ a translation invariant expression of the form
\begin{equation*}
    L(\phi_1,\dots,\phi_n)(x)=\int K(y)\phi_1(x+y_1)\cdots \phi_n(x+y_n)dy,
\end{equation*}
where $K\in L^1$. Of interest is the following Leibniz type rule from \cite{ifrim2017well,tao2001global} which will make certain commutator expressions simpler to estimate: 
\begin{lemma}\label{Commutator leibniz}
(Leibniz rule for $P_j$). We have the commutator identity
\begin{equation}
    [P_j,f]g=L(\partial_xf,2^{-j}g).
\end{equation}
\end{lemma}

\subsection{Frequency envelopes}
One way we will employ the Littlewood-Paley projections is to define frequency envelopes, which are another nice bookkeeping device introduced by Tao \cite{tao2001global}.  To define these, suppose we are given a Sobolev type space $X$ such that
\begin{equation}
\|P_{\leq 0}u\|_X^2+\sum_{j=1}^\infty\|P_ju\|_{X}^2\sim \|u\|_{X}^2.
\end{equation}
A frequency envelope for $u$ in $X$ is a positive sequence $(a_j)_{j\in \mathbb{N}_0}$ such that 
\begin{equation}\label{property1}
  \|P_{\leq 0}u\|_X\lesssim a_0\|u\|_X, \ \ \   \|P_ju\|_X\lesssim a_j\|u\|_X, \ \ \ \sum_{j=0}^\infty a_j^2\lesssim 1.
\end{equation}
We say that a frequency envelope is admissible if $a_0\approx 1$ and it is slowly varying, meaning that
$$a_j\leq 2^{\delta|j-k|}a_k, \ \ \ j,k\geq 0, \ \ \ 0<\delta \ll 1.$$
An admissible frequency envelope always exists, say by 
\begin{equation}\label{env1}
    a_j=2^{-\delta j}+\|u\|_X^{-1}\max_{k\geq 0} 2^{-\delta |j-k|}\|P_ku\|_X.
\end{equation}
In \eqref{env1} - and in the definitions of the $X^s_T$ and $H^s_x$ frequency envelope formulas defined later -  there is a slight notational conflict, and $P_0u$ should really be interpreted as $P_{\leq 0}u.$
\begin{remark}Frequency envelopes will be particularly convenient for expediting the proof of continuous dependence later on. See \cite{primer} for an expository example of the use of frequency envelopes and \cite{Euler,MHD_sharp,ifrim2020compressible} for applications in more diverse settings.  \end{remark}
\subsection{Strichartz and maximal function estimates}
	Next we recall some standard linear estimates for the  Schr\"odinger equation on the line, which will play a key role in our analysis. We begin with the relevant maximal function and Strichartz estimates for the linear Schr\"odinger flow:
	\begin{proposition}\label{Hom max}(Homogeneous Strichartz and maximal function estimates)
		For $v\in \mathcal{S}(\mathbb{R})$, $\theta\in [0,1]$ and $T\in (0,1)$, we have for $j>0$
		\begin{equation}
			\begin{split}
				&\|e^{it\partial_x^2}v\|_{L_T^{\frac{4}{\theta}}L_x^{\frac{2}{1-\theta}}}\lesssim \|v\|_{L^2},
				\\
				&\|e^{it\partial_x^2}P_jv\|_{L_x^{\frac{2}{1-\theta}}L_T^{\frac{2}{\theta}}}\lesssim 2^{j(\frac{1}{2}-\theta)}\|v\|_{L^2}.
			\end{split}
		\end{equation}
	\end{proposition}
	\begin{proof}
		See \cite[Lemma 3.1]{kenig2006global}.
	\end{proof}
	We will also need the inhomogeneous versions of these estimates. Here $D_x^s:=|\partial_x|^s$, $\langle D_x\rangle^s:=(1+|\partial_x|^2)^{\frac{s}{2}}$, and $|\partial_x|:=H\partial_x$ where $H$ is the Hilbert transform, $\widehat{Hu}=-i\text{sgn}(\xi)\widehat{u}$. We further note that both Propositions \ref{Hom max} and \ref{Inhom Max} hold for $j=0$, with the interpretation $P_0=P_{\leq 0}$.
	\begin{proposition}\label{Inhom Max}(Inhomogeneous Strichartz and maximal function estimates) For $f\in \mathcal{S}(\mathbb{R}^2)$, $\theta\in [0,1]$ and $T\in (0,1)$, we have for $j>0$
		\begin{equation}
			\begin{split}
				&\biggnorm{\int_{0}^{t}e^{i(t-s)\partial_x^2}f(s)ds}_{L_T^{\frac{4}{\theta}}L_x^{\frac{2}{1-\theta}}}\lesssim \|f\|_{L_T^{(\frac{4}{\theta})'}L_x^{(\frac{2}{1-\theta})'}},
				\\
				&\biggnorm{\langle D_x\rangle^{\frac{\theta}{2}}\int_{0}^{t}e^{i(t-s)\partial_x^2}f(s)ds}_{L_T^{\infty}L_x^2}\lesssim \|f\|_{L_x^{p(\theta)}L_T^{q(\theta)}},
				\\
				&\biggnorm{ D_x^{\frac{1+\theta}{2}}\int_{0}^{t}e^{i(t-s)\partial_x^2}f(s)ds}_{L_x^{\infty}L_T^2}\lesssim \|f\|_{L_x^{p(\theta)}L_T^{q(\theta)}},
				\\
				&\biggnorm{ \langle D_x\rangle^{\frac{\theta}{2}}\int_{0}^{t}e^{i(t-s)\partial_x^2}P_jf(s)ds}_{L_x^{2}L_T^\infty}\lesssim 2^{\frac{j}{2}}\|f\|_{L_x^{p(\theta)}L_T^{q(\theta)}},
				\\
				&\biggnorm{\int_{0}^{t}e^{i(t-s)\partial_x^2}P_jf(s)ds}_{L_x^{\frac{2}{1-\theta}}L_T^{\frac{2}{\theta}}}\lesssim 2^{j(\frac{1}{2}-\theta)}\|f\|_{L_T^1L_x^2},
			\end{split}
		\end{equation}
		where
		\begin{equation}
			\begin{split}
				\frac{1}{p(\theta)}=\frac{3+\theta}{4},\hspace{5mm}\frac{1}{q(\theta)}=\frac{3-\theta}{4}.
			\end{split}
		\end{equation}
	\end{proposition}
	\begin{proof}
			See \cite[Lemma 3.4 and Remark 3.7]{kenig2006global}. 
	\end{proof}
	The following fractional Leibniz rules will also be useful for some of the following estimates:
	\begin{proposition}\label{Leib2}
		Let $\alpha\in (0,1)$, $\alpha_1,\alpha_2\in [0,\alpha]$,  $p,p_1,p_2,q,q_1,q_2\in (1,\infty)$ satisfy $\alpha_1+\alpha_2=\alpha$ and $\frac{1}{p}=\frac{1}{p_1}+\frac{1}{p_2}$, $\frac{1}{q}=\frac{1}{q_1}+\frac{1}{q_2}$. Then
		\begin{equation}
			\|D_x^{\alpha}(fg)-D_x^{\alpha}fg-fD_x^{\alpha}g\|_{L_x^pL_T^q}\lesssim \|D_x^{\alpha_1}f\|_{L_x^{p_1}L_T^{q_1}}\|D_x^{\alpha_2}g\|_{L_x^{p_2}L_T^{q_2}}.
		\end{equation}
		The endpoint cases $q_1=\infty, \alpha_1=0$ as well as $(p,q)=(1,2)$ are also allowed.
	\end{proposition}
	\begin{proof}
	See \cite[Lemma 2.6]{kenig2003local} or \cite[Lemma 3.8]{kenig2006global}.
	\end{proof}
	Another variant of the fractional Leibniz rule for $L^p_x$ spaces is as follows: 
	\begin{proposition}\label{Leib1}
		Let $\alpha\in (0,1)$, $\alpha_1,\alpha_2\in (0,\alpha)$ and $p\in [1,\infty)$, $1<p_1,p_2<\infty$ satisfy $\alpha_1+\alpha_2=\alpha$ and $\frac{1}{p}=\frac{1}{p_1}+\frac{1}{p_2}$. Then
		\begin{equation}
			\|D_x^{\alpha}(fg)-D_x^{\alpha}fg-fD_x^{\alpha}g\|_{L_x^p}\lesssim \|D_x^{\alpha_1}f\|_{L_x^{p_1}}\|D_x^{\alpha_2}g\|_{L_x^{p_2}}.
		\end{equation}
		The endpoint case $\alpha_2=0$, $1<p_2\leq \infty$ is also allowed if $p>1$.
	\end{proposition}
	\begin{proof}
			See \cite[Lemma 2.6]{kenig2003local}.
	\end{proof}
Next, we need a vector-valued Moser type estimate which will be convenient when derivatives fall on $|u|^{2\sigma}.$
\begin{proposition}\label{Moservec} Let $F\in C^1(\mathbb{C})$. Let $\alpha\in (0,1)$, $p,q,p_1,p_2,q_2\in (1,\infty)$ and $q_1\in (1,\infty]$ with
\begin{equation}
\frac{1}{p}=\frac{1}{p_1}+\frac{1}{p_2},\hspace{5mm}\frac{1}{q}=\frac{1}{q_1}+\frac{1}{q_2}.
\end{equation}
Then
\begin{equation}
\|D_x^{\alpha}F(u)\|_{L_x^pL_T^q}\lesssim \|F'(u)\|_{L_x^{p_1}L_T^{q_1}}\|D_x^{\alpha}u\|_{L_x^{p_2}L_T^{q_2}}.
\end{equation}
\end{proposition}
\begin{proof}
See	Theorem A.6 of \cite{kenig1993well}. 
\end{proof}
We also recall the scalar version of the above estimate,
\begin{proposition}\label{scalarMoser}
Let $F\in C^1(\mathbb{C})$, $u\in L^{\infty}(\mathbb{R})$, $\alpha\in (0,1)$, $1<p,q,r<\infty$, and $\frac{1}{r}=\frac{1}{p}+\frac{1}{q}$. Then
\begin{equation}
\|D_x^{\alpha}F(u)\|_{L^r}\lesssim \|F'(u)\|_{L^p}\|D_x^{\alpha}u\|_{L^q}.   
\end{equation}
\end{proposition}
\begin{proof}
See \cite{MR1124294}, Proposition 3.1.
\end{proof}
We will also make use of not only the standard Bernstein estimates (see, for example, \cite[(A.2)-(A.6), page~333]{tao2006nonlinear}) but the following vector-valued version: 
\begin{proposition}
Let $1\leq p,q\leq \infty$, $j>0$ and $s\in \mathbb{R}$. Then we have 
\begin{equation}
    \|D^s_xP_ju\|_{L^p_xL^q_T}\sim 2^{js}\|P_ju\|_{L^p_xL^q_T}.
\end{equation}
\end{proposition}
\begin{proof}
Let $\tilde{P}_j$ have corresponding multiplier $\tilde{\phi}_j$, where, as in the preliminaries on Littlewood-Paley theory, we have $\tilde{\phi}_j(\xi)=\tilde{\phi}(2^{-j}\xi)$. Notice that
$$D^s_x(\tilde{P}_jP_ju)=(D^s_x\mathcal{F}^{-1}\tilde{\phi}_j)\ast P_ju.$$
For each $x$, we have the inequality
$$\|D^s_xP_ju\|_{L^q_T}\leq |D^s_x\mathcal{F}^{-1}\tilde{\phi}_j|\ast \|P_ju\|_{L^q_T}.$$
Hence, applying $L^p_x$ and Young's inequality, we have
$$\|D^s_xP_ju\|_{L^p_xL^q_T}\leq \|D^s_x\mathcal{F}^{-1}\tilde{\phi}_j\|_{L^1_x} \|P_ju\|_{L^p_xL^q_T}\lesssim 2^{js}\|P_ju\|_{L^p_xL^q_T}.$$

On the other hand,
$$2^{js}\|P_ju\|_{L^p_xL^q_T}=2^{js}\|D_x^{-s}D_x^sP_ju\|_{L^p_xL^q_T}\lesssim \|D^s_xP_ju\|_{L^p_xL^q_T}.$$
\end{proof}

\subsection{A useful lemma}
	Finally, we need a H\"older estimate, which we will use to extract all of the $C^{1,2\sigma-1}$-regularity that our nonlinearity offers. We will use this lemma, e.g., when derivatives fall on $|u|^{2\sigma-2}\overline{u}$, or more generally on terms with regularity $C^{0,\alpha}$ for $0<\alpha<1$.
	\\
	
	To set notation, for $\alpha\in (0,1]$ and $1\leq p\leq \infty$ define the H\"older space $\dot{\Lambda}_{\alpha}^p(\mathbb{R})$ by
\begin{equation}
\begin{split}
&\|u\|_{\dot{\Lambda}_{\alpha}^p}:=\sup_{|h|>0}\frac{\|u(\cdot+h)-u(\cdot)\|_{L^p}}{|h|^{\alpha}}.
\end{split}
\end{equation}

This is just the usual homogeneous H\"older space $\dot{C}^{0,\alpha}$ when $p=\infty$.

	\begin{lemma}\label{preHolder}
Suppose that $F\in \dot{C}^{0,\alpha}(\mathbb{C})$. Then for every $0<\beta<\alpha<1$, $p\in [1,\infty]$ with $\alpha p\geq 1$, we have
\begin{equation}\label{lemma1}
\|F(u)\|_{\dot{\Lambda}_{\beta}^p}\lesssim \|F\|_{\dot{C}^{0,\alpha}}\|u\|_{W^{\frac{\beta}{\alpha},p\alpha}}^{\alpha}.
\end{equation}
\end{lemma}
\begin{proof}
We have
\begin{equation}
\begin{split}
\frac{|F(u(x+h))-F(u(x))|}{|h|^{\beta}}&=\frac{|F(u(x+h))-F(u(x))|}{|u(x+h)-u(x)|^{\alpha}}\left(\frac{|u(x+h)-u(x)|}{|h|^{\frac{\beta}{\alpha}}}\right)^{\alpha}    
\\
&\leq \|F\|_{\dot{C}^{0,\alpha}}\left(\frac{|u(x+h)-u(x)|}{|h|^{\frac{\beta}{\alpha}}}\right)^{\alpha}.
\end{split}    
\end{equation}
Hence, 
\begin{equation}
\begin{split}
\|F(u)\|_{\dot{\Lambda}_{\beta}^p}&\leq \|F\|_{\dot{C}^{0,\alpha}}\sup_{|h|>0}\|\left(\frac{|u(x+h)-u(x)|}{|h|^{\frac{\beta}{\alpha}}}\right)^{\alpha}\|_{L^p}
\\
&\leq \|F\|_{\dot{C}^{0,\alpha}}\|u\|_{\dot{\Lambda}_{\frac{\beta}{\alpha}}^{p\alpha}}^{\alpha}
\\
&\lesssim \|F\|_{\dot{C}^{0,\alpha}}\|u\|_{W^{\frac{\beta}{\alpha},p\alpha}}^{\alpha}
\end{split}    
\end{equation}
where the last line follows from a standard embedding (c.f.~\cite[Exercise A.21]{tao2006nonlinear}). 
\end{proof}
We also have the following very useful corollary of the above lemma which we will use extensively.
\begin{corollary}\label{Holder}
Suppose that $F\in  \dot{C}^{0,\alpha}(\mathbb{C})$ with $F(0)=0$. Then for every $0<\beta<\alpha<1$, $p\in [1,\infty]$ with $\alpha p\geq 1$ and $\epsilon\in (0,\alpha-\beta)$, we have
\begin{equation}
\|F(u)\|_{W^{\beta,p}}\lesssim_{\epsilon} \|F\|_{\dot{C}^{0,\alpha}}\|u\|_{W^{\frac{\beta}{\alpha}+\epsilon,p\alpha}}^{\alpha}.
\end{equation}
\end{corollary}
\begin{proof}
This follows from the embedding (c.f.~\cite[Exercise A.21]{tao2006nonlinear}),
\begin{equation}
\|F(u)\|_{W^{\beta,p}}\lesssim_{\epsilon} \|F(u)\|_{L^p}+\|F(u)\|_{\dot{\Lambda}^p_{\beta+\alpha\epsilon}}    
\end{equation}
and \Cref{preHolder} as well as the fact that
\begin{equation}
\|F(u)\|_{L^p}\lesssim \|F\|_{\dot{C}^{0,\alpha}}\|u\|_{L^{p\alpha}}^{\alpha}.    
\end{equation}
\end{proof}
\begin{remark}
It is easy to see that $F(z)=\overline{z}|z|^{2\sigma-2}$ meets the hypothesis of the above corollary (c.f.~\cite[Lemma 2.4]{ginibre1989scattering}). The price to pay when using \Cref{Holder} is that there is a sort of ``loss of regularity" when derivatives fall on $F(u)$ in the sense that a derivative of order $0<s<2\sigma-1$ will be amplified by a factor of $\frac{1}{2\sigma-1}$.
\end{remark}

	\section{Low regularity estimates}
	Now, we proceed with the proof of \Cref{low reg theorem}. By the scaling symmetry $u_{\lambda}(t,x):=\lambda^{\frac{1}{2\sigma}}u(\lambda^2t,\lambda x)$, we see that the $L^2_x$ norm is subcritical with respect to scaling. Hence, we will assume without loss of generality throughout that for some small $0<\epsilon\ll 1$  the initial data satisfies $\|u_0\|_{H_x^s}\leq \epsilon$. We then will obtain local well-posedness on the time interval $[-T,T]$ where $T\lesssim 1$ is fixed. 

	\subsection{Function spaces}
	We now define the spaces where we seek solutions. To begin, we define our baseline Strichartz type space $Y^0_T$ via 
	\begin{equation}
		\|u\|_{Y^0_T}:=\left(\sum_{j>0}\|P_jD_x^{\sigma-1}u\|_{L_T^4L_x^{\infty}}^2\right)^{\frac{1}{2}}+\left(\sum_{j>0}\|P_jD_x^{\sigma-\frac{1}{2}}u\|_{L_x^{\infty}L_T^2}^2\right)^{\frac{1}{2}}+\left(\sum_{j>0}\|P_jD_x^{\sigma-\frac{3}{2}}u\|_{L_x^{2}L_T^{\infty}}^2\right)^{\frac{1}{2}}+\|P_{\leq 0}u\|_{L_x^2L_T^{\infty}}.
	\end{equation}
	Then we define the space $X^0_T$ by:
	\begin{equation}
		\|u\|_{X^0_T}:=\left(\sum_{j>0}\|P_ju\|_{L_T^\infty L^2_x}^2\right)^{\frac{1}{2}}+\|P_{\leq 0}u\|_{L_T^{\infty}L_x^2}+\|u\|_{Y^0_T}.
	\end{equation}
	For higher Sobolev indices, $s\geq 0$, we define the spaces $X^s_T$ and $Y^s_T$ by 
	\begin{equation}
	    	\|u\|_{Y^s_T}:=\|\langle D_x\rangle^su\|_{Y^0_T}, \ \ \ \|u\|_{X^s_T}:=\|\langle D_x\rangle^su\|_{X^0_T}.
	\end{equation}
One should observe that we trivially have $\|u\|_{C([-T,T];H_x^s)}\leq \|u\|_{X^s_T}$. 
\begin{remark}
One might wonder why the above $Y^s_T$ space is not defined in a more standard way, where one replaces $\sigma$ with $1$. Indeed, one can see from the proof of the following estimates that by using this stronger norm, one will incur a loss of $1-\sigma$ derivatives in excess of the $L_T^{\infty}H_x^s$ norm. The function spaces defined above account for this loss.
\end{remark}
Finally, it will be convenient to define the weaker norm $S_T^{s}$ which just involves the purely Strichartz components of the $X_T^{s}$ norm. Namely,
\begin{equation}
\|u\|_{S_T^s}=\|P_{\leq 0}u\|_{L_T^{\infty}L_x^2}+\left(\sum_{j>0}\|P_j\langle D_x\rangle^s u\|_{L_T^{\infty}L_x^2}^2\right)^{\frac{1}{2}}+\left(\sum_{j>0}\|P_j\langle D_x\rangle^{s-1+\sigma}u\|_{L_T^4L_x^{\infty}}^2\right)^{\frac{1}{2}}.
\end{equation}
The behavior of the $S_T^1$ norm will be relevant for continuing a local solution to a global one when $\sigma\in (\frac{\sqrt{3}}{2},1)$ in both the low and high regularity regimes. 
\subsection{$X^s_T$ frequency envelopes}\label{Freq envel section}
It is easy to see that for $s\geq 0$, we have
\begin{equation}\label{Xenv}
\|P_{\leq 0}u\|^2_{X^s_T}+\sum_{j=1}^\infty\|P_ju\|_{X^s_T}^2\sim \|u\|_{X^s_T}^2.
\end{equation}
Hence, for $u\in X^s_T$, we use $b_j$ to denote the $X^s_T$ frequency envelope for $u$ defined by 
\begin{equation}\label{envX}
b_j=2^{-\delta j}+\|u\|_{X^s_T}^{-1}\max_{k\geq 0}2^{-\delta|j-k|}\|P_ku\|_{X^s_T}
\end{equation}
where $\delta$ is some small, but fixed, positive parameter. Similarly, for $v\in H_x^s$, we use $a_j$ to denote the $H_x^s$ frequency envelope for $v$ defined by
\begin{equation}\label{Henv}
a_j=2^{-\delta j}+\|v\|_{H_x^s}^{-1}\max_{k\geq 0}2^{-\delta|j-k|}\|P_kv\|_{H_x^s}.
\end{equation}
Unless otherwise stated, $X^s_T$ and $H^s_x$ frequency envelopes will always be defined by the above formulae.
\begin{remark}
In an identical fashion, one can also define $S_T^s$ frequency envelopes.
\end{remark}
Next, we state a technical lemma which will be useful for tracking the contributions of the rough part of the nonlinearity in \eqref{gDNLS} when derivatives fall on it.
\begin{lemma}(Moser type estimate)
Let $s\in [1,\frac{3}{2}]$, $\sigma\in (\frac{1}{2},1)$, $0<T\lesssim 1$ and let $b_j$ be a $X^s_T$ frequency envelope for $u$. Write $\alpha=s-1+\sigma<2\sigma$. For $j>0$, we have the following Moser type estimate, 
\begin{equation}\label{roughmoser}
\|D_x^{\alpha}P_j|u|^{2\sigma}\|_{L_T^2L_x^{\infty}}\lesssim b_j\|u\|_{S^1_T}^{2\sigma-1}\|u\|_{X^s_T}.
\end{equation}
\begin{proof}
There are two cases to consider. First assume $\alpha>1$. We have 
\begin{equation}
\begin{split}
\|D_x^{\alpha}P_j(|u|^{2\sigma})\|_{L_T^{2}L_x^{\infty}}&\lesssim \|P_jD_x^{\alpha-1}(|u|^{2\sigma-2}\overline{u}u_x)\|_{L_T^{2}L_x^{\infty}}
\\
&\lesssim \|P_jD_x^{\alpha-1}(P_{<j-4}(|u|^{2\sigma-2}\overline{u})u_x)\|_{L_T^{2}L_x^{\infty}}+\|P_jD_x^{\alpha-1}(P_{\geq j-4}(|u|^{2\sigma-2}\overline{u})u_x)\|_{L_T^{2}L_x^{\infty}}.
\end{split}
\end{equation}
For the first term, we have by Bernstein, 
\begin{equation}
\begin{split}
\|P_jD_x^{\alpha-1}(P_{<j-4}(|u|^{2\sigma-2}\overline{u})u_x)\|_{L_T^{2}L_x^{\infty}}&=\|P_jD_x^{\alpha-1}(P_{<j-4}(|u|^{2\sigma-2}\overline{u})\tilde{P_j}u_x)\|_{L_T^{2}L_x^{\infty}}
\\
&\lesssim 2^{j(\alpha-1)}\|u\|_{L_T^{\infty}L_x^{\infty}}^{2\sigma-1}\|\tilde{P}_ju_x\|_{L_T^{2}L_x^{\infty}}
\\
&\lesssim \|u\|_{S^1_T}^{2\sigma-1}\|D_x^{\alpha}\tilde{P}_ju\|_{L_T^{2}L_x^{\infty}}
\\
&\lesssim b_j\|u\|_{S^1_T}^{2\sigma-1}\|u\|_{X^s_T}.
\end{split}
\end{equation}

For the second term, we have for $\delta>0$ small (under the additional assumption that $2^{-\delta j}\lesssim b_j$) \begin{equation}
\begin{split}
\|P_jD_x^{\alpha-1}(P_{\geq j-4}(|u|^{2\sigma-2}\overline{u})u_x)\|_{L_T^{2}L_x^{\infty}}&\lesssim 2^{j(\alpha-1)}\|P_{\geq j-4}(|u|^{2\sigma-2}\overline{u})u_x\|_{L_T^{2}L_x^{\infty}}
\\
& \lesssim 2^{j(\alpha-1)}\|P_{\geq j-4}(|u|^{2\sigma-2}\overline{u})\|_{L^4_T L^\infty_x}\|u_x\|_{L^4_T L^{\infty}_x}
\\
&\lesssim b_j\|D_x^{\alpha-1+\delta}(|u|^{2\sigma-2}\overline{u})\|_{L^4_T L^\infty_x}\|u_x\|_{L^4_T L^\infty_x}
\\
&\lesssim b_j\|D_x^{\alpha-1+\delta}(|u|^{2\sigma-2}\overline{u})\|_{L^4_T L^\infty_x}\|u\|_{X^s_T}.
\end{split}
\end{equation}
It now suffices to show that 
\begin{equation*}
    \|D_x^{\alpha-1+\delta}(|u|^{2\sigma-2}\overline{u})\|_{L^4_T L^\infty_x}\lesssim\|u\|_{S^1_T}^{2\sigma-1}.
\end{equation*}

For this we fix $\epsilon>0$ small and invoke  \Cref{Holder} and the fact that $2\sigma-1<1$,
\begin{equation}\begin{split}
\|D_x^{\alpha-1+\delta}(|u|^{2\sigma-2}\overline{u})\|_{L^4_T L^\infty_x}&\lesssim_T \|\langle D_x\rangle^{\frac{\alpha-1+\delta+\epsilon}{2\sigma-1}}u\|^{2\sigma-1}_{L^4_TL^\infty_x} 
\\
&\lesssim \|u\|_{S^1_T}^{2\sigma-1}
\end{split}
\end{equation}
where in the last line we take $\epsilon,\delta$ small enough and used that $\frac{\alpha-1}{2\sigma-1}<\sigma$ when $s\in [1,\frac{3}{2}]$ and $\sigma\in (\frac{1}{2},1)$.
\\
\\
This handles the case $\alpha>1$. Next, we assume $0<\alpha\leq 1$. For this, we write
\begin{equation}
P_j|u|^{2\sigma}=P_j|P_{<j}u|^{2\sigma}+P_j(|u|^{2\sigma}-|P_{<j}u|^{2\sigma}).
\end{equation}
We have for the first term, 
\begin{equation}
\begin{split}
\|D_x^{\alpha}P_j|P_{<j}u|^{2\sigma}\|_{L_x^{\infty}}&\lesssim 2^{j(\alpha-1)}\|P_j(|P_{<j}u|^{2\sigma-2}\overline{P_{<j}u}P_{<j}u_x)\|_{L_x^{\infty}}
\\
&\lesssim 2^{j(\alpha-1)}\|P_j(P_{<j-4}(|P_{<j}u|^{2\sigma-2}\overline{P_{<j}u})\tilde{P}_ju_x)\|_{L_x^{\infty}}
\\
&+2^{j(\alpha-1)}\|P_j(P_{\geq j-4}(|P_{<j}u|^{2\sigma-2}\overline{P_{<j}u})P_{<j}u_x)\|_{L_x^{\infty}}
\\
&\lesssim \|u\|_{L_x^{\infty}}^{2\sigma-1}\|\tilde{P}_jD_x^{\alpha}u\|_{L_x^{\infty}}+2^{-j\delta}\|D_x^{2\delta}(|P_{<j}u|^{2\sigma-2}\overline{P_{<j}u})\|_{L_x^{\infty}}\|D_x^{\alpha-1-\delta}u_x\|_{L_x^{\infty}}.
\end{split}
\end{equation}
Hence, by taking $\delta$ small enough, using \Cref{Holder}, and the fact that $2^{-j\delta}\lesssim b_j$, we obtain 
\begin{equation}
\|D_x^{\alpha}P_j|P_{<j}u|^{2\sigma}\|_{L_T^2L_x^{\infty}}\lesssim_T b_j\|u\|_{S^1_T}^{2\sigma-1}\|u\|_{X^s_T}.
\end{equation}
Next, we estimate 
\begin{equation}
\begin{split}
\|P_jD_x^{\alpha}(|u|^{2\sigma}-|P_{<j}u|^{2\sigma})\|_{L_T^2L_x^{\infty}}&\lesssim 2^{j\alpha}\|u\|_{L_T^{\infty}L_x^{\infty}}^{2\sigma-1}\sum_{k\geq j}\|P_ku\|_{L_T^{2}L_x^{\infty}}
\\
&\lesssim \|u\|_{S^1_T}^{2\sigma-1}\|u\|_{X^s_T}\sum_{k\geq j}2^{-\alpha|k-j|}b_k
\\
&\lesssim b_j\|u\|_{S^1_T}^{2\sigma-1}\|u\|_{X^s_T}
\end{split}
\end{equation}
where in the last line, we used the slowly varying property of $b_j$. This completes the proof.
\end{proof}
\begin{remark}\label{alternativemoser}
By repeating the proof almost verbatim, and taking $b_j$ instead to be a $S_T^{s}$ frequency envelope for $u$, we can modify the conclusion of the lemma to
\begin{equation}
\|D_x^{\alpha}P_j|u|^{2\sigma}\|_{L_T^2L_x^{\infty}}\lesssim b_j\|u\|_{S_T^1}^{2\sigma-1}\|u\|_{S_T^s}.
\end{equation}
\end{remark}
\end{lemma}
\begin{remark}
The $\|u\|_{S_T^1}^{2\sigma-1}$ coefficient in the estimate (\ref{roughmoser}) could be optimized in terms of the parameters $s$ and $\sigma$. We do not pursue this, for the sake of simplicity and also because it does not improve any of the later estimates in an important way.
\end{remark}
	\subsection{Uniform bounds}
	In this subsection, we prove a priori estimates for solutions to \eqref{gDNLS}. First, we prove uniform $X^s_T$ bounds:
	\begin{proposition}\label{Ubounds}
	\label{unif1}
		Let $0<\epsilon\ll 1$, $s\in [1,\frac{3}{2}]$, $\sigma\in (\frac{\sqrt{3}}{2},1)$ and let $u_0\in H_x^s$ with $\|u_0\|_{H_x^s}\leq \epsilon$. Let $T\lesssim 1$. Suppose $u\in X^s_T$ solves the equation, 
		\begin{equation}\label{Lin2}
		 \begin{cases}
		  &(i\partial_t+\partial_x^2)u=i|u|^{2\sigma}\partial_xu,
		  \\
		  &u(0)=u_0.
		 \end{cases}
		\end{equation}
Furthermore, let $a_j$ and $b_j$ be a $H_x^s$ and $X^s_T$ frequency envelope for $u_0$ and $u$ (on the time interval $[0,T]$), respectively, as defined in \Cref{Freq envel section}. Then we have the following $X^s_T$ estimates for $j>0$, 
\\

a) (Frequency localized $X^s_T$ bound)  
\begin{equation}
\begin{split}
\|P_ju\|_{X^s_T}&\lesssim_{\|u\|_{S^1_T}} a_j\|u_0\|_{H^s_x}
				+T^{\frac{1}{2}}b_j(1+\|u\|_{S^1_T}^{4\sigma})\|u\|_{X^s_T}+T^{\frac{1-\sigma}{2}}b_j\|u\|_{X^1_T}^{\sigma}\|u\|_{X^s_T}.
\end{split}
\end{equation}
b) (Uniform $X^s_T$ bound)
		\begin{equation}
			\|u\|_{X^s_T}\lesssim_{\|u\|_{X^1_T}} \|u_0\|_{H^s_x}\leq \epsilon.
		\end{equation}
\end{proposition}	
		We will also need the following result:
\begin{proposition}\label{linbounds}
    Let $0<\epsilon\ll 1$ and $\sigma$, $T$ and $s$ be as in \Cref{Ubounds}. Suppose $v\in X^0_T$ is a solution to the equation, 
			\begin{equation}\label{lineq} 
		\begin{cases}
			&(i\partial_t+\partial_x^2)v=i|w|^{2\sigma}\partial_xv+g\partial_xav+\overline{g}\partial_xa\overline{v},
			\\
			&v(0)=v_0,
		\end{cases}    
	\end{equation}
		 for some $w\in X^1_T$ solving \eqref{gDNLS} (with possibly different initial data), $g\in Z:=Z_T:=L_x^{\frac{2}{2\sigma-1}}L_T^{\infty}\cap L_T^{\infty}W_x^{\frac{3}{4\sigma}-\frac{1}{2}+\epsilon,\infty}\cap L_T^4W_x^{\frac{3}{2}-\sigma+\epsilon,\infty}$ and $a\in X^1_T$, all with sufficiently small norm $\ll 1$. Then $v$ satisfies the bound
		\begin{equation}
			\|v\|_{X^0_T}\lesssim \|v_0\|_{L^2}.
		\end{equation}
\begin{remark} In practice $g$ will correspond to terms which are of similar regularity to the term $|u|^{2\sigma-1}$. For such terms to lie in $Z$ (specifically the latter two components of this norm), we will need $\sigma>\frac{\sqrt{3}}{2}$. This will be elaborated on later in the proof. 
\end{remark}

	\end{proposition}
	\begin{remark}
	\Cref{linbounds} will be useful for establishing difference estimates for solutions in the weaker topology, $X_T^0$. This will allow us to show uniqueness for $X_T^1$ solutions, and to prove a weak Lipschitz type bound for the solution map. 
	\end{remark}
	
We begin with the proof of \Cref{Ubounds}.
	We divide the relevant estimates into two parts. First, we control the $Y^s_T$  component of the norm. Then we do an energy type estimate to control the $L_T^{\infty}H_x^s$ component. For this purpose, we have the following lemmas: 
	\begin{lemma} ($Y^s_T$ estimate)\label{Yestimate} Let $s\in [1,\frac{3}{2}]$, $\sigma\in (\frac{1}{2},1)$ and let $u$, $T$, $a_j$ and $b_j$ be as in \Cref{Ubounds}. Then for $j>0$ we have 
		\begin{equation}
			\begin{split}
				\|P_ju\|_{Y^s_T}&\lesssim_{\|u\|_{S^1_T}} a_j\|u_0\|_{H^s_x}
				+T^{\frac{1}{2}}b_j(1+\|u\|_{S^1_T}^{4\sigma})\|u\|_{X^s_T}.
			\end{split}
		\end{equation}
\end{lemma}	
	\begin{lemma}($L_T^{\infty}H_x^s$ estimate)\label{NLEE}
Let $s,\sigma,T, a_j, b_j$ and $u$ be as in $\Cref{Ubounds}$. Then for $j>0$ we have 
\begin{equation}
\begin{split}
\|P_ju\|_{L_T^{\infty}H^s_x}\lesssim a_j\|u_0\|_{H^s_x}+T^{\frac{1-\sigma}{2}}b_j\|u\|_{X^1_T}^{\sigma}\|u\|_{X^s_T}.
\end{split}
\end{equation}
		\end{lemma}

	\begin{proof}
		We begin with the proof of \Cref{Yestimate}. For this purpose, let us apply $P_j$ to (\ref{Lin2}) and write
		\begin{equation}\label{paradif1}
			\begin{split}
				(i\partial_t+\partial_x^2)u_j=iP_{<j-4}|u|^{2\sigma}\partial_xu_j+g_j
			\end{split}
		\end{equation}
		where  
		\begin{equation}
			g_j=iP_j(P_{\geq j-4}|u|^{2\sigma}\partial_xu)+i[P_j,P_{<j-4}|u|^{2\sigma}]\partial_xu.
		\end{equation}
		The term 
		\begin{equation}\label{paradif3}
			iP_{<j-4}|u|^{2\sigma}\partial_xu_j
		\end{equation}
		which corresponds to the worst interactions between $\partial_x u$ and $|u|^{2\sigma}$ is non-perturbative, and can lead to loss of derivatives in the $Y^s_T$ estimates for $u_j$. It is desirable to remove as much of this bad interaction as possible. As mentioned earlier, one might try to remove it entirely with a gauge transformation, but this will not work, because the function $z\mapsto |z|^{2\sigma}$ is not smooth enough. Fortunately, in some sense, formally, the worst terms introduced by a gauge transformation are only poorly behaved when $u$ is small (i.e.~sufficiently close to $0$). On the other hand, if $u$ is sufficiently small (on a scale depending on $j$), then we expect to be able to treat the associated part of the term (\ref{paradif3}) perturbatively. One then expects to be able to remove the other part (in which $u$ is bounded away from zero) with a gauge transformation, and gain some mileage. 
		\\
		\\
		With this strategy in mind, let $\varphi$ be a smooth compactly supported function on $\mathbb{R}$ with $\varphi=1$ on the unit interval and zero outside $(-2,2)$. Likewise, define $\chi=1-\varphi$. We want to tailor these functions to a particular frequency, which we do by defining the rescaled functions $\varphi_j(x)=\varphi(2^jx)$ and $\chi_j(x)=\chi(2^jx)$. Next, we further rewrite (\ref{paradif1}) as the following equation,
		\begin{equation}\label{paradiff2}
			\begin{split}
				(i\partial_t+\partial_x^2)u_j=iP_{<j-4}[\chi_j(|u|^2)|u|^{2\sigma}]\partial_xu_j+iP_{<j-4}[\varphi_j(|u|^2)|u|^{2\sigma}]\partial_xu_j+g_j.
			\end{split}
		\end{equation}
		\begin{remark}
			One might wonder whether one can modify the $2^j$ scale in the definition of $\varphi_j$ to $2^{j\alpha}$ for some $\alpha>0$. It turns out that $\alpha=1$ is the optimal choice, as one can ascertain from repeating the estimates below with this new parameter $\alpha$. This optimization is obtained by balancing the contributions from the terms $I_j^1$ and $I_j^3$ in the below estimates. 
		\end{remark}
		Now, we do a partial gauge transformation to remove  $iP_{<j-4}[\chi_j(|u|^2)|u|^{2\sigma}]\partial_xu_j$, which corresponds to the part of (\ref{paradif3}) for which the coefficient $|u|^{2\sigma}$ is bounded below by $2^{-j\sigma}$. Indeed, define
		\begin{equation}\label{gauge}
			\Phi_j(t,x):=-\frac{1}{2}P_{<j-4}\partial_x^{-1}[\chi_j(|u|^2)|u|^{2\sigma}]
		\end{equation}
		where
		\begin{equation}
			(\partial_x^{-1}f)(x):=\int_{-\infty}^{x}f(y)dy
		\end{equation}
		and then define 
		\begin{equation}
			w_j:=u_je^{i\Phi_j}.
		\end{equation}
Before proceeding, we need the following technical estimate which relates $u_j$ to $w_j$.  
\begin{lemma}\label{expestimate}
Let $S$ refer to any of the four spaces, $L_T^{\infty}L_x^2$, $L_x^{\infty}L_T^2$, $L_x^2L_T^{\infty}$, or $L_T^4L_x^{\infty}$. Let $\beta\in (-1,1)$ and $0<\epsilon\ll 1$. Then for $j>0$, we have
\begin{equation}
\|\langle D_x\rangle^{\beta}u_j\|_{S}\lesssim_{\epsilon} (1+\|u\|_{S^1_T})^{2\sigma}(\|\langle D_x\rangle^{\beta}\tilde{P}_jw_j\|_{S}+\|\langle D_x\rangle^{\beta-\epsilon}w_j\|_{S}).
\end{equation}
\begin{remark}
As a brief remark, the range on $\beta$ accounts for (more than) the greatest range of derivatives allowed in any component of the $X^{1-\sigma}_T$ norm, which will correspond to the situation in which we apply the estimate. Strictly speaking, this is overkill, but it lets us avoid dealing with several individual cases. Also, the $\beta-\epsilon$ factor in the second term in the above estimate is to compensate for terms in which $w_j$ is not frequency localized. In particular, later when applying \Cref{Inhom Max}, the $\epsilon$ will allow us to sum up the individual frequency dyadic contributions of $w_j$.
\end{remark}
\begin{proof}
We have using the fact that $u_j$ is frequency localized to frequency $\sim 2^j$,  
\begin{equation}
\begin{split}
\|\langle D_x\rangle^{\beta} u_j\|_S&=\|\langle D_x\rangle^{\beta}\tilde{P}_j(e^{-i\Phi_j}w_j)\|_S
\\
&\lesssim \|D_x^{\beta}\tilde{P}_j(P_{<j-2}e^{-i\Phi_j}\tilde{P}_jw_j)\|_S+\|D_x^{\beta}\tilde{P}_j(P_{\geq j-2}e^{-i\Phi_j}w_j)\|_S.
\end{split}
\end{equation}
For the first term, we have by the (vector-valued) Bernstein's inequality
\begin{equation}
\|D_x^{\beta}\tilde{P}_j(P_{<j-2}e^{-i\Phi_j}\tilde{P}_jw_j)\|_S\lesssim \|D_x^{\beta}\tilde{P}_jw_j\|_{S}.
\end{equation}
For the second term, we have from Bernstein's inequality (and since $j>0$), 
\begin{equation}
\begin{split}
\|D_x^{\beta}\tilde{P}_j(P_{\geq j-2}e^{-i\Phi_j}w_j)\|_S&\lesssim 2^{j\beta}\|\tilde{P}_j(P_{\geq j-2}e^{-i\Phi_j}w_j)\|_S
\\
&\lesssim 2^{j\beta}\|P_{\geq j-2}e^{-i\Phi_j}\|_{L_T^{\infty}L_x^{\infty}}\|P_{<j+2}w_j\|_S+2^{j\beta}\sum_{k\geq j}\|\tilde{P}_ke^{-i\Phi_j}\|_{L_T^{\infty}L_x^{\infty}}\|\tilde{P}_kw_j\|_{S}
\\
&\lesssim_{\epsilon} \|P_{\geq j-2}D_x^{|\beta|+2\epsilon}e^{-i\Phi_j}\|_{L_T^{\infty}L_x^{\infty}}\|\langle D_x\rangle^{\beta-\epsilon}w_j\|_S
\end{split}
\end{equation}
where $\epsilon>0$ is small enough so that for instance, $|\beta|+2\epsilon<1$. Then we have by Bernstein,
\begin{equation}
\begin{split}
\|P_{\geq j-2}D_x^{|\beta|+2\epsilon}e^{-i\Phi_j}\|_{L_T^{\infty}L_x^{\infty}}\lesssim \|\partial_xP_{\geq j-2}e^{-i\Phi_j}\|_{L_T^{\infty}L_x^{\infty}}&\lesssim \|u\|_{S^1_T}^{2\sigma}.
\end{split}
\end{equation}
Combining the above estimates completes the proof.
\end{proof}
\end{lemma}

		Given \Cref{expestimate}, we are in a position to convert estimates for $w_j$ into estimates for $u_j$.  A direct computation shows that $w_j$ satisfies the following equation:
		\begin{equation}
			\begin{cases}
				(i\partial_t+\partial_x^2)w_j &=ie^{i\Phi_j}P_{<j-4}[\varphi_j(|u|^2)|u|^{2\sigma}]\partial_xu_j+(-\partial_t\Phi_j+i\partial_x^2\Phi_j-(\partial_x\Phi_j)^2)w_j+e^{i\Phi_j}g_j,
				\\
				\ \ \ \ \ \ \ \   w_j(0)&=  e^{i\Phi_j}u_j(0).
			\end{cases}
		\end{equation}
		The goal is to prove a priori estimates for $w_j$ - and hence $u_j$ - in $Y^s_T$. We observe a couple of useful facts. First, by Bernstein, we have $\|u_j\|_{Y^s_T}\lesssim 2^{j(\sigma+s-1)}\|u_j\|_{Y^{1-\sigma}_T}$. Secondly, we obviously have $\|gw_j\|_{L_T^1L_x^2}=\|gu_j\|_{L_T^1L_x^2}$ for measurable functions, $g$. Using these observations, \Cref{expestimate}, the maximal function estimates and the usual Strichartz estimates from Propositions~\ref{Hom max} and \ref{Inhom Max} we have that 
		\begin{equation}
			\begin{split}
				\frac{\|u_j\|_{Y^s_T}}{(1+\|u\|_{S^1_T})^{2\sigma}}&\lesssim \|u_j(0)\|_{H_x^s}+2^{j(\sigma+s-1)}\|P_{<j-4}[\varphi_j(|u|^2)|u|^{2\sigma}]\partial_xu_j\|_{L^1_TL^2_x}+2^{j(\sigma+s-1)}\|g_j\|_{L^1_TL^2_x}
				\\
				&+2^{j(\sigma+s-1)}\|\partial_t\Phi_ju_j\|_{L^1_TL^2_x}+2^{j(\sigma+s-1)}\|\partial_x^2\Phi_ju_j\|_{L^1_TL^2_x}+2^{j(\sigma+s-1)}\|(\partial_x\Phi_j)^2u_j\|_{L^1_TL^2_x}
				\\
				&:=\|u_j(0)\|_{H_x^s}+I_1^j+I_2^j+I_3^j+I_4^j+I_5^j.
			\end{split}
		\end{equation}
We now estimate each of the above terms.
		\\
		\\
		\textbf{Estimate for $I_1^j$}
		\\
		\\
		By Bernstein and the fact that $|u|\lesssim 2^{-\frac{j}{2}}$ on the support of $\varphi_j$,
		\begin{equation}\label{eqn3}
			\begin{split}
				2^{j(\sigma+s-1)}\|P_{<j-4}[\varphi_j(|u|^2)|u|^{2\sigma}]\partial_xu_j\|_{L^1_TL^2_x}&\lesssim 2^{j(\sigma+s-1)}\|\varphi_j(|u|^2)|u|^{2\sigma}\|_{L^1_TL^{\infty}_x}\|\partial_xu_j\|_{L^{\infty}_TL^2_x}
				\\
				&\lesssim T\|u_j\|_{L^{\infty}_TH^s_x}
				\\
				&\lesssim Tb_j\|u\|_{X^s_T}.
			\end{split}
		\end{equation}
		\textbf{Estimate for $I_2^j$}
		\\
		\\
		We have
		\begin{equation}
			g_j=iP_j(P_{\geq j-4}|u|^{2\sigma}\partial_xu)+i[P_j,P_{< j-4}|u|^{2\sigma}]\partial_x\tilde{P}_ju
		\end{equation}
	where $\tilde{P_j}$ is a ``fattened" projection to frequency $\sim 2^j$. By the standard Littlewood-Paley trichotomy, we write
\begin{equation}
\begin{split}
P_j(P_{\geq j-4}|u|^{2\sigma}\partial_xu)&=P_j(\tilde{P}_j|u|^{2\sigma}\partial_x\tilde{P}_{<j}u)+P_j(\tilde{P}_j|u|^{2\sigma}\tilde{P}_j\partial_xu)
\\
&+\sum_{k>j}P_j(\tilde{P}_k|u|^{2\sigma}\tilde{P}_k\partial_xu).
\end{split}
\end{equation}
For the first term, we have by the Moser estimate (\ref{roughmoser}) and Bernstein's inequality,  
	\begin{equation}
	\begin{split}
	2^{j(\sigma+s-1)}\|P_j(\tilde{P}_{j}|u|^{2\sigma}\partial_x\tilde{P}_{<j}u)\|_{L_T^1L_x^2}&\lesssim 2^{j(\sigma+s-1)}\|\tilde{P}_{j}|u|^{2\sigma}\|_{L_T^{1}L_x^\infty}\|\partial_xu\|_{L_T^\infty L_x^{2}}
	\\
	&\lesssim \|\tilde{P}_jD_x^{\sigma+s-1}|u|^{2\sigma}\|_{L_T^1L_x^{\infty}}\|\partial_xu\|_{L_T^\infty L_x^{2}}
	\\
	&\lesssim T^{\frac{1}{2}}b_j\|u\|_{S^1_T}^{2\sigma}\|u\|_{X^s_T}.
	\end{split}
	\end{equation}

 The second term is dealt with similarly. For the third term, we have by Bernstein's inequality
\begin{equation}
\begin{split}
&2^{j(\sigma+s-1)}\|\sum_{k>j}P_j(\tilde{P}_k|u|^{2\sigma}\tilde{P}_k\partial_xu)\|_{L_T^1L_x^2}
\\
&\lesssim T^{\frac{3}{4}}\sum_{k>j}\|\tilde{P}_ku\|_{L_T^4L_x^{\infty}}2^{j(\sigma+s-1)}2^k\|\tilde{P}_k|u|^{2\sigma}\|_{L_T^{\infty}L_x^2}
\\
&\lesssim T^{\frac{3}{4}}\sum_{k>j}2^{(j-k)(\sigma+s-1)}\|D_x^{\sigma+s-1}\tilde{P}_ku\|_{L_T^4L_x^{\infty}}\|\tilde{P}_k\partial_x|u|^{2\sigma}\|_{L_T^{\infty}L_x^2}
\\
&\lesssim T^{\frac{3}{4}}\|u\|_{S^1_T}^{2\sigma}\|u\|_{X^s_T}\sum_{k>j}2^{-(\sigma+s-1)|k-j|}b_k
\\
&\lesssim T^{\frac{3}{4}}b_j\|u\|_{X^s_T}\|u\|_{S^1_T}^{2\sigma}\sum_{k>j}2^{-(\sigma+s-1-\delta)|k-j|}
\\
&\lesssim T^{\frac{3}{4}}b_j\|u\|_{X^s_T}\|u\|_{S^1_T}^{2\sigma}.
\end{split}
\end{equation}
For the commutator term, we have by \Cref{Commutator leibniz}
\begin{equation}
	\begin{split}
2^{j(\sigma+s-1)}[P_j,P_{<j-4}|u|^{2\sigma}]\partial_x\tilde{P}_ju&=2^{j(\sigma+s-2)}L(\partial_xP_{<j-4}|u|^{2\sigma},\tilde{P_j}\partial_xu)
    \end{split}
\end{equation}
for some appropriate translation invariant expression $L$.
\\
\\
This term is easily estimated by  
\begin{equation}
	\begin{split}
2^{j(\sigma+s-2)}\|L(\partial_xP_{<j-4}|u|^{2\sigma},\tilde{P_j}\partial_xu)\|_{L_T^1L_x^2}&\lesssim 2^{j(\sigma+s-2)}\|\partial_xP_{<j-4}|u|^{2\sigma}\|_{L_T^{\infty}L_x^{2}}\|\tilde{P}_j\partial_xu\|_{L_T^{1}L_x^{\infty}}
\\
&\lesssim \|u\|_{L_T^{\infty}L_x^{\infty}}^{2\sigma-1}\|\partial_xu\|_{L_T^{\infty}L_x^2}\|\tilde{P}_jD_x^{\sigma+s-1}u\|_{L_T^{1}L_x^{\infty}}
\\
&\lesssim b_jT^{\frac{3}{4}}\|u\|_{S^1_T}^{2\sigma}\|u\|_{X^s_T}.
    \end{split}
\end{equation} 
Hence, we have
\begin{equation}
I_2^j\lesssim T^{\frac{1}{2}}b_j\|u\|_{X^s_T}\|u\|_{S^1_T}^{2\sigma}.
\end{equation}
		\textbf{Estimate for $I_3^j$}
		\\
		\\
		We expand 
		\begin{equation}
			\begin{split}
				\partial_t\Phi_j=-\frac{1}{2}P_{<j-4}\partial_x^{-1}[2^{j}\chi'(2^j|u|^2)\partial_t|u|^2|u|^{2\sigma}]-\frac{1}{2}P_{<j-4}\partial_x^{-1}[\chi_j(|u|^2)\partial_t|u|^{2\sigma}]=:J_1+J_2.
			\end{split}    
		\end{equation}
		We have 
		\begin{equation}
			\begin{split}\label{eqn4}
				J_1&=-\frac{1}{2}P_{<j-4}\partial_x^{-1}[2^j\chi'(2^j|u|^2)\partial_t|u|^2|u|^{2\sigma}]
				\\
				&=-P_{<j-4}\partial_x^{-1}[2^j\chi'(2^j|u|^2)\RE(\overline{u}u_t)|u|^{2\sigma}]
				\\
				&=-P_{<j-4}\partial_x^{-1}[2^j\chi'(2^j|u|^2)\RE(i\overline{u}u_{xx})|u|^{2\sigma}]-P_{<j-4}\partial_x^{-1}[2^j\chi'(2^{j}|u|^2)\RE(\overline{u}|u|^{2\sigma}u_x)|u|^{2\sigma}]
				\\
				&=-P_{<j-4}\partial_x^{-1}[2^j\chi'(2^j|u|^2)\partial_x\RE(i\overline{u}u_{x})|u|^{2\sigma}]-P_{<j-4}\partial_x^{-1}[2^j\chi'(2^j|u|^2)\RE(\overline{u}|u|^{2\sigma}u_x)|u|^{2\sigma}]
				\\
				&:=K_1+K_2.
			\end{split}
		\end{equation}
		For the first term, $K_1$, in (\ref{eqn4}) we write 
		\begin{equation}\label{eqn5}
			\begin{split}
				-P_{<j-4}\partial_x^{-1}[2^j\chi'(2^j|u|^2)\partial_x\RE(i\overline{u}u_{x})|u|^{2\sigma}]&=-P_{<j-4}[2^j\chi'(2^{j}|u|^2)\RE(i\overline{u}u_{x})|u|^{2\sigma}]
				\\
				&+P_{<j-4}\partial_x^{-1}[2^{2j}\chi''(2^j|u|^2)\partial_x|u|^2\RE(i\overline{u}u_{x})|u|^{2\sigma}]
				\\
				&+P_{<j-4}\partial_x^{-1}[2^j\chi'(2^j|u|^2)\RE(i\overline{u}u_x)\partial_x|u|^{2\sigma}].
			\end{split}
		\end{equation}
		We have for the first term in (\ref{eqn5})
		\begin{equation}
			\begin{split}
				\|P_{<j-4}[2^{j}\chi'(2^j|u|^2)\RE(i\overline{u}u_{x})|u|^{2\sigma}]\|_{L^{\infty}_TL_x^2}&\lesssim 2^{j}\|\chi'(2^{j}|u|^2)\RE(i\overline{u}u_{x})|u|^{2\sigma}\|_{L^{\infty}_TL_x^2}
				\\
				&\lesssim \|u\|_{L^\infty_TL_x^{\infty}}^{2\sigma-1}\|u_x\|_{L^{\infty}_TL_x^2}
			\end{split}
		\end{equation}
		where we used the fact that 
		\begin{equation}
			\|\chi'(2^{j}|u|^2)|u|^{2\sigma+1}\|_{L^\infty_TL_x^{\infty}}=\|\varphi'(2^{j}|u|^2)|u|^{2\sigma+1}\|_{L^\infty_TL_x^{\infty}}\lesssim 2^{-j}\|u\|_{L_T^{\infty}L_x^{\infty}}^{2\sigma-1}.
		\end{equation}
		Now, for the second term in (\ref{eqn5}), we have  
		\begin{equation}
			\begin{split}
				2^{2j}\|P_{<j-4}\partial_x^{-1}[\chi''(2^{j}|u|^2)\partial_x|u|^2\RE(i\overline{u}u_{x})|u|^{2\sigma}]\|_{L_T^{\infty}L_x^\infty}&\lesssim 2^{2j}\|\chi''(2^{j}|u|^2)\partial_x|u|^2\RE(i\overline{u}u_{x})|u|^{2\sigma}\|_{L_T^{\infty}L^1_x}
				\\
				&\lesssim 2^{2j}\|\varphi''(2^{j}|u|^2)\RE(\overline{u}u_x)\RE(i\overline{u}u_{x})|u|^{2\sigma}\|_{L_T^{\infty}L^1_x}
				\\
				&\lesssim 2^{j(1-\sigma)}\|u_x\|_{L_T^{\infty}L^2_x}^2.
			\end{split}
		\end{equation}
		The third term in (\ref{eqn5}) is estimated similarly to the second term. 
		\\
		\\
		Hence, we obtain 
		\begin{equation}
			\begin{split}
				2^{j(\sigma+s-1)}\|K_1u_j\|_{L_T^1L_x^{2}}&\lesssim 2^{j(\sigma+s-1)}2^{j(1-\sigma)}T\|u_x\|_{L^{\infty}_TL_x^2}^2\|u_j\|_{L^{\infty}_TL_x^2}+2^{j(\sigma+s-1)}\|u\|_{L_T^{\infty}L_x^{\infty}}^{2\sigma-1}T^{\frac{3}{4}}\|u_x\|_{L^{\infty}_TL_x^2}\|u_j\|_{L_T^4L_x^{\infty}}
				\\
				&\lesssim T\|u_x\|_{L^{\infty}_TL_x^2}^2\|D_x^su_j\|_{L^{\infty}_TL_x^2}+T^{\frac{3}{4}}\|u\|_{L_T^{\infty}L_x^{\infty}}^{2\sigma-1}\|u_x\|_{L^{\infty}_TL_x^2}\|D_x^{\sigma+s-1}u_j\|_{L_T^4L_x^{\infty}}.
			\end{split}
		\end{equation}
		Next, we estimate $K_2$. We have by Cauchy Schwarz, and Sobolev embedding,
		\begin{equation}
			\begin{split}
				\|P_{<j-4}\partial_x^{-1}[2^{j}\chi'(2^{j}|u|^2)\RE(\overline{u}|u|^{2\sigma}u_x)|u|^{2\sigma}]\|_{L_T^{\infty}L_x^{\infty}}&\lesssim 2^{j}\|\varphi'(2^{j}|u|^2)\RE(\overline{u}|u|^{2\sigma}u_x)|u|^{2\sigma}\|_{L_T^{\infty}L^1_x}
				\\
				&\lesssim 2^{j(\frac{1}{2}-\sigma)}\|u\|_{L_T^{\infty}L_x^{4\sigma}}^{2\sigma}\|u_x\|_{L_T^{\infty}L_x^2}
				\\&\lesssim 2^{j(\frac{1}{2}-\sigma)}\|u\|_{S^1_T}^{2\sigma}\|u_x\|_{L_T^{\infty}L_x^2}
				\\
				&\lesssim \|u\|_{S^1_T}^{2\sigma}\|u_x\|_{L_T^{\infty}L_x^2}
			\end{split}
		\end{equation}
		where we used the fact that $\sigma\geq \frac{1}{2}$.
		\\
		\\
		Hence, we finally obtain the estimate, 
		\begin{equation}
			\begin{split}
				2^{j(\sigma+s-1)}\|u_jJ_1\|_{L_T^1L_x^2}&\lesssim T\|u_x\|_{L^{\infty}_TL_x^2}^2\|D_x^su_j\|_{L^{\infty}_TL_x^2}+T^{\frac{3}{4}}\|u\|_{L_T^{\infty}L_x^{\infty}}^{2\sigma-1}\|u_x\|_{L^{\infty}_TL_x^2}\|D_x^{\sigma+s-1}u_j\|_{L_T^4L_x^{\infty}}
				\\
				&+T\|u\|_{S^1_T}^{2\sigma}\|u_x\|_{L^{\infty}_TL_x^2}\|D_x^{\sigma+s-1}u_j\|_{L^{\infty}_TL^2_x}
				\\
				&\lesssim T^{\frac{3}{4}}(1+\|u\|_{S^1_T}^{4\sigma})\|u_j\|_{X^s_T}.
			\end{split}
		\end{equation}
		Next, we turn to the estimate for $J_2$. We have
		\begin{equation}
			\begin{split}
				J_2&=-\frac{1}{2}P_{<j-4}\partial_x^{-1}[\chi_j(|u|^2)\partial_t|u|^{2\sigma}]
				\\
				&=-\sigma P_{<j-4}\partial_x^{-1}[\chi_j(|u|^2)|u|^{2\sigma-2}\RE(\overline{u}u_t)]
				\\
				&=-\sigma P_{<j-4}\partial_x^{-1}[\chi_j(|u|^2)|u|^{2\sigma-2}\RE(i\overline{u}u_{xx})]-\sigma P_{<j-4}\partial_x^{-1}[\chi_j(|u|^2)|u|^{2\sigma-2}\RE(\overline{u}|u|^{2\sigma}u_x)]:=K_3+K_4.
			\end{split}
		\end{equation}
		For the first term, we have 
	\begin{equation}
			\begin{split}
				K_3&=-\sigma P_{<j-4}[\chi_j(|u|^2)|u|^{2\sigma-2}\RE(i\overline{u}u_{x})]+\sigma P_{<j-4}\partial_x^{-1}[\chi_j(|u|^2)\partial_x|u|^{2\sigma-2}\RE(i\overline{u}u_{x})]
				\\
				&-2^{j}\sigma P_{<j-4}\partial_x^{-1}[\varphi'(2^{j}|u|^2)\partial_x|u|^2|u|^{2\sigma-2}\RE(i\overline{u}u_{x})]
				\\
				&=K_{3,1}+K_{3,2}+K_{3,3}.
			\end{split}
		\end{equation}
		We now must estimate each of the above terms. For the first two terms, we have 
		\begin{equation}
			\|K_{3,1}\|_{L^{\infty}_TL^2_x}\lesssim \|u\|_{L^{\infty}_TL^{\infty}_x}^{2\sigma-1}\|u_x\|_{L^{\infty}_TL^2_x}
		\end{equation}
		and
		\begin{equation}
			\begin{split}
				\|K_{3,2}\|_{L^{\infty}_TL^{\infty}_x}&\lesssim \|\chi_j(|u|^2)\partial_x|u|^{2\sigma-2}\RE(i\overline{u}u_{x})\|_{L_T^{\infty}L^1_x}
				\\
				&\lesssim \|\chi_j(|u|^2)|u|^{2\sigma-4}\RE(\overline{u}u_x)\RE(i\overline{u}u_{x})\|_{L_T^{\infty}L^1_x}
				\\
				&\lesssim 2^{j(1-\sigma)}\|u_x\|_{L_T^{\infty}L^2_x}^2
			\end{split}
		\end{equation}
		where we used the fact that
		\begin{equation}
			\chi_j(|u|^2)|u|^{2\sigma-2}\lesssim 2^{j(1-\sigma)}.
		\end{equation}
		\begin{remark}
			It should be emphasized that the main point of the partial gauge transformation is to be able to estimate the term $K_{3,2}$ above, which involves negative powers of $|u|$.
		\end{remark}
		Now, we turn to the estimate for $K_{3,3}$. We have
		\begin{equation}
			\begin{split}
				\|K_{3,3}\|_{L^{\infty}_TL^{\infty}_x}&\lesssim 2^{j}\|\varphi'(2^{j}|u|^2)|u|^{2\sigma-2}\RE(\overline{u}u_x)\RE(i\overline{u}u_x)\|_{L_T^{\infty}L_x^1}
				\\
				&\lesssim 2^{j(1-\sigma)}\|u_x\|_{L_T^{\infty}L_x^2}^2.
			\end{split}
		\end{equation}
		Hence, we have 
		\begin{equation}
			2^{j(\sigma+s-1)}\|K_3u_j\|_{L^1_TL_x^2}\lesssim T^{\frac{3}{4}}\|u\|_{L^{\infty}_TL^{\infty}_x}^{2\sigma-1}\|u_x\|_{L^{\infty}_TL^2_x}\|D_x^{\sigma+s-1}u_j\|_{L_T^{4}L_x^\infty}+T\|u_x\|_{L_T^{\infty}L_x^2}^2\|D_x^su_j\|_{L^{\infty}_TL_x^2}.
		\end{equation}
		Finally, we estimate $K_4$. We have 
		\begin{equation}
			\begin{split}
				\|K_4\|_{L_T^{\infty}L_x^{\infty}}&\lesssim \|P_{<j-4}\partial_x^{-1}[\chi_j(|u|^2)|u|^{2\sigma-2}\RE(\overline{u}|u|^{2\sigma}u_x)]\|_{L_T^{\infty}L_x^{\infty}}\lesssim \|\chi_j(|u|^2)|u|^{2\sigma-2}\RE(\overline{u}|u|^{2\sigma}u_x)\|_{L_T^{\infty}L_x^{1}}
				\\
				&\lesssim \|u\|_{L_T^{\infty}L_x^{\infty}}^{4\sigma-2}\|u_x\|_{L_T^{\infty}L_x^2}\|u\|_{L_T^{\infty}L_x^2}.
			\end{split}
		\end{equation}
		Hence, combining with the estimate for $K_3$, we obtain  
		\begin{equation}
			\begin{split}
				2^{j(\sigma+s-1)}\|u_jJ_2\|_{L_T^1L_x^2}&\lesssim T^{\frac{3}{4}}\|u\|_{L^{\infty}_TL^{\infty}_x}^{2\sigma-1}\|u_x\|_{L^{\infty}_TL^2_x}\|D_x^{\sigma+s-1}u_j\|_{L_T^4L_x^{\infty}}+T\|u_x\|_{L_T^{\infty}L_x^2}^2\|D_x^su_j\|_{L^{\infty}_TL_x^2}
				\\
				&+T\|u\|_{L_T^\infty L_x^{\infty}}^{4\sigma-2}\|u_x\|_{L_T^{\infty}L_x^{2}}\|u\|_{L_T^{\infty}L_x^2}\|D_x^{\sigma+s-1}u_j\|_{L_T^{\infty}L_x^{2}}  
				\\
				&\lesssim_T T^{\frac{3}{4}}(1+\|u\|_{S^1_T}^{4\sigma})\|u_j\|_{X^s_T}.
			\end{split}
		\end{equation}
		Now combining this with the estimate for $J_1$ finally yields the desired estimate for $I_3^j$. Namely, we have
		\begin{equation}
			\begin{split}
				I_3^j&\lesssim T^{\frac{3}{4}}(1+\|u\|_{S^1_T}^{4\sigma})\|u_j\|_{X^s_T}
				\\
				&\lesssim T^{\frac{3}{4}}b_j(1+\|u\|_{S^1_T}^{4\sigma})\|u\|_{X^s_T}.
			\end{split}
		\end{equation}
		\textbf{Estimate for $I_4^j$}
		\\
		\\
		This term is straightforward to deal with. Indeed, after expanding $\partial_x^2\Phi_j$ we have
		\begin{equation}
			\begin{split}
				\|\partial_x^2\Phi_j\|_{L_T^{\infty}L_x^2}&\lesssim 2^{j}\|\varphi'(2^{j}|u|^2)\RE(\overline{u}u_x)|u|^{2\sigma}\|_{L^{\infty}_TL_x^2}+\|\chi_j(|u|^2)\RE(|u|^{2\sigma-2}\overline{u}u_x)\|_{L^{\infty}_TL_x^2}
				\\
				&\lesssim \|u\|_{L_T^{\infty}L_x^\infty}^{2\sigma-1}\|u_x\|_{L_T^{\infty}L_x^2}.
			\end{split}
		\end{equation}
		Hence, 
		\begin{equation}
		\begin{split}
			I^j_4&\lesssim T^{\frac{3}{4}}\|u\|_{L_T^{\infty}L_x^{\infty}}^{2\sigma-1}\|u_x\|_{L_T^{\infty}L_x^2}\|D_x^{\sigma+s-1}u_j\|_{L_T^4L_x^{\infty}}
			\\
			&\lesssim T^{\frac{3}{4}}\|u\|_{S^1_T}^{2\sigma}\|u_j\|_{X^s_T}
			\\
			&\lesssim T^{\frac{3}{4}}b_j\|u\|_{S^1_T}^{2\sigma}\|u\|_{X^s_T}.
			\end{split}
		\end{equation}
		\textbf{Estimate for $I_5^j$}
		\\
		\\
		The estimate for $I_5^j$ is also straightforward as it doesn't involve any differentiated terms. Indeed, we have 
		\begin{equation}
			\|\partial_x\Phi_j\|_{L_T^{\infty}L_x^\infty}^2\lesssim \|u\|_{L_T^\infty L_x^{\infty}}^{4\sigma}.
		\end{equation}
		Hence, by Sobolev embedding,
		\begin{equation}
		\begin{split}
			I_5^j&\lesssim T\|u\|_{L^{\infty}_TL^{\infty}_x}^{4\sigma}\|D_x^{\sigma+s-1}u_j\|_{L_T^{\infty}L_x^2}
			\\
			&\lesssim T\|u\|_{S^1_T}^{4\sigma}\|u_j\|_{X^s_T}
			\\
			&\lesssim Tb_j\|u\|_{S^1_T}^{4\sigma}\|u\|_{X^s_T}.
			\end{split}
		\end{equation}
		Now, combining all the estimates above  completes the proof of \Cref{Yestimate}. 
\begin{remark}\label{AlternateY}
By taking $b_j$ to instead be a $S_T^s$ frequency envelope for $u$, and repeating the proof almost verbatim with \Cref{alternativemoser} in place of (\ref{roughmoser}), we instead obtain
\begin{equation}\label{strichpure}
			\begin{split}
				\|P_ju\|_{Y^s_T}&\lesssim_{\|u\|_{S^1_T}} a_j\|u_0\|_{H^s_x}
				+T^{\frac{1}{2}}b_j(1+\|u\|_{S^1_T}^{4\sigma})\|u\|_{S^s_T}.
			\end{split}
		\end{equation}
This will be relevant for when we later establish local well-posedness in the high regularity regime $2-\sigma<s<4\sigma$ for the full range of $\frac{1}{2}<\sigma<1$. Specifically, this will be important for establishing a priori bounds in the range $2-\sigma<s\leq\frac{3}{2}$ when Sobolev embedding is not suitable for controlling the term $\|u_x\|_{L_T^4L_x^{\infty}}$. The reason the proof of \eqref{strichpure} is almost identical to the current proof is that we have not yet used the maximal function part of the norm of $X^s_T$; we will begin using this part of the norm in the proof of \Cref{NLEE}.
\end{remark}
\begin{remark}\label{highregremark}
As a second important remark, the estimate (\ref{strichpure}) also holds for $T\lesssim 1$ if the nonlinearity $i|u|^{2\sigma}u_x$ is replaced by the spatially regularized and time-truncated nonlinearity $i\eta P_{<k}|u|^{2\sigma}u_x$, where $k\in\mathbb{N}$ and $\eta=\eta(t)$ is a time-dependent cutoff function supported in $(-2,2)$ and equal to $1$ on $[-1,1]$. This fact won't be relevant for the low regularity construction, but will be important for the high regularity construction in Sections 5 and 6 where the cutoff $\eta$ is needed for estimating (fractional order) time derivatives of a solution $u$ to \eqref{gDNLS}. Since the proof of this estimate is nearly identical to \Cref{Yestimate}, we omit the details. Nevertheless, for the sake of completeness, we state this observation in the following lemma. 
\end{remark}

\begin{lemma}\label{modifiedlemma}
Let $k\in \mathbb{N}$, $\sigma\in (\frac{1}{2},1)$, $s\in [1,\frac{3}{2}]$,  and $T\lesssim 1$. Let $\eta$ be a time-dependent cutoff function supported in $(-2,2)$ with $\eta=1$ on $[-1,1]$. Let $v,w\in S_T^s$ with $\|v\|_{S_T^s}$, $\|w\|_{S_T^s}\lesssim 1$. Assume that $u,v\in S^s_T$ solve the equations 

\begin{equation}
			\begin{cases}
				&(i\partial_t+\partial_x^2)u =i\eta P_{<k}|v|^{2\sigma}\partial_xu,
				\\
				&u(0)=  u_0,
			\end{cases}
		\end{equation}
and
\begin{equation}\label{second eq}
			\begin{cases}
				&(i\partial_t+\partial_x^2)v =i\eta P_{<k}|w|^{2\sigma}\partial_xv,
				\\
				&v(0)=  u_0,
			\end{cases}
		\end{equation}
respectively. Then $u$ satisfies the estimate
\begin{equation}
\|u\|_{Y_T^s}\lesssim \|u_0\|_{H_x^s}+T^{\frac{1}{2}}\|u\|_{S^s_T}.
\end{equation}
\end{lemma}
As mentioned, the proof of \Cref{modifiedlemma} proceeds in a nearly identical fashion to \Cref{Yestimate}, so we omit the details. The main difference is that $\Phi_j$ is replaced by
\begin{equation}
\Phi_j=-\frac{1}{2}\eta(t)P_{<j-4}P_{<k}\partial_x^{-1}[\chi_j(|v|^2)|v|^{2\sigma}].
\end{equation}
The requirement \eqref{second eq} that $v$ solves an additional \eqref{gDNLS} type equation is merely relevant for the $I^j_3$ estimate when time derivatives fall on $\Phi_j$, and hence on $v$. In practice, \Cref{modifiedlemma} will be used in the construction of solutions at high regularity in Sections 5, 6 and 7. 
\\
\\
		Next, we turn to proving \Cref{NLEE}.
		\begin{proof}
	Again, we begin by writing the equation in a paradifferential fashion, 
		\begin{equation}
			i\partial_tu_j+\partial_x^2u_j=iP_{<j-4}|u|^{2\sigma}\partial_xu_j+iP_j(P_{\geq j-4}|u|^{2\sigma}\partial_xu)+i[P_j,P_{<j-4}|u|^{2\sigma}]\partial_xu.
		\end{equation}
A simple energy estimate (i.e.~multiplying by $-i2^{2js}\overline{u_j}$, taking real part and integrating), and Bernstein's inequality gives 
\begin{equation}
\begin{split}
\|u_j\|_{L_T^{\infty}H_x^s}^2&\lesssim \|u_j(0)\|_{H_x^s}^2+2^{2js}\int_{0}^{T}\left\lvert\int_{\mathbb{R}}P_{<j-4}|u|^{2\sigma}\partial_x|u_j|^2\right\rvert+2^{2js}\int_{0}^{T}\left\lvert\int_{\mathbb{R}}\overline{u_j}P_j(P_{\geq j-4}|u|^{2\sigma}\partial_xu)\right\rvert
\\
&+2^{2js}\int_{0}^{T}\left\lvert\int_{\mathbb{R}}\overline{u_j}[P_j,P_{<j-4}|u|^{2\sigma}]\partial_xu\right\rvert
\\
&:=\|u_j(0)\|_{H_x^s}^2+I_1^j+I_2^j+I_3^j.
\end{split}
\end{equation}
		\textbf{Estimate for $I_1^j$}
		\\
		\\
		For the first term, we integrate by parts and estimate using standard interpolation inequalities, Bernstein's inequality, H\"older's inequality and \Cref{Moservec}
		\begin{equation}
			\begin{split}
				2^{2js}\int_{0}^{T}\left\lvert\int_{\mathbb{R}}|u_j|^2P_{<j-4}\partial_x|u|^{2\sigma}\right\rvert&\lesssim 2^{2js}\|P_{<j-4}\partial_x|u|^{2\sigma}|u_j|^{2(1-\sigma)}\|_{L_x^1L_T^{\frac{1}{1-\sigma}}}\|u_j\|_{L_x^{\infty}L_T^2}^{2\sigma}
				\\
				&\lesssim 2^{2js}\|P_{<j-4}\partial_x|u|^{2\sigma}\|_{L_x^{\frac{1}{\sigma}}L_T^{\frac{1}{\epsilon(1-\sigma)}}}\|u_j\|_{L_x^2L_T^{\frac{2}{1-\epsilon}}}^{2(1-\sigma)}\|u_j\|_{L_x^{\infty}L_T^2}^{2\sigma}
				\\
				&\lesssim \|P_{<j-4}(D_x^{\sigma-\frac{1}{2}}|u|^{2\sigma})\|_{L_x^{\frac{1}{\sigma}}L_T^{\frac{1}{\epsilon(1-\sigma)}}}\|D_x^{s-c_1\epsilon}u_j\|_{L_x^2L_T^{\frac{2}{1-\epsilon}}}^{2(1-\sigma)}\|D_x^{s+\frac{3}{4\sigma}-\frac{1}{2}+c_2\epsilon}u_j\|_{L_x^{\infty}L_T^2}^{2\sigma}
				\\
				&\lesssim T^{(1-\sigma)(1-\epsilon)}\|P_{<j-4}(D_x^{\sigma-\frac{1}{2}-\epsilon}|u|^{2\sigma})\|_{L_x^{\frac{1}{\sigma}}L_T^{\frac{1}{\epsilon(1-\sigma)}}}\|u_j\|_{X^s_T}^2
				\\
				&\lesssim T^{1-\sigma}\|u\|_{L_x^{2}L_T^\infty}^{2\sigma-1}\|D_x^{\sigma-\frac{1}{2}-\epsilon}u\|_{L_x^2L_T^{\infty}}\|u_j\|_{X^s_T}^2
				\\
				&\lesssim T^{1-\sigma}\|u\|_{Y^1_T}^{2\sigma}\|u_j\|_{X^s_T}^2
				\\
				&\lesssim T^{1-\sigma}b_j^2\|u\|_{Y^1_T}^{2\sigma}\|u\|_{X^s_T}^2,
			\end{split}
		\end{equation}
where $c_1,c_2$ are fixed positive constants, and $\epsilon>0$ is sufficiently small. Observe that going from line 3 to line 4 uses the fact that $\sigma> \frac{\sqrt{3}}{2}$ since $s+\frac{3}{4\sigma}-\frac{1}{2}< s+\sigma-\frac{1}{2}$ precisely when $\sigma>\frac{\sqrt{3}}{2}$. 
\\
\\
		\textbf{Estimate for $I_2^j$}
		\\
		\\
		We have by the Littlewood-Paley trichotomy 
		\begin{equation}\label{3.71}
			\begin{split}
				2^{2js}\int_{\mathbb{R}}P_j(P_{\geq j-4}|u|^{2\sigma}\partial_xu)\overline{u_j}&=2^{2js}\int_{\mathbb{R}} \tilde{P}_j(|u|^{2\sigma})\tilde{P}_{<j}\partial_xu\overline{u}_j+2^{2js}\sum_{k>j}\int_{\mathbb{R}} \overline{u_j}P_j(\tilde{P}_k(|u|^{2\sigma})\tilde{P_k}\partial_xu)
			\end{split}
		\end{equation}
	for appropriate ``fattened" Littlewood-Paley projections $\tilde{P_j}$.
		For the first term, using Bernstein's inequality and H\"older's inequality, and that $2^{-\delta j}\lesssim b_j$ we have, 
		\begin{equation}\label{integral1}
			\begin{split}
				2^{2js}\int_{0}^{T}\left\lvert\int_\mathbb{R} \tilde{P}_j(|u|^{2\sigma})\tilde{P}_{<j}\partial_xu\overline{u}_j\right\rvert&\lesssim2^{j(\frac{5}{2}-\sigma+s)}\|\tilde{P_j}|u|^{2\sigma}\|_{L_x^{\frac{2}{2\sigma-1}}L_T^2}\|\tilde{P}_ju\|_{L_x^{\infty}L_T^2}^{2\sigma-1}\|\tilde{P}_ju\|_{L_x^{2}L_T^2}^{2(1-\sigma)}\|\tilde{P}_{<j}D_x^{\sigma+s-\frac{3}{2}}u\|_{L_x^2L_T^{\infty}}
				\\
				&\lesssim T^{1-\sigma}\|D_x^{2+\sigma-2\sigma^2+\delta}\tilde{P}_j(|u|^{2\sigma})\|_{L_x^{\frac{2}{2\sigma-1}}L_T^2}\|\tilde{P}_ju\|_{X^s_T}\|u\|_{X^s_T}
				\\
				&\lesssim T^{1-\sigma}b_j\|D_x^{2+\sigma-2\sigma^2+2\delta}\tilde{P}_j(|u|^{2\sigma})\|_{L_x^{\frac{2}{2\sigma-1}}L_T^2}\|\tilde{P}_ju\|_{X^s_T}\|u\|_{X^s_T}.
			\end{split}
		\end{equation}
Note that the first line follows since $s\in [1,\frac{3}{2}]$. Now, we estimate $\|D_x^{2+\sigma-2\sigma^2+2\delta}\tilde{P}_j(|u|^{2\sigma})\|_{L_x^{\frac{2}{2\sigma-1}}L_T^2}$. For notational convenience, write $2+\sigma-2\sigma^2+2\delta=\alpha$. We employ the Littlewood-Paley  trichotomy and then H\"older's and Bernstein's inequality to obtain
\begin{equation}
\begin{split}
\|D_x^{\alpha}\tilde{P}_j(|u|^{2\sigma})\|_{L_x^{\frac{2}{2\sigma-1}}L_T^2}&\lesssim \|D_x^{\alpha-1}\tilde{P}_j(|u|^{2\sigma-2}\overline{u}u_x)\|_{L_x^{\frac{2}{2\sigma-1}}L_T^2}
\\
&\lesssim \|D_x^{\alpha-1}\tilde{P}_j(\tilde{P}_{<j}(|u|^{2\sigma-2}\overline{u})\tilde{P}_ju_x)\|_{L_x^{\frac{2}{2\sigma-1}}L_T^2}+\|D_x^{\alpha-1}\tilde{P}_j(\tilde{P}_{>j}(|u|^{2\sigma-2}\overline{u})u_x)\|_{L_x^{\frac{2}{2\sigma-1}}L_T^2}
\\
&\lesssim \|u\|_{L_x^{2}L_T^{\infty}}^{2\sigma-1}\|D_x^{\alpha}\tilde{P}_ju\|_{L_x^{\infty}L_T^2}+\|D_x^{\alpha-1}(|u|^{2\sigma-2}\overline{u})\|_{L_T^{\infty}L_x^{\infty}}\|u_x\|_{L_x^{\frac{2}{2\sigma-1}}L_T^2}.
\end{split}
\end{equation}
 Observe that $\|D_x^{\alpha}u\|_{L_x^{\infty}L_T^2}\lesssim \|u\|_{Y_T^1}$ since $\alpha<\sigma+\frac{1}{2}$ when $\sigma>\frac{\sqrt{3}}{2}$. Furthermore, by \Cref{Holder} and Sobolev embedding, we have
\begin{equation}
\|D_x^{\alpha-1}(|u|^{2\sigma-2}\overline{u})\|_{L_T^{\infty}L_x^{\infty}}\lesssim \|\langle D_x\rangle^{\frac{\alpha-1+\epsilon}{2\sigma-1}}u\|_{L_T^{\infty}L_x^{\infty}}^{2\sigma-1}\lesssim \|u\|_{S_T^1}^{2\sigma-1}    
\end{equation}
where the last inequality again follows because $\sigma>\frac{\sqrt{3}}{2}$. Furthermore, by interpolating $\|u_x\|_{L_x^{\frac{2}{2\sigma-1}}L_T^2}$ between $L_x^2L_T^2$ and $L_x^{\infty}L_T^2$, we see that $\|u_x\|_{L_x^{\frac{2}{2\sigma-1}}L_T^2}\lesssim \|u\|_{X_T^1}$. Hence, we can control (\ref{integral1}) by
\begin{equation}
T^{1-\sigma}b_j^2\|u\|_{X_T^1}^{2\sigma}\|u\|_{X_T^s}^2.
\end{equation}
For the other term in \eqref{3.71}, we have
\begin{equation}
\begin{split}
&2^{2js}\int_{0}^{T}\left\lvert\sum_{k>j}\int_{\mathbb{R}} \overline{u_j}P_j(\tilde{P}_k(|u|^{2\sigma})\tilde{P_k}\partial_xu)\right\rvert
\\
&\lesssim 2^{js}T^{(1-\sigma)}\|D_x^su_j\|_{L_T^{\infty}L_x^2}^{2(1-\sigma)}\|D_x^su_j\|_{L_x^{\infty}L_T^2}^{2\sigma-1}\sum_{k>j}2^{k}\|\tilde{P}_k(|u|^{2\sigma})\|_{L_x^{\frac{2}{2\sigma-1}}L_T^2}\|\tilde{P}_ku\|_{L_x^{2}L_T^{\infty}}
\\
&\lesssim 2^{j(s-\frac{1}{2}(1-2\sigma)^2)}T^{(1-\sigma)}\|u_j\|_{X^s_T}\sum_{k>j}2^{k}\|\tilde{P}_k(|u|^{2\sigma})\|_{L_x^{\frac{2}{2\sigma-1}}L_T^2}\|\tilde{P}_ku\|_{L_x^{2}L_T^{\infty}}
\\
&\lesssim 2^{j(s-\frac{1}{2}(1-2\sigma)^2)}T^{(1-\sigma)}\|u_j\|_{X^s_T}\sum_{k>j}2^{k(\frac{3}{2}-\sigma-s+1)}\|\tilde{P}_k(|u|^{2\sigma})\|_{L_x^{\frac{2}{2\sigma-1}}L_T^2}\|\tilde{P}_kD_x^{s+\sigma-\frac{3}{2}}u\|_{L_x^{2}L_T^{\infty}}
\\
&\lesssim T^{(1-\sigma)}\|u_j\|_{X^s_T}\sum_{k>j}2^{(j-k)(s-\frac{1}{2}(1-2\sigma)^2)}\|\tilde{P}_k(D_x^{2+\sigma-2\sigma^2}|u|^{2\sigma})\|_{L_x^{\frac{2}{2\sigma-1}}L_T^2}\|\tilde{P}_kD_x^{s+\sigma-\frac{3}{2}}u\|_{L_x^{2}L_T^{\infty}}
\\
&\lesssim T^{(1-\sigma)}b_j^2\|u\|_{X^s_T}^2\|u\|_{X^1_T}^{2\sigma}\sum_{k>j}2^{(j-k)((s-\frac{1}{2}(1-2\sigma)^2)-\delta)}
\\
&\lesssim T^{(1-\sigma)}b_j^2\|u\|_{X^s_T}^2\|u\|_{X^1_T}^{2\sigma}
\end{split}
\end{equation}
where we estimated $\|\tilde{P}_k(D_x^{2+\sigma-2\sigma^2}|u|^{2\sigma})\|_{L_x^{\frac{2}{2\sigma-1}}L_T^2}$ in essentially the same way as we did with the previous term.
\\
\\
		\textbf{Estimate for $I_3^j$}
	\\
	\\
	We have 
	\begin{equation}
		\begin{split}
			[P_j,P_{<j-4}|u|^{2\sigma}]\partial_xu&=[P_j,P_{<j-4}|u|^{2\sigma}]\partial_x\tilde{P}_ju
			\\
			&=2^{-j}\int_{\mathbb{R}^2}K(y)\partial_xP_{<j-4}|u|^{2\sigma}(x+y_1)\partial_x\tilde{P}_ju(x+y_2)dy
		\end{split}
	\end{equation}
	for some kernel $K\in L^1$ with $\|K\|_{L^1}\lesssim 1$ (with a bound independent of $j$), see \Cref{Commutator leibniz}. Hence,  
	\begin{equation}
		\begin{split}
		\int_{0}^{T}	\left\lvert\int_{\mathbb{R}}\overline{u_j}[P_j,P_{<j-4}|u|^{2\sigma}]\partial_x\tilde{P}_ju\right\rvert&\lesssim 2^{-j}\sup_{y\in\mathbb{R}^2}\int_{0}^{T}\int_{\mathbb{R}}|\partial_xP_{<j-4}|u|^{2\sigma}(x+y_1)||\partial_x\tilde{P}_ju(x+y_2)||u_j|.
		\end{split}
	\end{equation}
	This is estimated analogously to $I_j^1$. Indeed, we obtain by Cauchy Schwarz, Bernstein's inequality and \Cref{Moservec},  
	\begin{equation}
		\begin{split}
			2^{2js}\int_{0}^{T}\left\lvert\int_{\mathbb{R}}\overline{u_j}[P_j,P_{<j-4}|u|^{2\sigma}]\partial_x\tilde{P}_ju\right\rvert&\lesssim 2^{2js}\|\tilde{P}_jD_x^{\frac{3}{4\sigma}-\frac{1}{2}+c_1\epsilon}u\|_{L_x^{\infty}L_T^2}^{2\sigma}\|\tilde{P}_ju\|_{L_T^2L_x^2}^{2(1-\sigma)}\|u\|_{L_x^2L_T^\infty}^{2\sigma-1}\|D_x^{\sigma-\frac{1}{2}-c_2\epsilon}u\|_{L_x^2L_T^{\infty}}
			\\
			&\lesssim T^{(1-\sigma)}\|u\|_{Y^1_T}^{2\sigma}\|\tilde{P}_ju\|_{X^s_T}^2
			\\
			&\lesssim T^{(1-\sigma)}b_j^2\|u\|_{Y^1_T}^{2\sigma}\|u\|_{X^s_T}^2,
		\end{split}
	\end{equation}
where $c_1,c_2$ are positive constants depending on $\sigma,s$. The second line follows from the fact that $\frac{3}{4\sigma}-\frac{1}{2}<\sigma-\frac{1}{2}$ as long as $\sigma>\frac{\sqrt{3}}{2}$. 
\\
\\
Hence, we obtain
\begin{equation}
\begin{split}
\|P_ju\|_{L_T^{\infty}H^s_x}\lesssim a_j\|u_0\|_{H^s_x}+T^{\frac{1-\sigma}{2}}b_j\|u\|_{X^1_T}^{\sigma}\|u\|_{X^s_T},
\end{split}
\end{equation}
thus completing the proof of the $L_T^{\infty}H_x^s$ estimate. 
\end{proof}
\textbf{Proof of \Cref{Ubounds}}
\\
\\
We combine the energy estimate and the $Y^s$ estimate to obtain 
\begin{equation}
\begin{split}
\|P_ju\|_{X^s_T}&\lesssim_{\|u\|_{S^1_T}} a_j\|u_0\|_{H^s_x}
				+T^{\frac{1}{2}}b_j(1+\|u\|_{S^1_T}^{4\sigma})\|u\|_{X^s_T}+T^{\frac{1-\sigma}{2}}b_j\|u\|_{X^1_T}^{\sigma}\|u\|_{X^s_T}.
\end{split}
\end{equation}
This proves part a) of \Cref{Ubounds}.
\\
\\
Now we move to part b). Let us first assume $T\ll 1$ (but independent of $\epsilon$). There are two components to consider. For high frequency, square summing over $j>0$ shows 
\begin{equation}
\begin{split}
\|P_{>0}u\|_{X^s_T}&\lesssim_{\|u\|_{S^1_T}} \|u_0\|_{H^s_x}
				+T^{\frac{1}{2}}(1+\|u\|_{S^1_T}^{4\sigma})\|u\|_{X^s_T}+T^{\frac{1-\sigma}{2}}\|u\|_{X^1_T}^{\sigma}\|u\|_{X^s_T}.
\end{split}
\end{equation}
On the other hand, directly applying the maximal function/Strichartz estimates in \Cref{Hom max} and \Cref{Inhom Max} and Bernstein's inequality to $P_{\leq 0}u$, we easily obtain 
\begin{equation}
\|P_{\leq 0}u\|_{X^s_T}\lesssim \|u_0\|_{L_x^2}+\|P_{\leq 0}(|u|^{2\sigma}u_x)\|_{L_T^1L_x^2}\lesssim \|u_0\|_{L_x^2}+T\|u\|^{2\sigma+1}_{S^1_T}.
\end{equation}
From the above bounds, we see that the $X_T^s$ norm of $u$ converges to the $H_x^1$ norm of the initial data as $T\to 0^+$. Let us now make the bootstrap assumption $\|u\|_{X^1_T} \leq \epsilon^{\frac{1}{2}}$. We then obtain from the above estimates,
\begin{equation}
\|u\|_{X^s_T}\lesssim_{\|u\|_{X^1_T}} \|u_0\|_{H_x^s}\leq \epsilon
\end{equation}
where $T\ll 1$ (but independent of $\epsilon$) and  $1\leq s\leq\frac{3}{2}$. To obtain the estimate for $T\sim 1$, we iterate the above procedure $O(T^{-1})$ many times (after suitable translating the initial data). This proves part b) of \Cref{Ubounds}.
\\
\\
Next, we turn to the proof of \Cref{linbounds}. We proceed in a similar manner to \Cref{Ubounds}, and prove separate estimates for the $Y^0_T$ and $L_T^{\infty}L_x^2$ components of the $X^0_T$ norm. For this purpose, we have the following two lemmas: 
	\begin{lemma} ($Y^0_T$ estimate)\label{Yestimate2} Let $v$, $\sigma$, $T$, $w$, $g$ and $a$ be as in \Cref{linbounds}. Then we have the $Y^0_T$ estimate, 
		\begin{equation}\label{Yestimate3}
			\|v\|_{Y^0_T}\lesssim \|v_0\|_{L_x^2}+T^{\frac{1}{2}}(1+\|w\|_{X^1_T}^{4\sigma})\|v\|_{X^0_T}+T^{1-\sigma}\|g\|_{Z}\|a\|_{X^1_T}\|v\|_{X^0_T}.
		\end{equation}
	\end{lemma}
	\begin{lemma}($L_T^{\infty}L_x^2$ estimate)\label{LEE}
Let $v$, $\sigma$, $T$, $w$, $g$ and $a$ be as in \Cref{linbounds}. Then we have the estimate,
\begin{equation}\label{Yestimate4}
\|P_jv\|_{l_j^2L_T^{\infty}L_x^2}^2\lesssim \|v_0\|_{L_x^2}^2+T^{1-\sigma}\|g\|_Z\|a\|_{X^1_T}\|v\|_{X^0_T}^2+T^{1-\sigma}\|w\|_{X^1_T}^{2\sigma}\|v\|_{X^0_T}^2.
\end{equation}
		\end{lemma}
		We begin with \Cref{Yestimate2}. The proof is almost the same as \Cref{Yestimate} with a couple of small differences. As in (\ref{paradiff2}), we consider a similar paradifferential truncation of (\ref{lineq}), 
		\begin{equation}
			\begin{split}
				(i\partial_t+\partial_x^2)v_j&=iP_{<j-4}(\chi_j(|w|^2)|w|^{2\sigma})\partial_xv_j+iP_{<j-4}(\varphi_j(|w|^2)|w|^{2\sigma})\partial_xv_j+f_j+g_j
			\end{split}
		\end{equation}
		where $\varphi_j$ and $\chi_j$ are as in (\ref{paradiff2}) and 
		\begin{equation}
			f_j:=iP_j(P_{\geq j-4}|w|^{2\sigma}\partial_xv)+i[P_j,P_{<j-4}|w|^{2\sigma}]\partial_xv,
		\end{equation}
		\begin{equation}
			g_j:=2P_j(\partial_xa\RE(gv)).
		\end{equation}
		Analogously to the proof of \Cref{Ubounds}, we define
		\begin{equation}
			\Psi_j(x)=-\frac{1}{2}P_{<j-4}\partial_x^{-1}[\chi_j(|w|^2)|w|^{2\sigma}]
		\end{equation}
		and consider the new variable
		\begin{equation}
			\tilde{v}_j:=v_je^{i\Psi_j}.
		\end{equation}
		By direct computation, $\tilde{v}_j$ solves the equation,
		\begin{equation}
			\begin{cases}
				&(i\partial_t+\partial_x^2)\tilde{v}_j=ie^{i\Psi_j}P_{<j-4}[\varphi_j(|w|^2)|w|^{2\sigma}]\partial_xv_j+(-\partial_t\Psi_j+i\partial_x^2\Psi_j-(\partial_x\Psi_j)^2)\tilde{v}_j
				\\
				&\hspace{20mm}+2e^{i\Psi_j}P_j(\partial_xa\RE(gv))+e^{i\Psi_j}f_j,
				\\
				&\tilde{v}_j(0)=e^{i\Psi_j}v_j(0).
			\end{cases}
		\end{equation}
		Now, \Cref{Hom max}, \Cref{Inhom Max}  and a similar argument to \Cref{Ubounds} yields the estimate   
		\begin{equation}
		\begin{split}
			\|v\|_{Y^0_T}&\lesssim_T \|v_0\|_{L_x^2}+T^{\frac{1}{2}}[1+\|w\|_{X^1_T}]^{4\sigma}\|v\|_{X^0_T}
		\\	
		&+\left(\sum_{j>0}\|\langle D_x\rangle^{\sigma-1} P_j(g\partial_xav)\|_{L_T^1L_x^2}^2\right)^{\frac{1}{2}}.
			\end{split}
		\end{equation}

It remains to control the last term. Indeed, we have by Bernstein and Sobolev embedding,
\begin{equation}
\begin{split}
\|\langle D_x\rangle^{\sigma-1}P_j(g\partial_xav)\|_{L_T^1L_x^2}&\lesssim 2^{j(\sigma-1)}\|P_j(g\partial_xav)\|_{L_T^1L_x^2}
\\
&\lesssim 2^{j(\sigma-1)}\|P_{<j-4}(\partial_xag)\tilde{P}_jv\|_{L_T^1L_x^2}+\|P_j(P_{\geq j-4}(\partial_xag)v)\|_{L_T^1L_x^{\frac{2}{3-2\sigma}}}.
\end{split}
\end{equation}
For the first term, we have by Bernstein's inequality, 
\begin{equation}
\begin{split}
2^{j(\sigma-1)}\|P_{<j-4}(\partial_xag)\tilde{P}_jv\|_{L_T^1L_x^2}&\lesssim T^{\frac{3}{4}}\|\partial_xa\|_{L_T^{\infty}L_x^2}\|g\|_{L_T^{\infty}L_x^{\infty}}\|\tilde{P}_jD_x^{\sigma-1}v\|_{L_T^4L_x^{\infty}}
\\
&\lesssim T^{\frac{3}{4}}\|a\|_{X^1_T}\|g\|_{Z}\|\tilde{P}_jD_x^{\sigma-1}v\|_{L_T^4L_x^{\infty}}.
\end{split}
\end{equation}
For the second term, we have by the usual Littlewood-Paley trichotomy,
\begin{equation}
\begin{split}
\|P_j(P_{\geq j-4}(\partial_xag)v)\|_{L_T^1L_x^{\frac{2}{3-2\sigma}}}&\lesssim \|P_j(\tilde{P}_j(\partial_xag)P_{<j}v)\|_{L_T^1L_x^{\frac{2}{3-2\sigma}}}+\sum_{k\geq j}\|P_j(\tilde{P}_k(\partial_xag)\tilde{P}_kv)\|_{L_T^1L_x^{\frac{2}{3-2\sigma}}}:=K_1^j+K_2^j.    
\end{split}
\end{equation}
To estimate $K_1^j$, we have 
\begin{equation}
\begin{split}
\|P_j(\tilde{P}_j(\partial_xag)P_{<j}v)\|_{L_T^1L_x^{\frac{2}{3-2\sigma}}}&\lesssim \|\tilde{P}_j(\partial_xag)\|_{L_T^{2}L_x^2}\|P_{<j}v\|_{L_T^2L_x^{\frac{1}{1-\sigma}}}
\\
&\lesssim \|D_x^{(1-\sigma+\epsilon)(2\sigma-1)}\tilde{P}_j(g\partial_xa)\|_{L_T^2L_x^2}\|P_{<j}v\|_{L_T^2L_x^2}^{2(1-\sigma)}\|P_{<j}D_x^{\sigma-1-\epsilon}v\|_{L_T^2L_x^{\infty}}^{2\sigma-1}
\\
&\lesssim T^{1-\sigma}\|D_x^{(1-\sigma+\epsilon)(2\sigma-1)}\tilde{P}_j(g\partial_xa)\|_{L_T^2L_x^2}\|v\|_{X^0_T}
\end{split}
\end{equation}
where in the last line we used the fact that by Sobolev embedding,
\begin{equation}
\begin{split}
\|P_{<j}D_x^{\sigma-1-\epsilon}v\|_{L_T^2L_x^{\infty}}&\lesssim \|v\|_{L_T^{\infty}L_x^2}+\|P_{>0}v\|_{X_T^0}\lesssim \|v\|_{X_T^0}
\end{split}
\end{equation}
as well as $\|P_{<j}v\|_{L_T^2L_x^2}\lesssim T^{\frac{1}{2}}\|v\|_{X_T^0}$.
Now, setting $\alpha=(1-\sigma+\epsilon)(2\sigma-1)$, we have by Bernstein's inequality, and a simple application of the Littlewood-Paley trichotomy, 
\begin{equation}
\begin{split}
\|D_x^{\alpha}\tilde{P}_j(g\partial_xa)\|_{L_T^2L_x^2}&\lesssim 2^{-j\epsilon}\|D_x^{\alpha+\epsilon}\tilde{P}_j(g\partial_xa)\|_{L_T^2L_x^2}
\\
&\lesssim 2^{-j\epsilon}\|D_x^{\alpha+\epsilon}\partial_xa\|_{L_x^{\frac{1}{1-\sigma}}L_T^2}\|g\|_{L_x^{\frac{2}{2\sigma-1}}L_T^{\infty}}+2^{-j\epsilon}\|\partial_xa\|_{L_T^{\infty}L_x^2}\|D_x^{\alpha+\epsilon}g\|_{L_{T}^2L_x^{\infty}}.
\end{split}
\end{equation}
Next, by interpolating $\|D^{\alpha+\epsilon}_x\partial_xa\|_{L_x^{\frac{1}{1-\sigma}}L_T^2}$ between $D_x^{\frac{\alpha+2\epsilon}{2\sigma-1}}a$ in $L_x^{\infty}L_T^2$ and $\partial_xa$ in $L_x^2L_T^2$, we see that for $\epsilon$ small enough, $\|D^{\alpha+\epsilon}_x\partial_xa\|_{L_x^{\frac{1}{1-\sigma}}L_T^2}\lesssim \|a\|_{X^1_T}$ as long as $\sigma>\frac{3}{4}$ (because this corresponds to when $\frac{\alpha}{2\sigma-1}<\sigma-\frac{1}{2}$). Furthermore, clearly $\|D_x^{\alpha+\epsilon}g\|_{L_T^2L_x^{\infty}}\lesssim \|g\|_{Z}$. Hence, 
\begin{equation}
\|D_x^{\alpha}\tilde{P}_j(g\partial_xa)\|_{L_T^2L_x^2}\lesssim 2^{-j\epsilon}\|g\|_Z\|a\|_{X_T^1}.
\end{equation}
It is easy to see that a similar analysis works for $K_2^j$. Hence, we ultimately deduce that
\begin{equation}
K_1^j+K_2^j\lesssim 2^{-j\epsilon}T^{1-\sigma}\|g\|_{Z}\|a\|_{X^1_T}\|v\|_{X^0_T}.   
\end{equation}
Square summing now gives 
\begin{equation}
\left(\sum_{j>0}\|\langle D_x\rangle^{\sigma-1}P_j(g\partial_xav)\|_{L_T^1L_x^2}^2\right)^{\frac{1}{2}}\lesssim T^{1-\sigma}\|g\|_{Z}\|a\|_{X^1_T}\|v\|_{X^0_T}.
\end{equation}
	\end{proof}
	Next, we turn to the energy type $L^{\infty}_TL^2_x$ estimate in \Cref{LEE}. First, it is straightforward to verify by a simple energy estimate that  $P_{\leq 0}v$ is controlled in $L_T^{\infty}L_x^2$ by the right hand side of (\ref{Yestimate4}). Hence, let us restrict to controlling $P_{>0}v$.
	\begin{proof} 
		Let $j>0$. Projecting \eqref{lineq} onto frequency $2^j$, multiplying by $-i\overline{P_jv}$, taking real part and integrating from $0$ to $T$ gives
		\begin{equation}
			\begin{split}
				\|P_jv\|_{L_T^{\infty}L_x^2}^2&\lesssim  \|P_jv_0\|_{L_x^2}^2+\int_{0}^{T}\left\lvert\int_{\mathbb{R}}P_j(g\partial_xav)\overline{v_j}+P_j(\overline{g}\partial_xa\overline{v})\overline{v_j}\right\rvert+\int_{0}^{T}\left\lvert\int_{\mathbb{R}}P_j(|w|^{2\sigma}\partial_xv)\overline{v_j}\right\rvert
				\\
	&:=\|P_jv_0\|_{L_x^2}^2+I_1^j+I_2^j.
			\end{split}
		\end{equation}
\textbf{Estimate for $I_1^j$}
		\\
		\\
For simplicity, we show how to deal with the first term,
\begin{equation}
\int_{\mathbb{R}}P_j(g\partial_xav)\overline{v}_j
\end{equation}
as the other term (involving the complex conjugate of $gv$) is essentially identical.
\\
\\
We have by the Littlewood-Paley trichotomy,
		\begin{equation}
			\begin{split}	\int_{\mathbb{R}}P_j(g\partial_xav)\overline{v_j}=\int_{\mathbb{R}}P_j(P_{\geq j-4}(g\partial_xa)v)\overline{v_j}+\int_{\mathbb{R}}\tilde{P}_{< j}(g\partial_xa)\tilde{P}_jv\overline{\tilde{P}_jv}.
			\end{split}    
		\end{equation}
We expand the first term as
\begin{equation}\label{twotermsagain}
P_j(P_{\geq j-4}(g\partial_xa)v)=P_j(\tilde{P}_j(g\partial_xa)\tilde{P}_{<j}v)+\sum_{k\geq j}P_j(\tilde{P}_k(g\partial_xa)\tilde{P}_kv).
\end{equation}
We obtain by Bernstein's inequality, H\"older and a simple application of the Littlewood-Paley trichotomy, 
		\begin{equation}
			\begin{split}	&\int_{0}^{T}\left\lvert\int_{\mathbb{R}}P_j(\tilde{P}_j(g\partial_xa)\tilde{P}_{<j}v)\overline{v_j}\right\rvert
			\\
				&\lesssim \|\tilde{P}_jD_x^{\frac{3}{4\sigma}-\frac{1}{2}+\epsilon}(g\partial_xa)\|_{L_x^{\frac{2}{2\sigma-1}}L_T^2}\|\tilde{P}_jD_x^{\frac{3}{4\sigma}-\frac{1}{2}}v\|_{L_x^{\infty}L_T^2}^{2\sigma-1}\|\tilde{P}_jv\|_{L_x^2L_T^2}^{2(1-\sigma)}\|\tilde{P}_{<j}\langle D_x\rangle^{\sigma-\frac{3}{2}-\epsilon}v\|_{L_x^2L_T^{\infty}}
				\\
				&\lesssim 2^{-j\epsilon}T^{1-\sigma}\|\tilde{P}_jD_x^{\frac{3}{4\sigma}-\frac{1}{2}+2\epsilon}(g\partial_xa)\|_{L_x^{\frac{2}{2\sigma-1}}L_T^2}\|v\|_{X_T^0}^2
				\\
				&\lesssim 2^{-j\epsilon}T^{1-\sigma}(\|D_x^{\frac{3}{4\sigma}-\frac{1}{2}+3\epsilon}g\|_{L_x^{\infty}L_T^{\infty}}\|\partial_xa\|_{L_x^{\frac{2}{2\sigma-1}}L_T^2}+\|g\|_{L_x^{\frac{2}{2\sigma-1}}L_T^{\infty}}\|D_x^{\frac{3}{4\sigma}-\frac{1}{2}+2\epsilon}\partial_xa\|_{L_x^{\infty}L_T^2})\|v\|_{X_T^0}^2
				\\
				&\lesssim 2^{-j\epsilon}T^{1-\sigma}\|g\|_{Z}\|a\|_{X_T^1}\|v\|_{X_T^0}^2
			\end{split}    
		\end{equation}
where in the last line, we used the assumption $\sigma>\frac{\sqrt{3}}{2}$. The second term in (\ref{twotermsagain}) is similarly estimated by $2^{-j\epsilon}T^{1-\sigma}\|g\|_{Z}\|a\|_{X_T^1}\|v\|_{X_T^0}^2$.  Hence,
\begin{equation}
\|P_j(P_{\geq j-4}(g\partial_xa)v)\overline{v_j}\|_{L_T^1L_x^1}\lesssim 2^{-j\epsilon}T^{1-\sigma}\|g\|_Z\|a\|_{X^1_T}\|v\|_{X^0_T}^2.
\end{equation}
		For the remaining term, we have 
		\begin{equation}\label{Leib}
			\begin{split}
				-g\partial_xa&=gD_xHa
			\\	&=D_x^{\frac{3}{2}-\sigma+\epsilon}(gD_x^{\sigma-\frac{1}{2}-\epsilon}Ha)-D_x^{\frac{3}{2}-\sigma+\epsilon}gD_x^{\sigma-\frac{1}{2}-\epsilon}Ha
				\\
				&-D_x^{\frac{3}{2}-\sigma+\epsilon}(gD_x^{\sigma-\frac{1}{2}-\epsilon}Ha)+D_x^{\frac{3}{2}-\sigma+\epsilon}gD_x^{\sigma-\frac{1}{2}-\epsilon}Ha+gD_xHa.
			\end{split}
		\end{equation}
		Now, we estimate each term, thinking of the second line as a single term for which we will apply fractional Leibniz. For the first term in (\ref{Leib}), we have by H\"older and Bernstein inequalities, 
		\begin{equation}
			\begin{split}
				\int_{0}^{T}\left\lvert\int_{\mathbb{R}}\tilde{P}_{<j}D_x^{\frac{3}{2}-\sigma+\epsilon}(gD_x^{\sigma-\frac{1}{2}-\epsilon}Ha)\tilde{P}_j\overline{v}\tilde{P}_jv\right\rvert&\lesssim \|\tilde{P}_{<j}D_x^{\frac{3}{2}-\sigma+\epsilon}(gD_x^{\sigma-\frac{1}{2}-\epsilon}Ha)|\tilde{P}_jv|^{2(1-\sigma)}\|_{L_x^1L_T^{\frac{1}{1-\sigma}}}\|\tilde{P}_jv\|_{L_x^{\infty}L_T^2}^{2\sigma}
				\\
				&\lesssim \|\tilde{P}_{<j}D_x^{\frac{3}{2}-\sigma+\epsilon}(gD_x^{\sigma-\frac{1}{2}-\epsilon}Ha)\|_{L_x^{\frac{1}{\sigma}}L_T^{\infty}}\|\tilde{P}_{j}v\|_{L_x^2L_T^2}^{2(1-\sigma)}\|\tilde{P}_{j}v\|_{L_x^{\infty}L_T^2}^{2\sigma}
				\\
				&\lesssim T^{1-\sigma}\|\tilde{P}_{<j}(gD_x^{\sigma-\frac{1}{2}-\epsilon}Ha)\|_{L_x^{\frac{1}{\sigma}}L_T^{\infty}}\|\tilde{P}_{j}v\|_{X^0_T}^2
				\\
				&\lesssim T^{1-\sigma}\|g\|_{L_x^{\frac{2}{2\sigma-1}}L_T^\infty}\|D_x^{\sigma-\frac{1}{2}-\epsilon}Ha\|_{L_x^2L_T^{\infty}}\|\tilde{P}_jv\|_{X^0_T}^2,
			\end{split}
		\end{equation}
		where going from the second to the third line uses the fact that $\sigma>\frac{\sqrt{3}}{2}$.
		\\
		\\
		Next, we estimate the second term in (\ref{Leib}),
		\begin{equation}
			\begin{split}
				\int_{0}^{T}\left\lvert\int_{\mathbb{R}}\tilde{P}_{<j}(D_x^{\frac{3}{2}-\sigma+\epsilon}gD_x^{\sigma-\frac{1}{2}-\epsilon}Ha)\tilde{P}_j\overline{v}\tilde{P}_jv\right\rvert&\lesssim \|\tilde{P}_jv\|_{L_T^{\infty}L_x^2}^2\|D_x^{\frac{3}{2}-\sigma+\epsilon}g\|_{L_T^2L_x^{\infty}}\|D_x^{\sigma-\frac{1}{2}-\epsilon}Ha\|_{L_T^2L_x^{\infty}}
				\\
				&\lesssim T^{\frac{1}{2}}\|g\|_Z\|a\|_{X^1_T}\|\tilde{P}_jv\|_{L_T^{\infty}L_x^2}^2.
			\end{split}
		\end{equation}
		Using Sobolev embedding and the fractional Leibniz rule, the third term is estimated analogously to the second term. 
		\\
		\\
	Combining the estimates and square summing then shows 
		\begin{equation}
		\|I_1^j\|_{l^1_j(\mathbb{N})}\lesssim T^{1-\sigma}\|g\|_Z\|a\|_{X_T^1}\|v\|_{X_T^0}^2.
		\end{equation}
\textbf{Estimate for $I_2^j$}.
A similar argument to \Cref{NLEE} shows that
\begin{equation}\label{auxest}
\|I_2^j\|_{l_j^1(\mathbb{N})}\lesssim T^{1-\sigma}\||w|^{2\sigma-1}\|_Z\|w\|_{X^1_T}\|v\|_{X^0_T}^2.
\end{equation}
We now use the fact that for $\sigma>\frac{\sqrt{3}}{2}$, we have
\begin{equation}\label{Zest}
\||w|^{2\sigma-1}\|_Z\lesssim \|w\|_{X^1_T}^{2\sigma-1}.
\end{equation}
To see  (\ref{Zest}), first note that the $L_x^{\frac{2}{2\sigma-1}}L_T^{\infty}$ component is controlled by
\begin{equation}
\||w|^{2\sigma-1}\|_{L_x^{\frac{2}{2\sigma-1}}L_T^{\infty}}\lesssim \|w\|_{L_T^2L_x^{\infty}}^{2\sigma-1}\lesssim \|w\|_{X_T^1}^{2\sigma-1}.    
\end{equation}
For the $L_T^{\infty}W_x^{\frac{3}{4\sigma}-\frac{1}{2}+\epsilon,\infty}$ component, we have by \Cref{Holder}, Sobolev embedding, and the fact that $\frac{(\frac{3}{4\sigma}-\frac{1}{2})}{2\sigma-1}<\frac{1}{2}$, 
\begin{equation}
\|D_x^{\frac{3}{4\sigma}-\frac{1}{2}+\epsilon}|w|^{2\sigma-1}\|_{L_T^{\infty}L_x^{\infty}}\lesssim \|w\|_{L_T^{\infty}H_x^1}^{2\sigma-1}\lesssim \|w\|_{X_T^1}^{2\sigma-1}.
\end{equation}
This easily gives
\begin{equation}
\||w|^{2\sigma-1}\|_{L_T^{\infty}W_x^{\frac{3}{4\sigma}-\frac{1}{2}+\epsilon,\infty}}\lesssim \|w\|_{X_T^1}^{2\sigma-1}.    
\end{equation}
Finally, for the $L_T^4W_x^{\frac{3}{2}-\sigma+\epsilon,\infty}$ component, we have by \Cref{Holder} and the fact that $\frac{\frac{3}{2}-\sigma}{2\sigma-1}<\sigma$,
\begin{equation}
\|D_x^{\frac{3}{2}-\sigma+\epsilon}|w|^{2\sigma-1}\|_{L_T^4L_x^{\infty}}\lesssim \|w\|_{L_T^4W_x^{\sigma,\infty}}^{2\sigma-1}\lesssim \|w\|_{X_T^1}^{2\sigma-1}    
\end{equation}
which clearly gives
\begin{equation}
\||w|^{2\sigma-1}\|_{L_T^4W_x^{\frac{3}{2}-\sigma+\epsilon}}\lesssim \|w\|_{X_T^1}^{2\sigma-1}.    
\end{equation}
Combining the above three estimates gives (\ref{Zest}).
\\
\\
Combining (\ref{Zest}) and (\ref{auxest}) gives
\begin{equation}
\|I_2^j\|_{l_j^1(\mathbb{N})}\lesssim T^{1-\sigma}\|w\|_{X^1_T}^{2\sigma}\|v\|_{X^0_T}^2.
\end{equation}
Combining the above estimates for $I_j^1$ and $I_j^2$ completes the proof of \Cref{LEE}. 
\\
\\
		\textbf{Proof of \Cref{linbounds}}.
Now we complete the proof of \Cref{linbounds}.
\begin{proof}
Combining \Cref{Yestimate2} and \Cref{LEE} with an argument similar to what was done in \Cref{Ubounds} gives for $T\sim 1$ and $\|g\|_{Z}, \|w\|_{X^1_T}, \|a\|_{X^1_T}\ll 1$, 
		\begin{equation}
		\begin{split}
			\|v\|_{X^0_T}&\lesssim \|v_0\|_{L_x^2}.
			\end{split}
		\end{equation}
\end{proof}
\section{Well-posedness at low regularity}\label{Section4}
In this section, we aim to prove local well-posedness in $H_x^s$ for $s\in [1,\frac{3}{2}]$ and $\sigma>\frac{\sqrt{3}}{2}$ assuming the conclusion of \Cref{low reg theorem} when $\frac{3}{2}<s<4\sigma$, which will be justified in a later section when we prove high-regularity estimates. Given the estimates established in the previous section, the scheme to prove well-posedness is relatively standard. We essentially follow the approach of \cite{MR2955206}. See also the recent preprint \cite{primer} for a more detailed overview. 
\subsection{Frequency envelope bounds}
\begin{proposition}\label{envbounds}
Let  $\frac{\sqrt{3}}{2}<\sigma<1$ and let $u$ be as in \Cref{Ubounds}. If $a_j$ is an admissible frequency envelope for $u_0$ in $H_x^s$, then $a_j$ is an admissible frequency envelope for $u$ in $X^s_T$.
\end{proposition}
Indeed, let $b_j$ be a $X^s_T$ frequency envelope for the solution $u$. Obviously $b_0\lesssim a_0$, so let us consider $j>0$. By \Cref{Ubounds} a), we have 
\begin{equation}
\begin{split}
\|P_ju\|_{X^s_T}&\lesssim_T a_j\|u_0\|_{H^s_x}
				+T^{\frac{1}{2}}b_j(1+\|u\|_{S^1_T}^{4\sigma})\|u\|_{X^s_T}+T^{\frac{1-\sigma}{2}}b_j\|u\|_{X^1_T}^{\sigma}\|u\|_{X^s_T}.
\end{split}
\end{equation}
Hence, by definition we have 
\begin{equation}
\begin{split}
b_j&\lesssim a_j(1 +\|u_0\|_{H_x^s}\|u\|_{X^s_T}^{-1})+T^{\frac{1-\sigma}{2}}b_j\|u\|_{X^1_T}^{\sigma}+T^{\frac{1}{2}}b_j(1+\|u\|_{X^1_T}^{4\sigma}).
\end{split}
\end{equation}
For $T$ small enough, it follows from \Cref{Ubounds} that
\begin{equation}
b_j\lesssim a_j.
\end{equation}
Iterating this procedure $O(T^{-1})$ many times shows that this is true for $T\lesssim 1$. This completes the proof.
\subsection{Existence of $H^s$ solutions}
Now, we construct local $H^s$ solutions to \eqref{gDNLS} for $1\leq s\leq \frac{3}{2}$ as limits of more regular solutions. 
\\
\\
Indeed, let $u_0\in H^s$. Let $u^{(n)}$ be the globally well-posed $C_{loc}(\mathbb{R};H_x^3)$ solution (to be constructed in a later section) to the equation, 
		\begin{equation}
			\begin{cases}
				&(i\partial_t+\partial_x^2)u^{(n)}=i|u^{(n)}|^{2\sigma}\partial_xu^{(n)},
				\\
				&u^{(n)}_0=P_{<n}u_0.
			\end{cases}
		\end{equation}
Let $n>m$. We see that $v^{(m,n)}:=u^{(n)}-u^{(m)}$ satisfies the equation 
		\begin{equation}
			\begin{cases}
				&(i\partial_t+\partial_x^2)v^{(m,n)}=i|u^{(n)}|^{2\sigma}\partial_xv^{(m,n)}+iG^{(n,m)}\partial_xu^{(m)}v^{(m,n)},
				\\
				&v^{(m,n)}(0)=P_{m\leq\cdot<n}u_0,
			\end{cases}
		\end{equation}
	where
	\begin{equation}
	  G^{(n,m)}:=\frac{(|u^{(n)}|^{2\sigma}-|u^{(m)}|^{2\sigma})}{u^{(n)}-u^{(m)}}.
	\end{equation}
Using \Cref{Holder}, Sobolev embedding, the fact that $\sigma>\frac{\sqrt{3}}{2}$ and \Cref{Ubounds}, one easily verifies that $G^{(n,m)}$ satisfies the conditions of \Cref{linbounds} with $\|G^{(n,m)}\|_Z\lesssim_{\|u_0\|_{H_x^s}} 1$ (with the implicit constant independent of $n$ and $m$). One likewise checks using \Cref{Ubounds} that $u^{(n)}$ satisfies $\|u^{(n)}\|_{X_T^1}\lesssim_{\|u_0\|_{H_x^s}} 1$ uniformly in $n$. Hence, by \Cref{linbounds}, we obtain for $T$ small enough (depending on the size of the $H_x^s$ norm of $u_0$), 
\begin{equation}
\|v^{(m,n)}\|_{X^0_T}\lesssim \|P_{m\leq \cdot<n}u_0\|_{L^2_x}.
\end{equation}
Hence, $u^{(n)}$ is Cauchy in $X^0_T$ and thus converges to some $u\in X^0_T$. We show that in fact $u^{(n)}\to u$ in $X_T^s$.
\\
\\
To see this, we let $a_j^n$ and $a_j$ be admissable frequency envelopes for $P_{<n}u_0$ and $u_0$ respectively, in $H_x^s$. Clearly $(a_j^n)\to (a_j)$ in $l_j^2(\mathbb{N}_0)$ as $n\to \infty$. Now let $\epsilon>0$. Then thanks to \Cref{envbounds}, we have 
\begin{equation}
\|P_{>j}u^{(n)}\|_{X_T^s}\lesssim \|(a_j^n)_{N>j}\|_{l^2_N(\mathbb{N})}\|u_0\|_{H_x^s}.
\end{equation}
Hence, for $n\geq n_0(\epsilon)$ large enough, we obtain the bound,
\begin{equation}
\|P_{>j}u^{(n)}\|_{X_T^s}\lesssim (\epsilon+\|(a_j)_{N>j}\|_{l^2_N(\mathbb{N})})\|u_0\|_{H_x^s}
\end{equation}
where the implicit constant is independent of $j$ and $n$. Hence, there is $j=j(\epsilon)$ such that for every $n>n_0$, we have
\begin{equation}
\|P_{>j}u^{(n)}\|_{X_T^s}\lesssim \epsilon.
\end{equation}
On the other hand, since $u^{(n)}$ converges in $X_T^0$, it follows that for $m,n>n_0$ large enough that
\begin{equation}
\|u^{(n)}-u^{(m)}\|_{X_T^s}\lesssim 2^{js}\|u^{(n)}-u^{(m)}\|_{X_T^0}+\|P_{\geq j}u^{(n)}\|_{X_T^s}+\|P_{\geq j}u^{(m)}\|_{X_T^s}\lesssim \epsilon.
\end{equation}
Hence, $u^{(n)}$ is Cauchy in $X_T^s$ and thus converges to $u$. It is clear at this regularity that $u$ solves the equation \eqref{gDNLS} in the sense of distributions. This shows existence. 
		\subsection{Uniqueness and Lipschitz dependence in $X^0$}
		Here, we aim to show that solutions in $X^1_T$ (and thus, also in $X^s_T$ for $s>1$) are unique and that they satisfy a weak Lipschitz type bound in $X_T^0$. For this, consider the difference of two solutions $u^1$ and $u^2$, $v:=u^1-u^2$. We see that $v$ solves the equation,
		\begin{equation}
			\begin{cases}
				&(i\partial_t+\partial_x^2)v=i|u^1|^{2\sigma}\partial_xv+iG\partial_xu^2v,
				\\
				&v(0)=u^1(0)-u^2(0),
			\end{cases}
		\end{equation}
		where 
		\begin{equation}
			G=\frac{|u^1|^{2\sigma}-|u^2|^{2\sigma}}{u^1-u^2}.
		\end{equation}
		We see that  \Cref{linbounds} applies, and we obtain the weak Lipschitz bound
		\begin{equation}
			\|u^1-u^2\|_{X^0_T}\lesssim \|u^1(0)-u^2(0)\|_{L_x^2}.    
		\end{equation}
		In particular, this shows uniqueness.
		\subsection{Continuous dependence in $H^s$}
Here, we aim to show that the solution map is continuous in $H^s$. Specifically, we show that for each $R>0$, there is $T=T(R)>0$ such that the solution map from $\{u_0: \|u_0\|_{H^s}<R\}$ to the corresponding $X^s_T$ space is continuous. By rescaling the data and restricting to small enough time, we may assume without loss of generality that the conditions of \Cref{envbounds} are satisfied.
\\
\\
	Now, let $u^{(n)}_0$ be a sequence in $H_x^s$ converging to $u_0$ in $H_x^s$. Let $a_j$ and $a_j^{(n)}$ be the associated frequency envelopes for $u_0$ and $u_0^{(n)}$ given by (\ref{Henv}).  We have $(a_j^{(n)})\to(a_j)$ in $l^2$. Now, let $\epsilon>0$. Let $N=N(\epsilon)$ be such that $\|a^{(n)}_{j>N}\|_{l_j^2}\lesssim \epsilon$. Using \Cref{envbounds}, we have $\|P_{>N}u^{(n)}\|_{X^s_T}\lesssim \epsilon$ for all $n$. On the other hand, using the Lipschitz dependence at low frequency, we have
	\begin{equation}
		\|P_{<N}(u^{(n)}-u)\|_{X^s_T}\lesssim 2^{sN}\|u^{(n)}_0-u_0\|_{L^2}.
	\end{equation}
Now, for $n(N)$ large enough, we have
\begin{equation}
\|P_{<N}u^{(n)}-P_{<N}u\|_{X^s_T}\lesssim \epsilon.
\end{equation}
Hence, for such $n$, we have 
\begin{equation}
\|u^{(n)}-u\|_{X^s_T}\lesssim \|P_{<N}(u^{(n)}-u)\|_{X^s_T}+\|P_{\geq N}u^{(n)}\|_{X^s_T}+\|P_{\geq N}u\|_{X^s_T}\lesssim \epsilon.
\end{equation}
It follows that
\begin{equation}
\limsup_{n\to\infty}\|u^{(n)}-u\|_{X^s_T}\lesssim \epsilon.
\end{equation}
Taking $\epsilon\to 0$ then yields
\begin{equation}
\lim_{n\to\infty}\|u^{(n)}-u\|_{X^s_T}=0
\end{equation}
as desired. This completes the proof of continuous dependence and also concludes the local well-posedness portion of the proof of \Cref{low reg theorem} when $s\leq \frac{3}{2}$.
	\end{proof}
	\subsection{Further discussion of the proofs}
	We now provide a brief discussion on how one can, in principle, go below the $H_x^1$ well-posedness threshold, as well as justify some of the choices made in the proof. 
	\\
	\\
		 It is instructive to discuss a version of this gauge transformation method which was successfully implemented in Tao's article \cite{tao2004global} which established local well-posedness of the Benjamin-Ono equation,
 \begin{equation}
 \begin{cases}
&u_t+Hu_{xx}=uu_x,
\\
&u(0)=u_0,
 \end{cases}
 \end{equation}
 in $H_x^1$. The idea in Tao's paper was to do a type of gauge transformation by defining essentially,
 \begin{equation}
w=P_{+hi}(e^{-iF})
 \end{equation}
 where $F(t,x)$ is a suitable spatial primitive of $u(t,x)$ and $P_{+hi}$ is a projection onto large positive frequencies. Then one proves a priori $H_x^2$ estimates for $w$ (which can be translated into $H_x^1$ estimates for $u$). While the coefficient $u$ in the nonlinearity in Benjamin-Ono is only of linear order (and so one might at first na\"ively suspect that this equation behaves similarly to \eqref{gDNLS} when $\sigma=\frac{1}{2}$), the spatial primitive $F$ still essentially solves a linear Schr\"odinger equation (up to a perturbative error). A refinement of this gauge transformation idea appeared in \cite{ifrim2017well} in which $L_x^2$ well-posedness (among other results) for Benjamin-Ono was proven. Loosely speaking, in this latter paper, the authors performed a gauge transformation on each frequency scale to remove the leading order paradifferential part of the nonlinearity and then performed a quadratic normal form correction to remove the milder terms in the nonlinearity. Our so-called partial gauge transformation is more analogous to what was done in that paper. Specifically, the analogue of $F$ in our proof is essentially the family of functions $\Phi_j$ as defined in (\ref{gauge}), which in addition to the frequency localization scale, takes into account the pointwise size of $u$ relative to the frequency scale. However, in our case, there is no obvious cancellation arising in the term $(i\partial_t\Phi_j+\partial_x^2\Phi_j)$, which forces us to estimate each term $\partial_t\Phi_j$ and $\partial_x^2\Phi_j$ separately. This is one of the major sources for the losses in our low regularity estimates.
 \\
 \\
This issue actually also adds technical difficulty when trying to lower the local well-posedness threshold below $H_x^1$. For instance, when estimating $\partial_t\Phi_j$ in \Cref{Ubounds}, there are expressions essentially of the form
\begin{equation}
P_{<j}\partial_x^{-1}(gv_1v_2)    
\end{equation}
that we bound in  $L_T^{1}L_x^{\infty}$, where $g$ is some bounded function and $v_1$ and $v_2$ are linear expressions in $u_x$ or $\overline{u_x}$. Unfortunately, in these expressions, it doesn't seem that typically the output frequency of the product $gv_1v_2$ is comparable to the frequencies of the individual terms $v_1$ and $v_2$, and so the $\partial_x^{-1}$ can't be ``distributed" amongst these factors to obtain expressions with lower order derivatives in place of $u_x$. One workaround to this issue could be to place any factors of $u_x$ arising in such an expression in an appropriate maximal function/smoothing space as in \Cref{Inhom Max}. Proceeding this way will likely lead to losses worse than the $1-\sigma$ derivatives already observed in the current low regularity estimates. However, this should work in principle to lower the well-posedness threshold below $H_x^1$ when $\sigma$ is close to $1$. We decided not to do this for the sake of simplicity, as our preliminary calculations suggested that the dependence of the well-posedness threshold on $\sigma$ would be rather complicated when $s<1$, at least without introducing some new tools.
\section{High regularity estimates}
In this section, we aim to prove a priori $H_x^{2s}$-type bounds for a global solution $u$ to a family of regularizations of \eqref{gDNLS},
\begin{equation}\label{gdnlsreg1}
\begin{cases}
&i\partial_tu+\partial_x^2u=i\eta P_{<k}|v|^{2\sigma}u_x,
\\
&u(0)=P_{<k}u_0,
\end{cases}
\end{equation}
where $k\in\mathbb{N}$, $v\in C^2(\mathbb{R};H_x^{\infty})$, $2s$ is in the range $2-\sigma<2s<4\sigma$, $\eta$ is a suitable time-dependent cutoff function which is equal to $1$ on the unit time interval $[-1,1]$ and supported within $(-2,2)$, and $u_0\in H_{x}^{2s}$ has sufficiently small norm. The key difficulty here is to obtain estimates independent of the regularization parameter $k$. As mentioned earlier, this is somewhat subtle because the nonlinearity is too rough to directly obtain an energy estimate by simply applying $D_x^{2s}$ to the equation. Our overarching idea, morally, is to instead obtain suitable estimates for time derivatives, $D_t^su$, of order $s<2\sigma$ for solutions to (\ref{gdnlsreg1}). This is one of the key technical reasons for truncating the nonlinearity with the time-dependent cutoff $\eta$ and working with global in-time solutions to (\ref{gdnlsreg1}). For small enough data, one expects to be able to construct a solution $u$ to this equation on the time interval $[-2,2]$, and then extend it to a global solution using the fact that $u$ should solve the linear Schr\"odinger equation for $|t|>2$. The idea of truncating the nonlinearity with a time-dependent cutoff in order to obtain global in time solutions (to facilitate use of Fourier analysis in the time variable) is not a new idea. See for instance, \cite{MR1209299} and \cite{MR1215780}. 
\\
\\
Before outlining our strategy in more detail, we give an overview of the functional setting and relevant notation for this problem. 
\subsection{Function spaces and notation} Here, we fix some basic notation and describe the function spaces used in our construction of solutions at high regularity. 
\\
\\
We will use $S_k$, $S_{<k}$ and $S_{\geq k}$ to denote the temporal variants of the spatial Littlewood-Paley projections $P_k, P_{<k}$ and $P_{\geq k}$ as defined in \Cref{LWP}. We write $\phi(2^{-j}\xi)$ to denote the spatial Fourier multiplier for $P_j$ and $\psi(2^{-k}\tau)$ to denote the temporal Fourier multiplier for $S_k$.
\\
\\
We will also need to sometimes distinguish between a compact time interval and the whole space in our estimates. For this purpose, let us denote for a Banach space $X$, $L^p_tX:=L^p(\mathbb{R}; X)$ (that is, we use a lowercase $t$ to emphasize when the underlying time interval is $\mathbb{R}$). For $T>0$, we  use $L_T^pX:= L^p([-T,T];X)$ when we want to emphasize that the time interval is compact.
\\
\\
Next, for the range of $2s\in (2-\sigma,4\sigma)$ we are considering, the smoothing and maximal function type norms from the low regularity estimates are not needed. We modify our function spaces accordingly and only use standard $L_x^2$ based Sobolev spaces and standard Strichartz spaces (see below). Since both spatial and temporal regularity will be relevant in our analysis, we make the convention from here on that a real number $s$ will correspond to the Sobolev regularity of a function in the time variable. In light of the scaling of the linear Schr\"odinger equation, it is natural to use $2s$ to denote the corresponding spatial regularity. With this in mind, for $s\geq 0$ and $T>0$, we denote the relevant Strichartz type space by $\mathcal{S}_T^{2s}:=L^4_TW_x^{2s,\infty}\cap L_T^{\infty}H_x^{2s}.$ We also define the energy type space $\mathcal{X}_T^{2s}$ by the norm,
\begin{equation}
\|u\|_{\mathcal{X}_T^{2s}}:=\|P_{\leq 0}u\|_{L_T^{\infty}H_x^{2s}}+\left(\sum_{j>0}\|P_ju\|_{L_T^{\infty}H_x^{2s}}^2\right)^{\frac{1}{2}}.
\end{equation}
Clearly this controls the $C([-T,T];H_x^{2s})$ norm.  The reason we opt for this slightly stronger norm (as opposed to just $\|u\|_{L_T^{\infty}H_x^{2s}}$) is because it will be slightly more convenient for proving frequency envelope bounds. Furthermore, we have the trivial embedding 
\begin{equation}
X_T^{2s}\subseteq \mathcal{X}_T^{2s}.
\end{equation}

Finally, since estimates for time derivatives will play a key role in our analysis, it will also be convenient to introduce the auxiliary norm
\begin{equation}
\|u\|_{Z^{s}_{p,q}}:=\|\langle D_t\rangle^su\|_{L_t^pL_x^q}+\|\langle D_x\rangle^{2s}u\|_{L_t^pL_x^q}.
\end{equation}
When $q=2$, we will simply abbreviate this by $Z_p^s$.
\\
\\
The reader should keep in mind that although we will often time-localize $u$ (or the nonlinearity) to be compactly supported in time, some mild care must be taken in the estimates when nonlocal operators such as $D_t^s$ are involved. This is especially relevant when comparing $L_tX$ and $L_TX$ type norms. 
\subsection{A frequency localized $H_x^{2s}$ bound}
The key result for this section is the following frequency localized $H_x^{2s}$ a priori bound for
(\ref{gdnlsreg1}). 
\begin{proposition} \label{keyestimate1}
Let $2-\sigma<2s<4\sigma$, $T=2$ and $u_0\in H_x^{2s}$. Suppose that $u\in C^2(\mathbb{R};H_x^{\infty})$ solves (\ref{gdnlsreg1}). Furthermore, let $a_j$ be a $H_x^{2s}$ frequency envelope for $u_0$ and let $b_j^1$ and $b_j^2$ be $\mathcal{X}_T^{2s}$ frequency envelopes for $u$ and $v$, respectively. Let $b_j:=\max\{b_j^1,b_j^2\}$. Furthermore, let $0<\epsilon\ll 1$ and assume that for each $0<\delta\ll 1$
\begin{equation}\label{vassumption}
\|v\|_{\mathcal{S}_T^{1+\delta}}+\|(i\partial_t+\partial_x^2)v\|_{Z_{\infty}^{s-1+\delta}\cap \mathcal{S}_T^{\delta}}\lesssim_{\delta} \epsilon.    
\end{equation}
Then $P_ju$ satisfies the estimate,
\begin{equation}\label{keybound}
\|P_ju\|_{\mathcal{X}_T^{2s}}^2\lesssim a_j^2\|u_0\|_{H_x^{2s}}^2+b_j^2\epsilon^{2\sigma}(\|u\|_{\mathcal{X}_T^{2s}}^2+\|u\|_{\mathcal{S}_T^{1}}^2)+b_j^2\epsilon^{2\sigma-1}\|u\|_{\mathcal{S}_T^1}\|u\|_{\mathcal{X}_T^{2s}}\|v\|_{\mathcal{X}_T^{2s}}+b_j^2\epsilon^{4\sigma-2}\|u\|_{\mathcal{S}_T^{1}}^2\|v\|_{\mathcal{X}_T^{2s}}^2.
\end{equation}
Furthermore, by square summing, we also have
\begin{equation}
\|u\|_{\mathcal{X}_T^{2s}}^2\lesssim \|u_0\|_{H_x^{2s}}^2+\epsilon^{2\sigma}(\|u\|_{\mathcal{X}_T^{2s}}^2+\|u\|_{\mathcal{S}_T^{1}}^2)+\epsilon^{2\sigma-1}\|u\|_{\mathcal{S}_T^1}\|u\|_{\mathcal{X}_T^{2s}}\|v\|_{\mathcal{X}_T^{2s}}+\epsilon^{4\sigma-2}\|u\|_{\mathcal{S}_T^{1}}^2\|v\|_{\mathcal{X}_T^{2s}}^2.    
\end{equation}
\end{proposition}
\begin{remark}
Crucially, it should be noted that the implied constant in the bound above does not depend on the regularization parameter $k$.
\end{remark}
\begin{remark}
The reader should carefully observe the restriction $T=2$ and not $T\leq 2$ in \Cref{keyestimate1}. This is because $\eta$ is localized in time to a unit scale. More work is required to show that we have suitable bounds for $T\leq 2$. This will be studied further in \Cref{Sec6}. 
\end{remark}
Next, we give a brief outline for how we will obtain such an estimate. As mentioned above, to minimize the number of derivatives which fall on the rough part of the inhomogeneous term, $|v|^{2\sigma}$, we will prove what is essentially an energy type estimate for $D_t^su$ instead of $D_x^{2s}u$ and use the bounds for $D_t^su$ to estimate $D_x^{2s}u$. This is consistent with the scaling symmetry of \eqref{gDNLS}. There is one technical caveat however. Namely, one expects to be able to convert estimates for $D_t^su$ to estimates for $D_x^{2s}u$ when the time frequency $\tau$ of a solution $u$ to (\ref{gdnlsreg1}) is close to $-\xi ^2$ where $\xi$ is the spatial frequency (i.e. in the so-called low modulation region). However, this is not guaranteed due to the presence of the inhomogeneous term in the equation. Therefore, we need a suitable way of controlling $D_x^{2s}u$ for the portion of $u$ which has space-time Fourier support far away from the characteristic hypersurface $\tau=-\xi^2$. In other words, we also need an estimate for $u$ in the so-called high modulation region.
\\
\\
With this in mind, we split our analysis into two parts. First, we prove an elliptic type estimate in the high modulation region for solutions to $(\ref{gdnlsreg1})$ which will allow us to suitably control $D_x^{2s}u$ in terms of the portion of $D_x^{2s}u$ localized near the characteristic hypersurface, as well as a lower order term stemming from the nonlinearity. To control $D_x^{2s}u$ in the low modulation region, we essentially obtain an energy type estimate for $D_t^su$ (the benefit being that we only have to differentiate the nonlinearity $s$ times in the time variable as opposed to $2s$ times in the spatial variable). When $u$ is localized near the characteristic hypersurface, this is precisely the regime in which we expect to be able to suitably control $D_x^{2s}u$ by $D_t^su$. \Cref{keyestimate1} will then follow from combining the low and high modulation analysis.
\subsection{The high modulation estimate}
We begin with the high modulation estimate, \Cref{spacetimeconv}. This will be useful for estimating the portion of a (time-localized) solution to (\ref{gdnlsreg1}) which has space-time Fourier support away from the characteristic hypersurface. This can also be thought of as an elliptic space-time estimate.
\begin{lemma}\label{spacetimeconv}
Let $u_0\in H_x^{\infty}$ and suppose $u\in C^1(\mathbb{R};H_x^{\infty})$ solves the equation,
\begin{equation}\label{gdnls5}
\begin{cases}
&(i\partial_t+\partial_x^2)u=f,
\\
&u(0)=u_0.
\end{cases}
\end{equation}
Let $0\leq s\leq 1$, $j,k>0$, $p,q\in [1,\infty]$ and suppose $|k-2j|>4$. Then $P_jS_ku$ satisfies the estimate, 
\begin{equation}
\|P_jS_k\langle D_x\rangle^{2s}u\|_{L_t^{p}L_x^q}+\|P_jS_k\langle D_t\rangle^{s}u\|_{L_t^{p}L_x^q}\lesssim \|\tilde{P}_j\tilde{S}_k\langle D_t\rangle^{s-1}f\|_{L_t^{p}L_x^q}.
\end{equation}
The result also holds for $k=0$, when $S_0$ is replaced by $S_{\leq 0}$. 
\end{lemma}
\begin{proof}
We prove the estimate for $\langle D_x\rangle^{2s}u$. The estimate for $\langle D_t\rangle^{s}u$ is similar. Notice that
\begin{equation}
\begin{split}
[\mathcal{F}_{t,x}(\langle D_x\rangle^{2s}S_kP_ju)](\tau,\xi)&=\langle \xi\rangle^{2s}\psi(2^{-k}\tau)\phi(2^{-j}\xi)[\mathcal{F}_{t,x}(\tilde{S}_k\tilde{P}_ju)](\tau,\xi)
\\
&=-\frac{\langle\xi\rangle^{2s}}{\tau+\xi^2}\psi(2^{-k}\tau)\phi(2^{-j}\xi)[\mathcal{F}_{t,x}\tilde{S}_k\tilde{P}_j(i\partial_t+\partial_x^2)u](\tau,\xi).
\end{split}
\end{equation}
Hence, by Young's inequality and (\ref{gdnls5}), we have (using that $\psi(2^{-k}\tau)\phi(2^{-j}\xi)$ is supported away from $\tau+\xi^2=0$),
\begin{equation}
\begin{split}
\|\langle D_x\rangle^{2s}S_kP_ju\|_{L_t^{p}L_x^q}&\lesssim \|\mathcal{F}_{t,x}^{-1}[\frac{\langle\xi\rangle^{2s}}{\tau+\xi^2}\psi(2^{-k}\tau)\phi(2^{-j}\xi)]\|_{L_t^1L_x^1}\|(i\partial_t+\partial_x^2)\tilde{S}_k\tilde{P}_ju\|_{L_t^{p}L_x^q} 
\\
&\lesssim \|\mathcal{F}_{t,x}^{-1}[\frac{\langle\xi\rangle^{2s}}{\tau+\xi^2}\psi(2^{-k}\tau)\phi(2^{-j}\xi)]\|_{L_t^1L_x^1}\|\tilde{S}_k\tilde{P}_jf\|_{L_t^{p}L_x^q}.
\end{split}
\end{equation}
It remains then to show that
\begin{equation}
\|\mathcal{F}_{t,x}^{-1}[\frac{\langle\xi\rangle^{2s}}{\tau+\xi^2}\psi(2^{-k}\tau)\phi(2^{-j}\xi)]\|_{L_t^1L_x^1}\lesssim 2^{-k(1-s)}.
\end{equation}
A simple change of variables shows that
\begin{equation}
\begin{split}
\|\mathcal{F}_{t,x}^{-1}[\frac{\langle\xi\rangle^{2s}}{\tau+\xi^2}\psi(2^{-k}\tau)\phi(2^{-j}\xi)]\|_{L_t^1L_x^1}&=\|\mathcal{F}_{t,x}^{-1}[\frac{\langle2^j\xi\rangle^{2s}}{2^k\tau+2^{2j}\xi^2}\psi(\tau)\phi(\xi)]\|_{L_t^1L_x^1}.
\end{split}
\end{equation}
Then we have
\begin{equation}
\frac{\langle 2^j\xi\rangle^{2s}}{2^k\tau+2^{2j}\xi^2}\psi(\tau)\phi(\xi)=2^{k(s-1)}\frac{(2^{-k}+2^{2j-k}\xi^2)^{s}}{\tau+2^{2j-k}\xi^2}\psi(\tau)\phi(\xi):=2^{k(s-1)}F_{j,k}(\tau,\xi).
\end{equation}
It is easy to see that for multi-indices $0\leq |\alpha|\leq 3$,
\begin{equation}
|\partial^{\alpha}_{\tau,\xi}F_{j,k}|\lesssim 1
\end{equation}
so that (since $\phi\psi$ is supported on $[-2,2]\times [-2,2]$)
\begin{equation}
\|\partial^{\alpha}_{\tau,\xi}F_{j,k}\|_{L_{\tau,\xi}^1}\lesssim 1
\end{equation}
with bound independent of $j$ and $k$. It follows that
\begin{equation}
\|\mathcal{F}_{t,x}^{-1}[\frac{\langle 2^j\xi\rangle^{2s}}{2^k\tau+2^{2j}\xi^2}\psi(\tau)\phi(\xi)]\|_{L_t^1L_x^1}\lesssim 2^{k(s-1)}\|(1+|x|+|t|)^{-3}\|_{L_t^1L_x^1}\lesssim 2^{k(s-1)}
\end{equation}
which is what we wanted to show. The case for $\langle D_t\rangle^su$ is similar.
\end{proof}
From this lemma, we obtain a very useful corollary which will allow us to control derivatives of $u$ in the high modulation region with convenience and reduce matters to proving a suitable low modulation bound.
\begin{corollary}\label{cor to spacetimeconv}
Let $u\in C^2(\mathbb{R};H_x^{\infty})$, and let the notation be as in \Cref{spacetimeconv}. Then for every $\delta>0$ and $j>0$, we have
\\
\\
a) If $0\leq s<1$,
\begin{equation}
\|P_j\langle D_x\rangle^{2s}u\|_{L_t^{p}L_x^q}+\|P_j\langle D_t\rangle^{s}u\|_{L_t^{p}L_x^q}\lesssim_{\delta} \|\tilde{S}_{2j}P_j\langle D_x\rangle^{2s}u\|_{L_t^{p}L_x^q}+\|\tilde{P}_j\langle D_t\rangle^{s-1+\delta}f\|_{L^p_tL^q_x}
\end{equation}
and
\\
\\
b) If $1\leq s<2\sigma$, 
\begin{equation}
\|P_j\langle D_x\rangle^{2s}u\|_{L_t^{p}L_x^2}+\|P_j\partial_t\langle D_t\rangle^{s-1}u\|_{L_t^{p}L_x^2}\lesssim_{\delta} \|\tilde{S}_{2j}P_j\langle D_x\rangle^{2s}u\|_{L_t^{p}L_x^2}+\|\tilde{P}_jf\|_{Z^{s-1+\delta}_{p,2}}
\end{equation}
where $\tilde{S}_{2j}=S_{[2j-4,2j+4]}$. 
\end{corollary}
\begin{proof}
For a), this follows from the Bernstein type estimate $\|D_x^{2s}\tilde{S}_{2j}P_ju\|_{L_t^{p}L_x^q}\sim \|D_t^{s}\tilde{S}_{2j}P_ju\|_{L_t^{p}L_x^q}$ and from \Cref{spacetimeconv} by summing over $k>0, |k-2j|>4$ (which is where the requirement of having $\delta>0$ comes in to play). Then b) follows from part a) with $u$ replaced by $\partial_tu$ and $s$ replaced by $s-1$, and then by expanding $\partial_tD_x^{2s-2}P_ju=i\partial_x^2D_x^{2s-2}P_ju-iD_x^{2s-2}P_jf$.
\end{proof}
\begin{remark}
We remark that in part b), if $f$ takes the form of $f=i\eta P_{<k}|u|^{2\sigma}u_x$ as in (\ref{gdnlsreg1}) then if $\delta$ is sufficiently small, we expect to be able to control the last term on the right as long as $2s-2<2\sigma$ which is satisfied automatically, because $2s<4\sigma<2\sigma+2$ in the range $\frac{1}{2}<\sigma<1$. If we were looking at the case $\sigma>1,$ this would present a new limiting threshold for which we expect to obtain estimates for $u$, c.f.~\cite{uchizono2012well}.
\end{remark}
In light of the above remark, one should observe at this point that the high modulation estimate above essentially reduces proving \Cref{keyestimate1} to obtaining an estimate for the $L_T^{\infty}H_x^{2s}$ norm of a solution $u$ to (\ref{gdnlsreg1}) in the low modulation region, as well as controlling an essentially perturbative source term stemming from the nonlinearity in (\ref{gdnlsreg1}). With this in mind, we now turn to the low modulation estimate, which is essentially the heart of the matter.
\subsection{Low modulation estimates}
Next we prove suitable bounds for the $L_T^{\infty}H_x^{2s}$ norm of a solution $u$ to (\ref{gdnlsreg1}) in the low modulation region. Specifically, we prove the following energy type bound to control the portion of $u$ which is localized near the characteristic hypersurface.
\begin{lemma}\label{Etime1}
Let $u_0\in H_x^{2s}$ and suppose that $u\in C^2(\mathbb{R};H_x^{\infty})$ solves (\ref{gdnlsreg1}). Let $T=2$, $2-\sigma<2s<4\sigma$, $a_j$ be an admissible $H_x^{2s}$ frequency envelope for $u_0$, and $b_j^1$, $b_j^2$ be  $\mathcal{X}_T^{2s}$ frequency envelopes for $u$ and $v$, respectively. Take $b_j:=\max\{b_j^1,b_j^2\}$. Let $0<\epsilon\ll 1$ and suppose $v$ satisfies the estimates,
\begin{equation}
\|v\|_{\mathcal{S}_T^{1+\delta}}+\|(i\partial_t+\partial_x^2)v\|_{Z_{\infty}^{s-1+\delta}\cap\mathcal{S}_T^{\delta}}\lesssim_{\delta}\epsilon    
\end{equation}
for each $0<\delta\ll 1$. Then for every $j\geq 0$,  we have 
\begin{equation}
\begin{split}
\|\tilde{S}_{2j}P_jD_x^{2s}u\|_{L_T^{\infty}L_x^2}^2&\lesssim_{\delta} a_j^2\|u_0\|_{H_x^{2s}}^2+b_j^2\epsilon^{2\sigma}(\|u\|_{\mathcal{X}_T^{2s}}^2+\|u\|_{\mathcal{S}_T^{1}}^2)+b_j^2\epsilon^{2\sigma-1}\|u\|_{\mathcal{S}_T^1}\|u\|_{\mathcal{X}_T^{2s}}\|v\|_{\mathcal{X}_T^{2s}}
\\
&+b_j^2\epsilon^{4\sigma-2}\|u\|_{\mathcal{S}_T^{1}}^2\|v\|_{\mathcal{X}_T^{2s}}^2.
\end{split}
\end{equation}
\end{lemma}
\begin{remark}
As a brief but important remark, it should be noted that for $\alpha\geq 0$ there is no need to distinguish between $\|u\|_{L_t^{\infty}H_x^{\alpha}}$ and $\|u\|_{L_T^{\infty}H_x^{\alpha}}$. This is because outside of $[-2,2]$, $u$ solves a linear Schr\"odinger equation, and so the $H_x^{\alpha}$ norms are constant on both $(-\infty,-2]$ and $[2,\infty)$.
\end{remark}
It will also be convenient to introduce the notation $\tilde{v}:=\tilde{\eta}v$ where $\tilde{\eta}$ is a time-dependent cutoff supported in $(-2,2)$ which is equal to $1$ on the support of $\eta$. For notational convenience, we also write $|v|^{2\sigma}_{<k}$ to denote $P_{<k}|v|^{2\sigma}$. Now, we begin with the proof of the energy type bound in \Cref{Etime1}. 
\begin{proof}
 Note that we can write $\eta|v|^{2\sigma}_{<k}=\eta|\tilde{v}|^{2\sigma}_{<k}$. Next, we apply $\tilde{S}_{2j}P_j:=S_{[2j-4,2j+4]}P_j$ to the equation and see that $\tilde{S}_{2j}P_ju$ solves the equation,
\begin{equation}\label{localized}
(i\partial_t+\partial_x^2)\tilde{S}_{2j}P_ju=i\tilde{S}_{2j}P_j(\eta|\tilde{v}|^{2\sigma}_{<k}u_x),
\end{equation}
with initial data $(\tilde{S}_{2j}P_ju)(0)$. Next, we do a paradifferential expansion of the ``nonlinear" term $i\tilde{S}_{2j}P_j(\eta|\tilde{v}|^{2\sigma}_{<k}u_x)$,  in both the space and time variable, which splits this term into five interactions. Indeed, first by commuting the spatial projection $P_j$, we have 
\begin{equation}
\begin{split}
\tilde{S}_{2j}P_j(i\eta|\tilde{v}|^{2\sigma}_{<k}u_x)&=\tilde{S}_{2j}(i\eta P_{<j-4}|\tilde{v}|_{<k}^{2\sigma}\partial_xP_ju)+\tilde{S}_{2j}(i\eta [P_j,P_{<j-4}|\tilde{v}|^{2\sigma}_{<k}]\partial_xu)+\tilde{S}_{2j}P_j(i\eta P_{\geq j-4}|\tilde{v}|^{2\sigma}_{<k}\partial_xu).
\end{split}
\end{equation}
Then by commuting the temporal projection $\tilde{S}_{2j}$ in the first term, we obtain 
\begin{equation}
\begin{split}
\tilde{S}_{2j}P_j(i\eta|\tilde{v}|^{2\sigma}_{<k}u_x)&=S_{<2j-8}(i\eta P_{<j-4}|\tilde{v}|_{<k}^{2\sigma})\partial_xP_j\tilde{S}_{2j}u+[\tilde{S}_{2j},S_{<2j-8}(i\eta P_{<j-4}|\tilde{v}|_{<k}^{2\sigma})]\partial_xP_ju
\\
&+\tilde{S}_{2j}(S_{\geq 2j-8}(i\eta P_{<j-4}|\tilde{v}|_{<k}^{2\sigma})\partial_xP_ju)+\tilde{S}_{2j}(i\eta [P_j,P_{<j-4}|\tilde{v}|^{2\sigma}_{<k}]\partial_xu)
\\
&+i\tilde{S}_{2j}P_j(\eta P_{\geq j-4}|\tilde{v}|^{2\sigma}_{<k}\partial_xu).
\end{split}
\end{equation}
We label these terms in the order they appear above as $A_1,\dots,A_5$.
\\
\\
We make a brief remark about each of the above interactions before proceeding with the estimates. The first term, $A_1$, which corresponds to the  low-high interaction (in spatial frequency) between the coefficient $i\eta |\tilde{v}|^{2\sigma}_{<k}$ and $\partial_x u$ reacts well to a standard energy type estimate for $P_j\tilde{S}_{2j}u$ since the single spatial derivative $\partial_x$ on $P_j\tilde{S}_{2j}u$ can be integrated by parts onto the coefficient $i\eta |\tilde{v}|^{2\sigma}_{<k}$. The terms $A_2, A_3$ and $A_4$ are expected to be treated perturbatively. These in a very loose sense correspond to more balanced frequency interactions for which (space or time) derivatives can be distributed somewhat evenly between the terms $\partial_xu$ and $i\eta |\tilde{v}|^{2\sigma}_{<k}$. The most serious issue comes from $A_5$, which is the situation in which the coefficient $i\eta |\tilde{v}|^{2\sigma}_{<k}$ is at high spatial frequency compared to $\partial_xu$. Some care must be taken here to ensure that this term is not ``differentiated" $2s$ times in the spatial variable, but instead ``differentiated" at most only $s$ times in the time variable. 
\\
\\
Now, we continue with the proof. We begin with a standard energy type estimate. Indeed, multiplying (\ref{localized}) by $-i2^{4js}\overline{\tilde{S}_{2j}P_ju}$, taking real part and integrating over $\mathbb{R}$ in the spatial variable and from $0$ to $\overline{T}$ with $|\overline{T}|\leq 2$ gives,  
\begin{equation}
\begin{split}
\|D_x^{2s}\tilde{S}_{2j}P_ju\|_{L_T^{\infty}L_x^2}^2&\lesssim 2^{4js}\|(\tilde{S}_{2j}P_ju)(0)\|_{L_x^2}^2+\sum_{k=1}^{5}I_j^k
\end{split}    
\end{equation}
where 
\begin{equation}
I_j^k:=2^{4js}\int_{-T}^{T}\big\lvert\RE\int_{\mathbb{R}}-iA_k\overline{\tilde{S}_{2j}P_ju}\big\rvert.
\end{equation}
Now, we estimate each term. We need to deal with both the initial data term $2^{4js}\|(\tilde{S}_{2j}P_ju)(0)\|_{L_x^2}^2$ and the $I_j^k$ terms for $k=1,...,5$. First we deal with the latter terms.
\\
\\
\textbf{Estimate for $I_j^1$}
\\
\\
We integrate by parts and use Bernstein's inequality to obtain 
\begin{equation}
\begin{split}
I_j^1&=2^{4js}\int_{-T}^{T}\big\lvert\RE\int_{\mathbb{R}} S_{<2j-8}(\eta P_{<j-4}|\tilde{v}|_{<k}^{2\sigma})\partial_x\tilde{S}_{2j}P_ju\overline{\tilde{S}_{2j}P_ju}\big\rvert
\\
&\lesssim 2^{4js}\int_{-T}^{T}\big\lvert\RE\int_{\mathbb{R}} S_{<2j-8}\partial_x(\eta P_{<j-4}|\tilde{v}|_{<k}^{2\sigma})|\tilde{S}_{2j}P_ju|^2\big\rvert
\\
&\lesssim 2^{4js}\|v\|_{L_T^{\infty}L_x^{\infty}}^{2\sigma-1}\|v_x\|_{L_T^1L_x^{\infty}}\|P_ju\|_{L_T^{\infty}L_x^2}^2
\\
&\lesssim \|v\|_{\mathcal{S}_T^1}^{2\sigma}\|D_x^{2s}P_ju\|_{L_T^{\infty}L_x^2}^2
\\
&\lesssim b_j^2\|v\|_{\mathcal{S}_T^1}^{2\sigma}\|u\|_{\mathcal{X}_T^{2s}}^2
\\
&\lesssim b_j^2\epsilon^{2\sigma}\|u\|_{\mathcal{X}_T^{2s}}^2.
\end{split}
\end{equation}
\textbf{Estimate for $I_j^2$}
\\
\\
As mentioned above, this term can be treated perturbatively. For simplicity, we denote $g:=i\eta P_{<j-4}|\tilde{v}|_{<k}^{2\sigma}$. Then \Cref{Commutator leibniz} gives
\begin{equation}
[\tilde{S}_{2j},S_{<2j-8}(i\eta P_{<j-4}|\tilde{v}|_{<k}^{2\sigma})]\partial_xP_ju=2^{-2j}\int_{\mathbb{R}^2}K(s)[\partial_tS_{<2j-8}g](t+s_1,x)[\partial_xP_ju](t+s_2,x)ds
\end{equation}
for some $K\in L^1(\mathbb{R}^2)$. H\"older's inequality, Minkowski's inequality, Bernstein's inequality and the fact that $\|P_ju\|_{L_t^{\infty}L_x^2}=\|P_ju\|_{L_T^{\infty}L_x^2}$ then gives 
\begin{equation}
\begin{split}
I_j^2&\lesssim 2^{-2j}2^{4js}\int_{\mathbb{R}^2}|K(s)|\int_{-T}^{T}\int_{\mathbb{R}}|[\partial_tS_{<2j-8}g](t+s_1,x)[\partial_xP_ju](t+s_2,x)||(\tilde{S}_{2j}P_ju)(t,x)|dxdtds
\\
&\lesssim 2^{-j}2^{4js}\|\partial_tS_{<2j-8}g\|_{L_t^2L_x^{\infty}}\|P_ju\|_{L_T^{\infty}L_x^2}^2
\\
&\lesssim 2^{(\epsilon_0-1)j}2^{4js}\|\partial_tS_{<2j-8}g\|_{L_t^2L_x^{\frac{1}{\epsilon_0}}}\|P_ju\|_{L_T^{\infty}L_x^2}^2
\\
&\lesssim \| D_t^{\frac{1}{2}+\frac{\epsilon_0}{2}}(\eta P_{<j-4}|\tilde{v}|_{<k}^{2\sigma})\|_{L_t^2L_x^{\frac{1}{\epsilon_0}}}\|P_jD_x^{2s}u\|_{L_T^{\infty}L_x^2}^2,
\end{split}
\end{equation}
where $\epsilon_0<\delta$ is some small positive constant. From the fractional Leibniz rule and then the vector valued Moser bound \Cref{Moservec}, Sobolev embedding and then \Cref{cor to spacetimeconv}, we obtain 
\begin{equation} 
\begin{split}
\|D_t^{\frac{1}{2}+\frac{\epsilon_0}{2}}(\eta P_{<j-4}|\tilde{v}|_{<k}^{2\sigma})\|_{L_t^2L_x^{\frac{1}{\epsilon_0}}}&\lesssim \|\tilde{v}\|_{\mathcal{S}_T^1}^{2\sigma}+\|\tilde{v}\|_{L_t^{\infty}L_x^{\infty}}^{2\sigma-1}\|D_t^{\frac{1}{2}+\frac{\epsilon_0}{2}}\tilde{v}\|_{L_t^{4}L_x^{\frac{1}{\epsilon_0}}}
\\
&\lesssim \epsilon^{2\sigma}.
\end{split}
\end{equation}
Hence, 
\begin{equation}
I_j^2\lesssim b_j^2\epsilon^{2\sigma}\|u\|_{\mathcal{X}_T^{2s}}^2.
\end{equation}
\textbf{Estimate for $I_j^3$}
\\
\\
This term can also be dealt with perturbatively. Indeed, we can use H\"older and then Bernstein's inequality to shift a factor of $D_t^{\frac{1}{2}}$ onto the rough part of the nonlinearity, 
\begin{equation}
\begin{split}
I_j^3&\lesssim 2^{4js}\|\tilde{S}_{2j}(S_{\geq 2j-8}(\eta P_{<j-4}|\tilde{v}|_{<k}^{2\sigma})\partial_xP_ju)\|_{L_t^2L_x^2}\|P_ju\|_{L_T^{2}L_x^2}
\\
&\lesssim 2^j\|S_{\geq 2j-8}(\eta P_{<j-4}|\tilde{v}|_{<k}^{2\sigma})\|_{L_t^2L_x^{\infty}}\|P_jD_x^{2s}u\|_{L_T^{\infty}L_x^2}^2
\\
&\lesssim \|S_{\geq 2j-8}D_t^{\frac{1}{2}}(\eta P_{<j-4}|\tilde{v}|_{<k}^{2\sigma})\|_{L_t^2L_x^{\infty}}\|P_jD_x^{2s}u\|_{L_T^{\infty}L_x^2}^2
\\
&\lesssim b_j^2\|D_t^{\frac{1}{2}+\frac{\epsilon_0}{2}}(\eta P_{<j-4}|\tilde{v}|_{<k}^{2\sigma})\|_{L_t^2L_x^{\frac{1}{\epsilon_0}}}\|u\|_{\mathcal{X}_T^{2s}}^2.
\end{split}
\end{equation}
By a similar argument to the estimate for $I_j^2$, we then obtain,
\begin{equation}
I_j^3\lesssim b_j^2\epsilon^{2\sigma}\|u\|_{\mathcal{X}_T^{2s}}^2.
\end{equation}
\textbf{Estimate for $I_j^4$}
\\
\\
This term is also straightforward to deal with directly. The estimate is somewhat analogous to $I_j^2$. We have by \Cref{Commutator leibniz},
\begin{equation}
[P_j,P_{<j-4}|\tilde{v}|^{2\sigma}_{<k}]\partial_xu=2^{-j}\int_{\mathbb{R}^2}K(y)[P_{<j-4}\partial_x|\tilde{v}|_{<k}^{2\sigma}](x+y_1)[\tilde{P}_j\partial_xu](x+y_2)dy
\end{equation}
for some integrable kernel $K\in L^1(\mathbb{R}^2)$. Hence, by Minkowski's inequality, H\"older's inequality and Bernstein's inequality, 
\begin{equation}
\begin{split}
I_j^4&\lesssim \|\partial_x|\tilde{v}|_{<k}^{2\sigma}\|_{L_T^1L_x^{\infty}}\|D_x^{2s}\tilde{P}_ju\|_{L_T^{\infty}L_x^2}^2
\\
&\lesssim b_j^2\epsilon^{2\sigma}\|u\|_{\mathcal{X}_T^{2s}}^2.
\end{split}
\end{equation}
\textbf{Estimate for $I_j^5$}
\\
\\
As remarked on earlier, this is the most troublesome term to deal with since the rough coefficient $|\tilde{v}|^{2\sigma}_{<k}$ is at high spatial frequency. To deal with this, first write $w=\eta u$. We expand using the Littlewood-Paley trichotomy,
\begin{equation}\label{twoterms2}
\tilde{S}_{2j}P_j(\eta P_{\geq j-4}|\tilde{v}|^{2\sigma}_{<k}\partial_xu)=\sum_{m\geq j}\tilde{S}_{2j}P_j(\tilde{P}_m|\tilde{v}|^{2\sigma}_{<k}\partial_x\tilde{P}_mw)+\tilde{S}_{2j}P_j(\tilde{P}_j|\tilde{v}|^{2\sigma}_{<k}\partial_x\tilde{P}_{<j}w).
\end{equation}
The first term above, where the frequency interactions between $\partial_x w$ and $|\tilde{v}|_{<k}^{2\sigma}$ are balanced, is relatively straightforward to estimate. Indeed,
\begin{equation}\label{balancedhighreg}
\begin{split}
2^{4js}\int_{-T}^{T}\lvert\int_{\mathbb{R}}\overline{\tilde{S}_{2j}P_ju}\sum_{m\geq j}\tilde{S}_{2j}P_j(\tilde{P}_m|\tilde{v}|^{2\sigma}_{<k}\partial_x\tilde{P}_mw)\rvert&\lesssim \|D_x^{2s}\tilde{S}_{2j}P_ju\|_{L_T^{\infty}L_x^2}2^{2js}\|\sum_{m\geq j}\tilde{S}_{2j}P_j(\tilde{P}_m|\tilde{v}|^{2\sigma}_{<k}\partial_x\tilde{P}_mw)\|_{L_T^1L_x^2} 
\\
&\lesssim \|D_x^{2s}\tilde{S}_{2j}P_ju\|_{L_T^{\infty}L_x^2}\sum_{m\geq j}2^{2js}\|\tilde{P}_m|\tilde{v}|^{2\sigma}_{<k}\partial_x\tilde{P}_mw\|_{L_t^1L_x^2} 
\\
&\lesssim b_j\|u\|_{\mathcal{X}_T^{2s}}\sum_{m\geq j}2^{2js}\|\tilde{P}_m|\tilde{v}|^{2\sigma}_{<k}\partial_x\tilde{P}_mw\|_{L_t^1L_x^2}
\\
&\lesssim b_j\|\partial_x|\tilde{v}|^{2\sigma}_{<k}\|_{L_T^1L_x^{\infty}}\|u\|_{\mathcal{X}_T^{2s}}\sum_{m\geq j}2^{2(j-m)s}\|\tilde{P}_mD_x^{2s}u\|_{L_T^{\infty}L_x^2} 
\\
&\lesssim b_j\epsilon^{2\sigma}\|u\|_{\mathcal{X}_T^{2s}}^2\sum_{m\geq j}2^{2(j-m)s}b_m
\\
&\lesssim b_j^2\epsilon^{2\sigma}\|u\|_{\mathcal{X}_T^{2s}}^2 
\end{split}
\end{equation}
where in the last line we used the slowly varying property of $b_j$.
\\
\\
For the second term in (\ref{twoterms2}), we distribute the temporal projection to obtain 
\begin{equation}\label{twoterms4}
\begin{split}
\tilde{S}_{2j}P_j(\tilde{P}_j|\tilde{v}|^{2\sigma}_{<k}\partial_x\tilde{P}_{<j}w)=\tilde{S}_{2j}P_j(\tilde{P}_j|\tilde{v}|^{2\sigma}_{<k}\partial_x\tilde{P}_{<j}S_{\geq 2j-8}w)+\tilde{S}_{2j}P_j(\tilde{P}_j\tilde{S}_{2j}|\tilde{v}|_{<k}^{2\sigma}\partial_xP_{<j}S_{<2j-8}w).
\end{split}
\end{equation}
For the first term in (\ref{twoterms4}), we use Bernstein's inequality and then  \Cref{cor to spacetimeconv}, which yields
\begin{equation}
\begin{split}
2^{2js}\|\tilde{S}_{2j}P_j(\tilde{P}_j|\tilde{v}|^{2\sigma}_{<k}\partial_x\tilde{P}_{<j}S_{\geq 2j-8}w)\|_{L_T^1L_x^2}&\lesssim 2^{-j\epsilon_0}\|D_x^{1+\epsilon_0}|v|^{2\sigma}\|_{L_T^1L_x^{\infty}}\|S_{\geq 2j-8}D_t^{s}w\|_{L_t^{\infty}L_x^2}
\\
&\lesssim 2^{-j\epsilon_0}\|D_x^{1+\epsilon_0}|v|^{2\sigma}\|_{L_T^1L_x^{\infty}}\|P_{\leq 0}S_{\geq 2j-8}D_t^{s}w\|_{L_t^{\infty}L_x^2}
\\
&+2^{-j\epsilon_0}\|D_x^{1+\epsilon_0}|v|^{2\sigma}\|_{L_T^1L_x^{\infty}}\left(\sum_{m>0}\|P_{m}D_t^{s}w\|_{L_t^{\infty}L_x^2}^2\right)^{\frac{1}{2}}
\\
&\lesssim 2^{-j\epsilon_0}\|D_x^{1+\epsilon_0}|v|^{2\sigma}\|_{L_T^1L_x^{\infty}}(\|u\|_{\mathcal{X}_T^{2s}}+\|g\|_{Z_{\infty}^{s-1+\delta}})
\end{split}
\end{equation}
where $g:=(i\partial_t+\partial_x^2)w$ and $0<\epsilon_0\ll \delta$ is some small positive constant. If $\epsilon_0$ is small enough, then \Cref{Holder} gives $\|D_x^{1+\epsilon_0}|v|^{2\sigma}\|_{L_T^1L_x^{\infty}}\lesssim \|v\|_{\mathcal{S}_T^{1+\delta}}^{2\sigma}\lesssim\epsilon^{2\sigma}$. Then finally by taking $2^{-j\epsilon_0}\lesssim b_j$, it follows that
\begin{equation}
\begin{split}
2^{2js}\|\tilde{S}_{2j}P_j(\tilde{P}_j|\tilde{v}|^{2\sigma}_{<k}\partial_x\tilde{P}_{<j}S_{\geq 2j-8}w)\|_{L_T^1L_x^2}&\lesssim b_j\epsilon^{2\sigma}(\|u\|_{\mathcal{X}_T^{2s}}+\|g\|_{Z_{\infty}^{s-1+\delta}}).
\end{split}
\end{equation}
Now we look at controlling the second term in (\ref{twoterms4}). We use Bernstein's inequality and the fact that $w=\eta u$ is time-localized to obtain
\begin{equation}\label{derivativonrough}
\begin{split}
2^{2js}\|\tilde{S}_{2j}P_j(\tilde{P}_j\tilde{S}_{2j}|\tilde{v}|_{<k}^{2\sigma}\partial_xP_{<j}S_{<2j-8}w)\|_{L_T^1L_x^2}&\lesssim 2^{2js}\|\tilde{P}_j\tilde{S}_{2j}|\tilde{v}|_{<k}^{2\sigma}\partial_xP_{<j}S_{<2j-8}w\|_{L_t^1L_x^2}
\\
&\lesssim \|u\|_{\mathcal{S}_T^1}\|D_t^s\tilde{P}_j\tilde{S}_{2j}|\tilde{v}|^{2\sigma}\|_{L_t^2L_x^2}.
\end{split}    
\end{equation}
Here we crucially ensured that the time derivative $D_t^s$, rather than the spatial derivative $D_x^{2s}$ fell on the rough part of the nonlinearity.
\\
\\
To control $\|D_t^s\tilde{P}_j\tilde{S}_{2j}|\tilde{v}|^{2\sigma}\|_{L_t^2L_x^2}$ we will need the following low modulation Moser type estimate.
\begin{lemma}\label{timemoser2}
Given the conditions of \Cref{Etime1}, the following estimate holds: 
\begin{equation}
\|\tilde{P}_j\tilde{S}_{2j}D_t^s|\tilde{v}|^{2\sigma}\|_{L_t^2L_x^2}\lesssim b_j\epsilon^{2\sigma-1}(\epsilon+\|v\|_{\mathcal{X}_T^{2s}}).
\end{equation}
\end{lemma}
We will postpone the proof of this technical lemma until the end of the section. 
\\
\\
Combining \Cref{timemoser2} and the estimate (\ref{balancedhighreg}) allows us to estimate $I_j^5$ by
\begin{equation}
I_j^5\lesssim b_j^2\epsilon^{2\sigma}(\|u\|_{\mathcal{X}_{T}^{2s}}^2+\|u\|_{\mathcal{S}_T^{1}}^2+\|g\|_{Z_{\infty}^{s-1+\delta}}^2)+b_j^2\epsilon^{2\sigma-1}\|u\|_{\mathcal{S}_T^1}\|u\|_{\mathcal{X}_T^{2s}}\|v\|_{\mathcal{X}_T^{2s}}.
\end{equation}
Finally, combining the estimates for $I_j^1,\dots,I_j^5$ now yields
\begin{equation}\label{Etime2}
\begin{split}
    \|D_x^{2s}\tilde{S}_{2j}P_ju\|_{L_T^{\infty}L_x^2}^2&\lesssim 2^{4js}\|(\tilde{S}_{2j}P_ju)(0)\|_{L_x^2}^2+b_j^2\epsilon^{2\sigma}(\|u\|_{\mathcal{X}_{T}^{2s}}^2+\|u\|_{S_T^1}^2+\|g\|_{Z_{\infty}^{s-1+\delta}}^2)
    \\
    &+b_j^2\epsilon^{2\sigma-1}\|u\|_{\mathcal{S}_T^1}\|u\|_{\mathcal{X}_T^{2s}}\|v\|_{\mathcal{X}_T^{2s}}.
\end{split}
\end{equation}
Next, we need to control $(\tilde{S}_{2j}P_ju)(0)$ in terms of $P_ju_0$. To accomplish this, we use the high modulation estimate \Cref{spacetimeconv}. Namely,
\begin{equation}\label{initestimate}
\begin{split}
2^{2js}\|(\tilde{S}_{2j}P_ju)(0)\|_{L_x^2}&\lesssim \|D_x^{2s}P_ju_0\|_{L_x^2}+\|(1-\tilde{S}_{2j})P_jD_x^{2s}u\|_{L_t^{\infty}L_x^2}
\\
&\lesssim \|D_x^{2s}P_ju_0\|_{L_x^2}+\|S_{\leq 0}P_jD_x^{2s}u\|_{L_t^{\infty}L_x^2}+\sum_{m>0, |m-2j|>4}\|P_jS_mD_x^{2s}u\|_{L_t^{\infty}L_x^2}
\\
&\lesssim \|D_x^{2s}P_ju_0\|_{L_x^2}+\|\langle D_t\rangle^{s-1+\delta}\tilde{P}_j(\eta|\tilde{v}|_{<k}^{2\sigma}u_x)\|_{L_t^{\infty}L_x^2}.
\end{split}
\end{equation}
In light of (\ref{Etime2}) and (\ref{initestimate}), to complete the proof of \Cref{Etime1} it remains to estimate the latter term on the right hand side of (\ref{initestimate}) as well as $\|g\|_{Z_{\infty}^{s-1+\delta}}$. This is done in the following lemma. 
\begin{lemma}\label{nonlinest} Let $s,\sigma, T,u_0,u,a_j$ and $b_j$ be as in \Cref{keyestimate1}. Let $v$ also be as in \Cref{keyestimate1}, but with \eqref{vassumption} replaced by the weaker assumption that for all $0<\epsilon\ll 1$ and  $0<\delta\ll 1$,
\begin{equation}\label{vassumption2}
\|v\|_{\mathcal{S}_T^{1+\delta}}+\|(i\partial_t+\partial_x^2)v\|_{Z_{\infty}^{s-\frac{3}{2}+\delta}}\lesssim_{\delta} \epsilon.    
\end{equation}
Then we have
\begin{equation}\label{finalest1}
 \|\tilde{P}_j(\eta|\tilde{v}|_{<k}^{2\sigma}u_x)\|_{Z_{\infty}^{s-1+\delta}}\lesssim b_j\epsilon^{2\sigma}(\|u\|_{\mathcal{S}_T^1}+\|u\|_{\mathcal{X}_T^{2s-1+c\delta}})+b_j\epsilon^{2\sigma-1}\|u\|_{\mathcal{S}_T^1}\|v\|_{\mathcal{X}_T^{2s-1+c\delta}}  
\end{equation}
and
\begin{equation}\label{finalest2}
\|(i\partial_t+\partial_x^2)w\|_{Z_{\infty}^{s-1+\delta}}:=\|g\|_{Z_{\infty}^{s-1+\delta}}\lesssim \|u\|_{\mathcal{X}_T^{2s-1+c\delta}}+\|u\|_{\mathcal{S}_T^1}+\|u\|_{\mathcal{S}_T^1}\|v\|_{\mathcal{X}_T^{2s-1+c\delta}},
\end{equation}
for some constant $c>0$.
\end{lemma}
\begin{remark}
The reader may wonder why we estimate the full $Z_{\infty}^{s-1+\delta}$ norm in the above lemma. Although the argument up until this point only requires us to estimate the component of the $Z_{\infty}^{s-1+\delta}$ norm involving the time derivative, we will need to also estimate the component involving spatial derivatives in the next section when we establish well-posedness for the full equation in $\mathcal{X}_T^{2s}$. 
\end{remark}
\begin{proof}
We begin with (\ref{finalest1}). For the purpose of not having to track all the factors of $\delta$ that appear throughout the proof, we will denote by $c>0$ some positive constant which is allowed to grow from line to line. First we study the component of the $Z_{\infty}^{s-1+\delta}$ norm which involves the time derivative. By considering separately temporal frequencies larger than $2^{2j}$ and smaller than $2^{2j}$, we obtain (using the vector valued Bernstein inequality),
\begin{equation}
\begin{split}
\|\tilde{P}_j\langle D_t\rangle^{s-1+\delta}(\eta|\tilde{v}|_{<k}^{2\sigma}u_x)\|_{L_t^{\infty}L_x^2}&\lesssim 2^{-2j\delta}\|\tilde{P}_j\langle D_x\rangle^{2s-2+c\delta}(\eta|\tilde{v}|_{<k}^{2\sigma}u_x)\|_{L_t^{\infty}L_x^2}
\\
&+2^{-2j\delta}\|\tilde{P}_jS_{> 2j}\langle D_t\rangle^{s-1+c\delta}(\eta|\tilde{v}|_{<k}^{2\sigma}u_x)\|_{L_t^{\infty}L_x^2}.
\end{split}
\end{equation}
Hence,
\begin{equation}\label{twoestimates}
\begin{split}
\|\tilde{P}_j(\eta|\tilde{v}|_{<k}^{2\sigma}u_x)\|_{Z_{\infty}^{s-1+\delta}}&\lesssim 2^{-2j\delta}\|\tilde{P}_j\langle D_x\rangle^{2s-2+c\delta}(\eta|\tilde{v}|_{<k}^{2\sigma}u_x)\|_{L_t^{\infty}L_x^2}
\\
&+2^{-2j\delta}\|\tilde{P}_jS_{> 2j}\langle D_t\rangle^{s-1+c\delta}(\eta|\tilde{v}|_{<k}^{2\sigma}u_x)\|_{L_t^{\infty}L_x^2}.   
\end{split}
\end{equation}
We now look at the first term in (\ref{twoestimates}).
 The bound 
 \begin{equation}\label{spacebound}
 2^{-2j\delta}\|\tilde{P}_j\langle D_x\rangle^{2s-2+c\delta}(\eta|\tilde{v}|_{<k}^{2\sigma}u_x)\|_{L_t^{\infty}L_x^2}\lesssim b_j\epsilon^{2\sigma}\|u\|_{\mathcal{X}_T^{2s-1+c\delta}}+b_j\epsilon^{2\sigma-1}\|u\|_{\mathcal{S}_T^1}\|v\|_{\mathcal{X}_T^{2s-1+c\delta}}  
 \end{equation}
 is a straightforward consequence of $2^{-2j\delta}\lesssim b_j$ and the fractional Leibniz rule if $2s-2<1$. If $2s-2\geq 1$, then for the homogeneous component, we have 
\begin{equation}
\begin{split}
\|D_x^{2s-2+c\delta}(i\eta P_{<k}|v|^{2\sigma}u_x)\|_{L_t^{\infty}L_x^2}&\lesssim \|D_x^{2s-3+c\delta}(i\eta P_{<k}|v|^{2\sigma}u_{xx})\|_{L_t^{\infty}L_x^2}
\\
&+\|\eta D_x^{2s-3+c\delta}(\RE P_{<k}(|v|^{2\sigma-2}\overline{v}v_{x})u_x)\|_{L_t^{\infty}L_x^2}.
\end{split}
\end{equation}
By the fractional Leibniz rule and Sobolev embedding, clearly the first term above can be controlled by $\epsilon^{2\sigma}\|u\|_{\mathcal{X}_T^{2s-1+c\delta}}$. Using the fact that $2s-3<2\sigma-1$ and applying the fractional Leibniz rule, \Cref{Holder} (when $D_x^{2s-3+c\delta}$ falls on $|v|^{2\sigma-2}\overline{v}$) and interpolation, we  can control the second term by
\begin{equation}
\epsilon^{2\sigma}\|u\|_{\mathcal{X}_T^{2s-1+c\delta}}+\epsilon^{2\sigma-1}\|u\|_{\mathcal{S}_T^1}\|v\|_{\mathcal{X}_T^{2s-1+c\delta}}    
\end{equation}
to obtain the desired bound (\ref{spacebound}).
\\
\\
Now, to estimate the second term on the right hand side of (\ref{twoestimates}), we use that $2^{-2j\delta}\lesssim b_j$ and estimate
\begin{equation}\label{twoestimates2}
\begin{split}
2^{-2j\delta}\|\tilde{P}_jS_{>2j}\langle D_t\rangle^{s-1+c\delta}(\eta|\tilde{v}|_{<k}^{2\sigma}u_x)\|_{L_t^{\infty}L_x^2}&\lesssim b_j\|\tilde{P}_jS_{>2j}\langle D_t\rangle^{s-1+c\delta}(\eta|\tilde{v}|_{<k}^{2\sigma}u_x)\|_{L_t^{\infty}L_x^2}  
\\
&\lesssim b_j\sum_{m\geq 2j}\|\tilde{P}_jS_{m}\langle D_t\rangle^{s-1+c\delta}(\eta|\tilde{v}|_{<k}^{2\sigma}u_x)\|_{L_t^{\infty}L_x^2}
\\
&\lesssim b_j\sum_{m\geq 2j}\|\tilde{P}_jS_{m}\langle D_t\rangle^{s-1+c\delta}(S_{<m-4}(\eta|\tilde{v}|_{<k}^{2\sigma})\tilde{S}_{m}u_x)\|_{L_t^{\infty}L_x^2}
\\
&+b_j\sum_{m\geq 2j}\|\tilde{P}_jS_{m}\langle D_t\rangle^{s-1+c\delta}(S_{\geq m-4}(\eta|\tilde{v}|_{<k}^{2\sigma})u_x)\|_{L_t^{\infty}L_x^2}.
\end{split}
\end{equation}
For the first term in (\ref{twoestimates2}), we have by  Bernstein's inequality,
\begin{equation}\label{timebound}
\begin{split}
&b_j\sum_{m\geq 2j}\|\tilde{P}_jS_{m}\langle D_t\rangle^{s-1+c\delta}(S_{<m-4}(\eta|\tilde{v}|_{<k}^{2\sigma})\tilde{S}_{m}u_x)\|_{L_t^{\infty}L_x^2} 
\\
&\lesssim b_j\sum_{m\geq 2j}\|\tilde{P}_jS_{m}\langle D_t\rangle^{s-\frac{1}{2}+c\delta}(S_{<m-4}(\eta|\tilde{v}|_{<k}^{2\sigma})\tilde{S}_{m}u)\|_{L_t^{\infty}L_x^2}
\\
&+b_j\sum_{m\geq 2j}\|\tilde{P}_jS_{m}\langle D_t\rangle^{s-1+c\delta}(S_{<m-4}(\eta\partial_x|\tilde{v}|_{<k}^{2\sigma})\tilde{S}_{m}u)\|_{L_t^{\infty}L_x^2}.
\\
\end{split}
\end{equation}
Using Bernstein's inequality and \Cref{cor to spacetimeconv}, we may control the first term by
\begin{equation}
\begin{split}
b_j\sum_{m\geq 2j}\|\tilde{P}_jS_{m}\langle D_t\rangle^{s-\frac{1}{2}+c\delta}(S_{<m-4}(\eta|\tilde{v}|_{<k}^{2\sigma})\tilde{S}_{m}u)\|_{L_t^{\infty}L_x^2}&\lesssim b_j\|v\|_{\mathcal{S}_T^{1}}^{2\sigma}\|\langle D_t\rangle^{s-\frac{1}{2}+c\delta}u\|_{L_t^{\infty}L_x^2}   
\\
&\lesssim b_j\epsilon^{2\sigma}\|\langle D_t\rangle^{s-\frac{1}{2}+c\delta}u\|_{L_t^{\infty}L_x^2}
\\
&\lesssim b_j\epsilon^{2\sigma}(\|u\|_{\mathcal{X}_T^{2s-1+c\delta}}+\|\eta|v|^{2\sigma}_{<k}u_x\|_{Z_{\infty}^{s-\frac{3}{2}+c\delta}}).
\end{split}
\end{equation}
For the second term in (\ref{timebound}), we obtain also 
\begin{equation}
\begin{split}
b_j\sum_{m\geq 2j}\|\tilde{P}_jS_{m}\langle D_t\rangle^{s-1+c\delta}(S_{<m-4}(\eta\partial_x|\tilde{v}|_{<k}^{2\sigma})\tilde{S}_{m}u)\|_{L_t^{\infty}L_x^2}&\lesssim b_j\||\tilde{v}|^{2\sigma}\|_{L_t^{\infty}H_x^1}\|\langle D_t\rangle^{s-1+c\delta}u\|_{L_t^{\infty}L_x^{\infty}}
\\
&\lesssim b_j\||\tilde{v}|^{2\sigma}\|_{L_t^{\infty}H_x^1}\|\langle D_t\rangle^{s-1+c\delta}\langle D_x\rangle^{\frac{1}{2}+\delta}u\|_{L_t^{\infty}L_x^{2}}
\\
&\lesssim b_j\epsilon^{2\sigma}(\|\langle D_t\rangle^{s-\frac{1}{2}+c\delta}u\|_{L_t^{\infty}L_x^2}+\|\langle D_x\rangle^{s-\frac{1}{2}+c\delta}u\|_{L_t^{\infty}L_x^2})
\\
&\lesssim b_j\epsilon^{2\sigma}(\|u\|_{\mathcal{X}_T^{2s-1+c\delta}}+\|\eta|v|^{2\sigma}_{<k}u_x\|_{Z_{\infty}^{s-\frac{3}{2}+c\delta}}).
\end{split}
\end{equation}
 For the second term in (\ref{twoestimates2}), we obtain
\begin{equation}
\begin{split}
b_j\sum_{m\geq 2j}\|\tilde{P}_jS_{m}\langle D_t\rangle^{s-1+c\delta}(S_{\geq m-4}(\eta|\tilde{v}|_{<k}^{2\sigma})u_x)\|_{L_t^{\infty}L_x^2}&\lesssim b_j\|u_x\|_{L_T^{\infty}L_x^2}\|\langle D_t\rangle^{s-1+c\delta}(\eta|\tilde{v}|_{<k}^{2\sigma})\|_{L_t^{\infty}L_x^{\infty}}
\\
&\lesssim b_j \|u\|_{\mathcal{S}_T^1}\|\langle D_t\rangle^{s-1+c\delta}(\eta|\tilde{v}|_{<k}^{2\sigma})\|_{L_t^{\infty}L_x^{\infty}}.
\end{split}
\end{equation}
We have by Sobolev embedding, the fractional Leibniz rule and the Moser bound \Cref{scalarMoser},
\begin{equation}
\begin{split}
\|\langle D_t\rangle^{s-1+c\delta}(\eta|\tilde{v}|_{<k}^{2\sigma})\|_{L_t^{\infty}L_x^{\infty}}&\lesssim \|\langle D_t\rangle^{s-1+c\delta}(\eta|\tilde{v}|_{<k}^{2\sigma})\|_{L_x^{\infty}L_t^{\frac{1}{\delta}}}
\\
&\lesssim \||\tilde{v}|^{2\sigma}\|_{L_x^{\infty}L_t^{\frac{2}{\delta}}}+\|\langle D_t\rangle^{s-1+c\delta}|\tilde{v}|_{<k}^{2\sigma}\|_{L_x^{\infty}L_t^{\frac{2}{\delta}}}
\\
&\lesssim \||\tilde{v}|^{2\sigma}\|_{L_t^{\frac{2}{\delta}}L_x^{\infty}}+\||\tilde{v}|^{2\sigma-1}\|_{L_x^{\infty}L_t^{\frac{4}{\delta}}}\|\langle D_t\rangle^{s-1+c\delta}\tilde{v}\|_{L_x^{\infty}L_t^{\frac{4}{\delta}}}
\\
&\lesssim \epsilon^{2\sigma}+\||\tilde{v}|^{2\sigma-1}\|_{L_t^{\frac{4}{\delta}}L_x^{\infty}}\|\langle D_t\rangle^{s-1+c\delta}\tilde{v}\|_{L_t^{\frac{4}{\delta}}L_x^{\infty}}
\\
&\lesssim \epsilon^{2\sigma}+\|v\|_{\mathcal{S}_T^1}^{2\sigma-1}\|\langle D_t\rangle^{s-1+c\delta}\tilde{v}\|_{L_t^{\frac{4}{\delta}}L_x^{\infty}}.
\end{split}
\end{equation}
Now, notice that by \Cref{cor to spacetimeconv},
\begin{equation}
    \begin{split}
        \|\langle D_t\rangle^{s-1+c\delta}\tilde{v}\|_{L_t^{\frac{4}{\delta}}L_x^{\infty}}&\lesssim \sum_{j\geq 0}\|\langle D_t\rangle^{s-1+c\delta}P_j\tilde{v}\|_{L_t^{\frac{4}{\delta}}L_x^{\infty}}
        \\
       & \lesssim \sum_{j\geq 0}\|\langle D_t\rangle^{s-1+c\delta}\langle D_x\rangle^{\frac{1}{2}}P_j\tilde{v}\|_{L_t^{\frac{4}{\delta}}L_x^2}
       \\
       & \lesssim\sum_{j\geq 0}\|\langle D_x\rangle^{s-\frac{1}{2}+c\delta}P_j\tilde{v}\|_{L_t^{\frac{4}{\delta}}L_x^2}+\sum_{j\geq 0}\|\langle D_t\rangle^{s-\frac{1}{2}+c\delta}P_j\tilde{v}\|_{L_t^{\frac{4}{\delta}}L_x^2}
       \\
       &\lesssim \|\tilde{v}\|_{\mathcal{X}_T^{2s-1+c\delta}}+\sum_{j\geq 0}\left(\|\langle D_x\rangle^{2s-1+c\delta}P_j\tilde{v}\|_{L_t^{\frac{4}{\delta}}L_x^2}+\|\tilde{P}_j\langle D_t\rangle^{s-\frac{3}{2}+c\delta}(i\partial_t+\partial_x^2)\tilde{v}\|_{L_t^{\frac{4}{\delta}}L_x^2}\right)
       \\
       & \lesssim \|v\|_{\mathcal{X}_T^{2s-1+c\delta}}+\sum_{j\geq 0}\|\tilde{P}_j\langle D_t\rangle^{s-\frac{3}{2}+c\delta}(i\partial_t+\partial_x^2)\tilde{v}\|_{L_t^{\frac{4}{\delta}}L_x^2}.
    \end{split}
\end{equation}
To control the latter term above, there are two cases. If $s-\frac{3}{2}<0$, then this term can be easily controlled by $\epsilon$ by commuting $(i\partial_t+\partial_x^2)$ with $\tilde{\eta}$ and applying H\"older's inequality. If $s-\frac{3}{2}\geq 0$, then we have (after possibly enlarging $c\delta$)
\begin{equation}\label{commutet}
\begin{split}
\sum_{j\geq 0}\|\tilde{P}_j\langle D_t\rangle^{s-\frac{3}{2}+c\delta}(i\partial_t+\partial_x^2)\tilde{v}\|_{L_t^{\frac{4}{\delta}}L_x^2}&\lesssim \|\langle D_t\rangle^{s-\frac{3}{2}+c\delta}(\eta(i\partial_t+\partial_x^2)v)\|_{L_t^{\frac{4}{\delta}}L_x^{2}}+\|\langle D_t\rangle^{s-\frac{3}{2}+c\delta}(\partial_t\tilde{\eta} v)\|_{L_t^{\frac{4}{\delta}}L_x^{2}}
\\
&\lesssim \sum_{k\geq 0}\|\langle D_t\rangle^{s-\frac{3}{2}+c\delta}S_k(\eta(i\partial_t+\partial_x^2)v)\|_{L_t^{\frac{4}{\delta}}L_x^{2}}+\|\langle D_t\rangle^{s-\frac{3}{2}+c\delta}S_k(\partial_t\tilde{\eta} v)\|_{L_t^{\frac{4}{\delta}}L_x^{2}}.
\end{split}    
\end{equation}
By doing a paraproduct expansion of $S_k(\eta(i\partial_t+\partial_x^2)v)=S_k(S_{<k-4}\eta(i\partial_t+\partial_x^2)\tilde{S}_kv)+S_k(S_{\geq k-4}\eta(i\partial_t+\partial_x^2)v)$, using Bernstein and H\"older's inequality, summing over $k$, and possibly enlarging the factor of $c\delta$, we obtain
\begin{equation}
\sum_{k\geq 0}\|\langle D_t\rangle^{s-\frac{3}{2}+c\delta}S_k(\eta(i\partial_t+\partial_x^2)v)\|_{L_t^{\frac{4}{\delta}}L_x^{2}}\lesssim \|(i\partial_t+\partial_x^2)v\|_{Z_{\infty}^{s-\frac{3}{2}+c\delta}}\lesssim\epsilon.
\end{equation}
A similar argument involving a paraproduct expansion of $S_k(\partial_t\eta v)$ can  be used to show
\begin{equation}
\|\langle D_t\rangle^{s-\frac{3}{2}+c\delta}S_k(\partial_t\eta v)\|_{L_t^{\frac{4}{\delta}}L_x^{2}}\lesssim\epsilon.
\end{equation}
 Therefore, the second term in (\ref{twoestimates2}) can be controlled by
\begin{equation}
b_j\epsilon^{2\sigma}\|u\|_{\mathcal{S}_T^{1}}+b_j\epsilon^{2\sigma-1}\|u\|_{\mathcal{S}_T^1}\|v\|_{\mathcal{X}_T^{2s-1+c\delta}}.    
\end{equation}
Combining this and (\ref{timebound}) with (\ref{spacebound}) yields the estimate,
\begin{equation}\label{closebound}
\|\tilde{P}_j(\eta|\tilde{v}|_{<k}^{2\sigma}u_x)\|_{Z_{\infty}^{s-1+\delta}}\lesssim b_j\epsilon^{2\sigma}(\|u\|_{\mathcal{X}_T^{2s-1+c\delta}}+\|u\|_{\mathcal{S}_T^1}+\|\eta|v|^{2\sigma}_{<k}u_x\|_{Z_{\infty}^{s-\frac{3}{2}+c\delta}})+b_j\epsilon^{2\sigma-1}\|u\|_{\mathcal{S}_T^1}\|v\|_{\mathcal{X}_T^{2s-1+c\delta}}.
\end{equation}
By square summing (\ref{closebound}) and applying (\ref{closebound}) with $s-1$ replaced by $s-\frac{3}{2}$, we obtain
\begin{equation}
\|\eta|v|^{2\sigma}_{<k}u_x\|_{Z_{\infty}^{s-\frac{3}{2}+c\delta}}\lesssim \epsilon^{2\sigma}(\|u\|_{\mathcal{X}_T^{2s-1+c\delta}}+\|u\|_{\mathcal{S}_T^1}+\|\eta|v|^{2\sigma}_{<k}u_x\|_{Z_{\infty}^{s-2+c\delta}})+\epsilon^{2\sigma-1}\|u\|_{\mathcal{S}_T^1}\|v\|_{\mathcal{X}_T^{2s-1+c\delta}} 
\end{equation}
and since $s<2$, it follows that if $\delta$ is small enough, then 
\begin{equation}
\|\eta|v|^{2\sigma}_{<k}u_x\|_{Z_{\infty}^{s-2+c\delta}}\lesssim \epsilon^{2\sigma}\|u\|_{\mathcal{X}_T^{2s-1+c\delta}}.   
\end{equation}
Therefore, the bound
\begin{equation}\label{keybound2}
\|\tilde{P}_j(\eta|\tilde{v}|_{<k}^{2\sigma}u_x)\|_{Z_{\infty}^{s-1+\delta}}\lesssim b_j\epsilon^{2\sigma}(\|u\|_{\mathcal{X}_T^{2s-1+c\delta}}+\|u\|_{\mathcal{S}_T^1})+b_j\epsilon^{2\sigma-1}\|u\|_{\mathcal{S}_T^1}\|v\|_{\mathcal{X}_T^{2s-1+c\delta}}    
\end{equation}
follows.
\\
\\
For the estimate (\ref{finalest2}), we have
\begin{equation}\label{twoterms7}
\|g\|_{Z_{\infty}^{s-1+\delta}}\lesssim \|\partial_t\eta  u\|_{Z_{\infty}^{s-1+\delta}}+\|\eta P_{<k}|v|^{2\sigma}u_x\|_{Z_{\infty}^{s-1+\delta}}.    
\end{equation}
The first term above is controlled using \Cref{cor to spacetimeconv} by
\begin{equation}
\begin{split}
\|\partial_t\eta u\|_{Z_{\infty}^{s-1+\delta}}&\lesssim \|u\|_{\mathcal{X}_T^{2s-1+c\delta}}+ \|\langle D_t\rangle^{s-1+\delta}(\partial_t\eta u)\|_{L_t^{\infty}L_x^2} 
\\
&\lesssim \|u\|_{\mathcal{X}_T^{2s-1+c\delta}}+ \|\partial_t^2\eta u\|_{Z_{\infty}^{s-2+c\delta}}+\|\partial_t\eta (\eta P_{<k}|v|^{2\sigma}u_x)\|_{Z_{\infty}^{s-2+c\delta}}
\\
&\lesssim \|u\|_{\mathcal{X}_T^{2s-1+c\delta}}
\end{split}
\end{equation}
where in the last line, we used that $s<2$. The second term in (\ref{twoterms7}) can be estimated by square summing (\ref{keybound2}). This completes the proof of \Cref{nonlinest}.
\end{proof}
Finally, we complete the proof of \Cref{Etime1}. This simply follows by combining \Cref{nonlinest} with the estimates (\ref{Etime2}) and (\ref{initestimate}).
\end{proof}
\subsection{Proof of \Cref{keyestimate1}}
Finally, we prove the main estimate of the section, \Cref{keyestimate1}. 
\begin{proof}
Let $0<\delta\ll 1$. From \Cref{cor to spacetimeconv}, we have
\begin{equation}
\|P_ju\|_{L_T^{\infty}H_x^{2s}}^2\lesssim_{\delta}\|\tilde{S}_{2j}P_ju\|_{L_T^{\infty}H_x^{2s}}^2+\|\tilde{P}_j(\eta|\tilde{v}|^{2\sigma}_{<k}u_x)\|_{Z_{\infty}^{s-1+\delta}}^2.
\end{equation}
By \Cref{Etime1}, we have
\begin{equation}
\begin{split}
\|\tilde{S}_{2j}P_ju\|_{L_T^{\infty}H_x^{2s}}^2&\lesssim_{\delta} a_j^2\|u_0\|_{H_x^{2s}}^2+b_j^2\epsilon^{2\sigma}(\|u\|_{\mathcal{X}_T^{2s}}^2+\|u\|_{\mathcal{S}_T^{1}}^2)+b_j^2\epsilon^{2\sigma-1}\|u\|_{\mathcal{S}_T^1}\|u\|_{\mathcal{X}_T^{2s}}\|v\|_{\mathcal{X}_T^{2s}}
\\
&+b_j^2\epsilon^{4\sigma-2}\|u\|_{\mathcal{S}_T^{1}}^2\|v\|_{\mathcal{X}_T^{2s}}^2. 
\end{split}
\end{equation}
Furthermore, by \Cref{nonlinest}, we have  
\begin{equation}
\begin{split}
\|\tilde{P}_j(\eta|\tilde{v}|^{2\sigma}_{<k}u_x)\|_{Z_{\infty}^{s-1+\delta}}^2&\lesssim b_j^2\epsilon^{4\sigma}(\|u\|_{\mathcal{X}_T^{2s}}^2+\|u\|_{\mathcal{S}_T^1}^2)+b_j^2\epsilon^{4\sigma-2}\|u\|_{\mathcal{S}_T^{1}}^2\|v\|_{\mathcal{X}_T^{2s}}^2
\\
&\lesssim b_j^2\epsilon^{2\sigma}(\|u\|_{\mathcal{X}_T^{2s}}^2+\|u\|_{\mathcal{S}_T^1}^2)+b_j^2\epsilon^{4\sigma-2}\|u\|_{\mathcal{S}_T^{1}}^2\|v\|_{\mathcal{X}_T^{2s}}^2
.  
\end{split}
\end{equation}
This completes the proof.
\end{proof}
\subsection{Proof of \Cref{timemoser2}}
It remains to prove the technical estimate \Cref{timemoser2}.
This will follow from the slightly more general estimate:
\begin{lemma}\label{timemoserest}
Let $T=2$, $\frac{1}{2}<\sigma<1$ and $u$ be a $C^2(\mathbb{R};H_x^{\infty})$ solution to the inhomogeneous Schr\"odinger equation,
\begin{equation}\label{inhomsc}
(i\partial_t+\partial_x^2)u=f
\end{equation}
 supported in the time interval $[-2,2]$. Furthermore, let $b_j$ be an admissible $\mathcal{X}_T^{2s}$ frequency envelope for $u$ (here we don't assume that the formula is necessarily given explicitly by (\ref{env1})).  Then for $j>0$ we have, 
\\
\\
a) If $0<s<1$, then
\begin{equation}
\|\tilde{P}_j\tilde{S}_{2j}D_t^s(|u|^{2\sigma})\|_{L_t^2L_x^2}\lesssim b_j\|u\|_{\mathcal{S}_T^1}^{2\sigma-1}(\|u\|_{\mathcal{X}_T^{2s}}+\|f\|_{\mathcal{S}_T^0}).
\end{equation}
\\
\\
b) If $1\leq s<2\sigma$ and $0<\delta\ll 1$, then 
\begin{equation}
\|\tilde{P}_j\tilde{S}_{2j}D_t^s(|u|^{2\sigma})\|_{L_t^2L_x^2}\lesssim_{\delta} b_j\Lambda (\|u\|_{\mathcal{X}_T^{2s}}+\|f\|_{\mathcal{S}_T^0}+\|f\|_{Z_{\infty}^{s-1+\delta}})
\end{equation}
where 
\begin{equation}
\Lambda:=(\|u\|_{\mathcal{S}_T^{1+\delta}}+\|f\|_{\mathcal{S}_T^{\delta}})^{2\sigma-1}\Lambda_0
\end{equation}
and $\Lambda_0$ is some polynomial in $\|u\|_{\mathcal{S}_T^{1+\delta}}+\|f\|_{\mathcal{S}_T^{\delta}}$.
\end{lemma}
\begin{remark}
We only prove the above estimate  for $\tilde{P}_j\tilde{S}_{2j}D_t^s(|u|^{2\sigma})$ in $L_t^2L_x^2$. Although the estimate is almost certainly true for a suitable range of $p\geq 1$ in $L_t^pL_x^2$, we do not pursue this here, so as not to further complicate the argument (specifically, the proof of b)).
\end{remark}
\begin{remark}
We do not claim that the factors of $\|f\|_{\mathcal{S}_T^0}$, $\|f\|_{Z_{\infty}^{s-1+\delta}}$ and $\Lambda$ that appear in the estimate are in any way optimal (in fact, in many instances in the below estimates, they arise in relatively crude ways). We opted not to carefully optimize the inequality because it will not affect the range of $s$ for which \Cref{Etime1} holds, and also  because the current form of \Cref{timemoserest} can be more easily applied to establish \Cref{keyestimate1}.
\end{remark}
\begin{proof}
a) For notational convenience, we will sometimes write $F(z)=|z|^{2\sigma-2}\overline{z}$ and $P_{<j}u=u_{<j}$. Now, for each $j>0$, we write 
\begin{equation}\label{Moser1}
\begin{split}
D_t^s\tilde{S}_{2j}P_j|u|^{2\sigma}&=D_t^s\tilde{S}_{2j}P_j|P_{<j}u|^{2\sigma}-D_t^s\tilde{S}_{2j}P_j(|P_{<j}u|^{2\sigma}-|u|^{2\sigma})
\\
&=D_t^s\tilde{S}_{2j}P_j|P_{<j}u|^{2\sigma}+2\sigma\RE\int_{0}^{1}P_j\tilde{S}_{2j}D_t^s(F(y(\theta))P_{\geq j}u)d\theta
\end{split}
\end{equation}
where 
\begin{equation}
y(\theta):=\theta u+(1-\theta)P_{<j}u.
\end{equation}
For the first term, interpolating gives
\begin{equation}\label{Moserlow}
\begin{split}
\|D_t^s\tilde{S}_{2j}P_j|P_{<j}u|^{2\sigma}\|_{L_t^{2}L_x^2}&\lesssim \|\tilde{S}_{2j}P_j|P_{<j}u|^{2\sigma}\|_{L_t^{2}L_x^2}^{1-s}\|\tilde{S}_{2j}P_j(F(u_{<j})P_{<j}u_t)\|_{L_t^{2}L_x^2}^{s}.
\end{split}
\end{equation}
By expanding $u_t$ in the second factor, we obtain
\begin{equation}\label{twoterms}
\begin{split}
\|\tilde{S}_{2j}P_j(F(u_{<j})P_{<j}u_t)\|_{L_t^2L_x^2}&\lesssim \|\tilde{S}_{2j}P_j(F(u_{<j})P_{<j}u_{xx})\|_{L_t^2L_x^2}
\\
&+\|\tilde{S}_{2j}P_j(F(u_{<j})P_{<j}f)\|_{L_t^2L_x^2}.
\end{split}
\end{equation}
We expand the first term in (\ref{twoterms}) using the Littlewood-Paley trichotomy. Then Bernstein's inequality and \Cref{Holder} yields 

\begin{equation}\label{LWT}
\begin{split}
&\|\tilde{S}_{2j}P_j(F(u_{<j})P_{<j}u_{xx})\|_{L_t^2L_x^2}
\\
&\lesssim \|P_{<j-4}F(u_{<j})\tilde{P}_ju_{xx}\|_{L_t^2L_x^2}+\|P_{\geq j-4}F(u_{<j})P_{<j}u_{xx}\|_{L_t^2L_x^2}
\\
&\lesssim b_j2^{2j(1-s)}\|u\|_{\mathcal{S}_T^1}^{2\sigma-1}\|u\|_{\mathcal{X}_T^{2s}}+2^{2j(1-s)}2^{-\delta j}\|D_x^{\delta}F(u_{<j})\|_{L_t^{\infty}L_x^{\infty}}\|D_x^{2s}u\|_{L_t^{\infty}L_x^2}
\\
&\lesssim b_j2^{2j(1-s)}\|u\|_{\mathcal{S}_T^1}^{2\sigma-1}\|u\|_{\mathcal{X}_T^{2s}}.
\end{split}
\end{equation}
For the second term in (\ref{twoterms}), we obtain (by taking $2^{2j(s-1)}\lesssim b_j$) 
\begin{equation}
\begin{split}
\|\tilde{S}_{2j}P_j(F(u_{<j})P_{<j}f)\|_{L_t^2L_x^2}&\lesssim b_j2^{2j(1-s)}\|u\|_{\mathcal{S}_T^1}^{2\sigma-1}\|f\|_{L_t^{\infty}L_x^2}
\end{split}
\end{equation}
and so by Bernstein, the estimate (\ref{Moserlow}) becomes
\begin{equation}
\begin{split}
\|D_t^s\tilde{S}_{2j}P_j|P_{<j}u|^{2\sigma}\|_{L_t^{2}L_x^2}&\lesssim 2^{2js(1-s)}[b_j\|u\|_{\mathcal{S}_T^1}^{2\sigma-1}\|u\|_{\mathcal{X}_T^{2s}}+b_j\|u\|_{\mathcal{S}_T^1}^{2\sigma-1}\|f\|_{L_t^{\infty}L_x^2}]^s\|\tilde{S}_{2j}P_j|P_{<j}u|^{2\sigma}\|_{L_t^{2}L_x^2}^{1-s}
\\
&\lesssim [b_j\|u\|_{\mathcal{S}_T^1}^{2\sigma-1}\|u\|_{\mathcal{X}_T^{2s}}+b_j\|u\|_{\mathcal{S}_T^1}^{2\sigma-1}\|f\|_{L_t^{\infty}L_x^2}]^s\|D_t^s\tilde{S}_{2j}P_j|P_{<j}u|^{2\sigma}\|_{L_t^{2}L_x^2}^{1-s}.
\end{split}
\end{equation}
Hence,
\begin{equation}
\begin{split}
\|D_t^s\tilde{S}_{2j}P_j|P_{<j}u|^{2\sigma}\|_{L_t^{2}L_x^2}&\lesssim b_j\|u\|_{\mathcal{S}_T^1}^{2\sigma-1}\|u\|_{\mathcal{X}_T^{2s}}+b_j\|u\|_{\mathcal{S}_T^1}^{2\sigma-1}\|f\|_{L_t^{\infty}L_x^2}.
\end{split}
\end{equation}
For the second term in (\ref{Moser1}), using that $2^{2js}\|P_{\geq j}u\|_{L_t^2L_x^2}\lesssim \|D_x^{2s}P_{\geq j}u\|_{L_t^2L_x^2}$ and \Cref{Holder} leads to the estimate,  
\begin{equation}
\begin{split}
&\|P_j\tilde{S}_{2j}D_t^s(F(y(\theta))P_{\geq j}u)\|_{L_t^2L_x^2}
\\
&\lesssim 2^{2js}\|P_j(F(y(\theta))P_{\geq j}u)\|_{L_t^2L_x^2}
\\
&\lesssim b_j\|\langle D_x\rangle^{\delta}F(y(\theta))\|_{L_t^{\infty}L_x^{\infty}}\|u\|_{\mathcal{X}_T^{2s}}
\\
&\lesssim b_j\|u\|_{\mathcal{S}_T^1}^{2\sigma-1}\|u\|_{\mathcal{X}_T^{2s}}.
\end{split}
\end{equation}
Hence, by Minkowski's inequality,
\begin{equation}
2\sigma\|\RE\int_{0}^{1}P_j\tilde{S}_{2j}D_t^s(F(y(\theta))P_{\geq j}u)d\theta\|_{L_t^2L_x^2}\lesssim b_j\|u\|_{\mathcal{S}_T^1}^{2\sigma-1}\|u\|_{\mathcal{X}_T^{2s}}.
\end{equation}
Combining everything shows that
\begin{equation}
\|D_t^s\tilde{S}_{2j}P_j|u|^{2\sigma}\|_{L_t^{2}L_x^2}\lesssim b_j\|u\|_{\mathcal{S}_T^1}^{2\sigma-1}\|u\|_{\mathcal{X}_T^{2s}}+b_j\|u\|_{\mathcal{S}_T^1}^{2\sigma-1}\|f\|_{L_t^{\infty}L_x^2}.
\end{equation}
This proves part a).
\\
\\
Next, we prove part b). By commuting through the temporal projection, we obtain
\begin{equation}\label{twoterms3}
\begin{split}
\|\tilde{S}_{2j}P_jD_t^s(|u|^{2\sigma})\|_{L_t^2L_x^2}&\lesssim \|\tilde{S}_{2j}D_t^{s-1}(\tilde{S}_{<2j}(|u|^{2\sigma-2}\overline{u})\partial_t\tilde{S}_{2j}u)\|_{L_t^2L_x^2}
\\
&+\|\tilde{S}_{2j}D_t^{s-1}(\tilde{S}_{\geq 2j}(|u|^{2\sigma-2}\overline{u})\partial_tu)\|_{L_t^2L_x^2}.
\end{split}
\end{equation}
The first term in (\ref{twoterms3}) can be estimated by Bernstein's inequality to obtain 
\begin{equation}
\begin{split}
\|\tilde{S}_{2j}D_t^{s-1}(\tilde{S}_{<2j}(|u|^{2\sigma-2}\overline{u})\partial_t\tilde{S}_{2j}u)\|_{L_t^2L_x^2}&\lesssim \|u\|^{2\sigma-1}_{\mathcal{S}_T^1}\|D_t^{s-1}\partial_t\tilde{S}_{2j}u\|_{L_t^2L_x^2}.
\end{split}
\end{equation}
Then writing 
\begin{equation}
\|D_t^{s-1}\partial_t\tilde{S}_{2j}u\|_{L_t^2L_x^2}\sim \|D_t^{s-1}\partial_t\tilde{S}_{2j}P_{\leq 0}u\|_{L_t^2L_x^2}+\left(\sum_{k>0}\|D_t^{s-1}\partial_t\tilde{S}_{2j}P_ku\|_{L_t^2L_x^2}^2\right)^{\frac{1}{2}}
\end{equation}
and requiring $b_j\geq 2^{-j\delta}$, applying \Cref{spacetimeconv} and then square summing over $k$ yields
\begin{equation}
\begin{split}
\|D_t^{s-1}\partial_t\tilde{S}_{2j}u\|_{L_t^2L_x^2}&\lesssim \|D_x^{2s}\tilde{S}_{2j}P_ju\|_{L_t^2L_x^2}+2^{-j\delta}\|\langle D_t\rangle^{s-1+\delta}\tilde{S}_{2j}f\|_{L_t^2L_x^2}
\\
&\lesssim b_j\|u\|_{\mathcal{X}_T^{2s}}+b_j\|f\|_{Z_{\infty}^{s-1+\delta}}.
\end{split}
\end{equation}
To estimate the second term in (\ref{twoterms3}), we have two cases:
\\
\\
If $1\leq s\leq\sigma+\frac{1}{2}$, we obtain from the equation, 
\begin{equation}\label{moreterms}
\begin{split}
&\|\tilde{S}_{2j}D_t^{s-1}(\tilde{S}_{\geq 2j}(|u|^{2\sigma-2}\overline{u})\partial_tu)\|_{L_t^2L_x^2}
\\
&\lesssim \|\tilde{S}_{2j}D_t^{s-1}(\tilde{S}_{\geq 2j}(|u|^{2\sigma-2}\overline{u})\partial_x^2u)\|_{L_t^2L_x^2}+\|\tilde{S}_{2j}D_t^{s-1}(\tilde{S}_{\geq 2j}(|u|^{2\sigma-2}\overline{u})f)\|_{L_t^2L_x^2}.
\end{split}
\end{equation}
By H\"older and Bernstein's inequality, Sobolev embedding and \Cref{Holder}, the first term can be estimated by
\begin{equation}
\begin{split}
\|\tilde{S}_{2j}D_t^{s-1}(\tilde{S}_{\geq 2j}(|u|^{2\sigma-2}\overline{u})\partial_x^2u)\|_{L_t^2L_x^2}&\lesssim 2^{-j\delta}\|D_t^{s-1+\frac{\delta}{2}}(|u|^{2\sigma-2}\overline{u})\|_{L_t^{\frac{1}{s-1}}L_x^{\frac{1}{s-1}}}\|\partial_x^2u\|_{L_t^{\infty}L_x^{\frac{2}{3-2s}}}
\\
&\lesssim 2^{-j\delta}\|\langle D_t\rangle^{\frac{s-1+\delta}{2\sigma-1}}u\|_{L_t^{\frac{2\sigma-1}{s-1}}L_x^{\frac{2\sigma-1}{s-1}}}^{2\sigma-1}\|\partial_x^2u\|_{L_t^{\infty}L_x^{\frac{2}{3-2s}}}
\\
&\lesssim 2^{-j\delta}\|\langle D_t\rangle^{\frac{s-1+\delta}{2\sigma-1}}u\|_{L_t^{\frac{2\sigma-1}{s-1}}L_x^{\frac{2\sigma-1}{s-1}}}^{2\sigma-1}\|u\|_{L_t^{\infty}H_x^{s+1}}.
\end{split}
\end{equation}
Applying \Cref{cor to spacetimeconv} gives 
\begin{equation}
\|\langle D_t\rangle^{\frac{s-1+\delta}{2\sigma-1}}u\|_{L_t^{\frac{2\sigma-1}{s-1}}L_x^{\frac{2\sigma-1}{s-1}}}^{2\sigma-1}\lesssim \|\langle D_x\rangle^{\frac{2s-2+4\delta}{2\sigma-1}}u\|_{L_t^{\frac{2\sigma-1}{s-1}}L_x^{\frac{2\sigma-1}{s-1}}}^{2\sigma-1}+\|\langle D_t\rangle^{\frac{s-1+2\delta}{2\sigma-1}-1}f\|_{L_t^{\frac{2\sigma-1}{s-1}}L_x^{\frac{2\sigma-1}{s-1}}}^{2\sigma-1}.
\end{equation}
Since $\frac{2s-2}{2\sigma-1}\leq 1-(\frac{1}{2}-\frac{s-1}{2\sigma-1})$, when $s\leq \sigma+\frac{1}{2}$, we have by Sobolev embedding in the spatial variable,
\begin{equation} \label{interpolation}
\|\langle D_x\rangle^{\frac{2s-2+4\delta}{2\sigma-1}}u\|_{L_t^{\frac{2\sigma-1}{s-1}}L_x^{\frac{2\sigma-1}{s-1}}}^{2\sigma-1}\|u\|_{L_t^{\infty}H_x^{s+1}}\lesssim \|u\|_{\mathcal{S}_T^{1+c\delta}}^{2\sigma-1}\|u\|_{L_t^{\infty}H_x^{2s}}
\end{equation}
for some fixed constant $c>0$.
\\
\\
 Next, applying Sobolev embedding in the time variable, and using the inequality $\|g\|_{L_x^pL_t^q}\lesssim \|g\|_{L_t^qL_x^p}$ when $p\geq q$, we also obtain 
\begin{equation}
\begin{split}
\|\langle D_t\rangle^{\frac{s-1+2\delta}{2\sigma-1}-1}f\|_{L_t^{\frac{2\sigma-1}{s-1}}L_x^{\frac{2\sigma-1}{s-1}}}&\lesssim \|f\|_{L_x^{\frac{2\sigma-1}{s-1}}L_t^2}\lesssim \|f\|_{L_t^2L_x^{\frac{2\sigma-1}{s-1}}}
\\
&\lesssim \|f\|_{\mathcal{S}_T^0}
\end{split}
\end{equation}
and so, the first term in (\ref{moreterms}) can be controlled by (after possibly relabelling $\delta$),
\begin{equation}
2^{-j\delta}(\|u\|_{\mathcal{S}_T^{1+c\delta}}+\|f\|_{\mathcal{S}_T^0})^{2\sigma-1}\|u\|_{L_t^{\infty}H_x^{2s}}\lesssim b_j\Lambda\|u\|_{L_t^{\infty}H_x^{2s}}\lesssim b_j\Lambda\|u\|_{\mathcal{X}_T^{2s}}.
\end{equation}
For the second term in (\ref{moreterms}), we simply have by Bernstein, and \Cref{Holder} and \Cref{cor to spacetimeconv},
\begin{equation}
\begin{split}
\|\tilde{S}_{2j}D_t^{s-1}(\tilde{S}_{\geq 2j}(|u|^{2\sigma-2}\overline{u})f)\|_{L_t^2L_x^2}&\lesssim 2^{-j\delta}\|D_t^{s-1+\frac{\delta}{2}}(|u|^{2\sigma-2}\overline{u})\|_{L_t^{\frac{4}{2\sigma-1}}L_x^{\frac{4}{2\sigma-1}}}\|f\|_{L_t^4L_x^{\frac{4}{3-2\sigma}}}
\\
&\lesssim 2^{-j\delta}\|\langle D_t\rangle^{\frac{1}{2}+\frac{\delta}{2}}u\|_{L_t^4L_x^4}^{2\sigma-1}\|f\|_{\mathcal{S}_T^0}
\\
&\lesssim 2^{-j\delta}(\|u\|_{\mathcal{S}_T^{1+\delta}}^{2\sigma-1}+\|f\|_{\mathcal{S}_T^0}^{2\sigma-1})\|f\|_{\mathcal{S}_T^0}
\\
&\lesssim b_j\Lambda\|f\|_{\mathcal{S}_T^0}.
\end{split}
\end{equation}
This handles the case $1\leq s\leq \sigma+\frac{1}{2}$.
\\
\\
Next, suppose $\sigma+\frac{1}{2}< s<2\sigma$. By Bernstein's inequality, 
\begin{equation}
\begin{split}
\|\tilde{S}_{2j}D_t^{s-1}(\tilde{S}_{\geq 2j}(|u|^{2\sigma-2}\overline{u})\partial_tu)\|_{L_t^2L_x^2}&\lesssim 2^{-j\delta}\|D_t^{s-1+\frac{\delta}{2}}(|u|^{2\sigma-2}\overline{u})\|_{L_t^{\frac{2}{2\sigma-1}}L_x^{\frac{2}{2\sigma-1}}}\|\partial_tu\|_{L_t^{\frac{1}{1-\sigma}}L_x^{\frac{1}{1-\sigma}}}
\\
&\lesssim b_j\|D_t^{s-1+\frac{\delta}{2}}(|u|^{2\sigma-2}\overline{u})\|_{L_t^{\frac{2}{2\sigma-1}}L_x^{\frac{2}{2\sigma-1}}}\|\partial_tu\|_{L_t^{\frac{1}{1-\sigma}}L_x^{\frac{1}{1-\sigma}}}.
\end{split}
\end{equation}
Using \Cref{Holder} and then \Cref{cor to spacetimeconv}, we estimate, 
\begin{equation}
\begin{split}
\|D_t^{s-1+\frac{\delta}{2}}(|u|^{2\sigma-2}\overline{u})\|_{L_t^{\frac{2}{2\sigma-1}}L_x^{\frac{2}{2\sigma-1}}}&\lesssim \|\langle D_t\rangle^{\frac{s-1+\delta}{2\sigma-1}}u\|_{L_t^2L_x^2}^{2\sigma-1}
\\
&\lesssim \|P_{\leq 0}\langle D_t\rangle^{\frac{s-1+\delta}{2\sigma-1}}u\|_{L_t^2L_x^2}^{2\sigma-1}+\left(\sum_{j>0}\|\langle D_t\rangle^{\frac{s-1+\delta}{2\sigma-1}}P_ju\|_{L_t^2L_x^2}^2\right)^{\frac{1}{2}(2\sigma-1)}
\\
&\lesssim \|\langle D_x\rangle^{\frac{2s-2+2\delta}{2\sigma-1}}u\|_{L_t^2L_x^2}^{2\sigma-1}+\|f\|_{\mathcal{S}_T^{0}}^{2\sigma-1}.
\end{split}
\end{equation}
Furthermore, we have by Sobolev embedding and the equation, 
\begin{equation}
\begin{split}
\|\partial_tu\|_{L_t^{\frac{1}{1-\sigma}}L_x^{\frac{1}{1-\sigma}}}&\lesssim \|\langle D_x\rangle^{\sigma+\frac{3}{2}}u\|_{L_t^{\infty}L_x^2}+\|\langle D_x\rangle ^{\sigma-\frac{1}{2}}f\|_{L_t^{\infty}L_x^{2}}.
\end{split}
\end{equation}
Hence, we obtain
\begin{equation}
\begin{split}
&b_j\|D_t^{s-1+\frac{\delta}{2}}(|u|^{2\sigma-2}\overline{u})\|_{L_t^{\frac{2}{2\sigma-1}}L_x^{\frac{2}{2\sigma-1}}}\|\partial_tu\|_{L_t^{\frac{1}{1-\sigma}}L_x^{\frac{1}{1-\sigma}}}
\\
&\lesssim b_j(\|\langle D_x\rangle^{\sigma+\frac{3}{2}}u\|_{L_t^{\infty}L_x^2}+\|\langle D_x\rangle ^{\sigma-\frac{1}{2}}f\|_{L_t^{\infty}L_x^{2}})\|\langle D_x\rangle^{\frac{2s-2+2\delta}{2\sigma-1}}u\|_{L_t^2L_x^2}^{2\sigma-1}+\Lambda b_j(\|u\|_{L_t^{\infty}H_x^{2s}}+\|f\|_{Z_{\infty}^{s-1+\delta}}).
\end{split}
\end{equation}
To control the first term, interpolating each factor between $L_t^{\infty}H_x^{2s}$ and $L_t^{\infty}H_x^1$ shows that
\begin{equation}
\|\langle D_x\rangle^{\frac{2s-2+2\delta}{2\sigma-1}}u\|_{L_t^2L_x^2}^{2\sigma-1}\|\langle D_x\rangle^{\sigma+\frac{3}{2}}u\|_{L_t^{\infty}L_x^2}\lesssim \|u\|_{\mathcal{S}_T^{1+\delta}}^{2\sigma-1}\|u\|_{L_t^{\infty}H_x^{2s}}.
\end{equation}
For the second term, interpolating  the $\langle D_x\rangle^{\frac{2s-2+2\delta}{2\sigma-1}}u$ factor between $L_t^{\infty}H_x^1$ and $L_t^{\infty}H_x^{2s}$ and the $\langle D_x\rangle^{\sigma-\frac{1}{2}}f$ factor between $L_t^{\infty}L_x^2$ and $L_t^{\infty}H_x^{2s-2+\delta}$ and using that $s>\sigma+\frac{1}{2}$ leads to
\begin{equation}
\|\langle D_x\rangle^{\frac{2s-2+2\delta}{2\sigma-1}}u\|_{L_t^2L_x^2}^{2\sigma-1}\|\langle D_x\rangle ^{\sigma-\frac{1}{2}}f\|_{L_t^{\infty}L_x^{2}}\lesssim \Lambda (\|u\|_{L_t^{\infty}H_x^{2s}}+\|f\|_{Z_{\infty}^{s-1+\delta}}).
\end{equation}
Now, collecting all of the estimates and using that $\|u\|_{L_t^{\infty}H_x^{2s}}\lesssim \|u\|_{\mathcal{X}_T^{2s}}$ completes the proof. 
\end{proof}
Finally, we use \Cref{timemoserest} to establish \Cref{timemoser2}.
\begin{proof}
First, it is straightforward to verify that $b_j^2$ is a $\mathcal{X}_T^{2s}$ frequency envelope for $\tilde{v}$ in the sense that $b_j^2$ satisfies property (\ref{property1}) and is slowly varying. Next, we expand 
\begin{equation}
(i\partial_t+\partial_x^2)\tilde{v}=i\partial_t\tilde{\eta}v+\tilde{\eta}(i\partial_t+\partial_x^2)v:=f.
\end{equation}
Using an argument similar to what was done to estimate (\ref{commutet}) and applying \Cref{cor to spacetimeconv}, it is straightforward to verify $\|f\|_{\mathcal{S}_T^{\delta}}+\|f\|_{Z_{\infty}^{s-1+\delta}}\lesssim \epsilon+\|v\|_{\mathcal{X}_T^{2s}}$, and so the conclusion immediately follows from \Cref{timemoserest}.
\end{proof}
\section{Well-posedness at high regularity}\label{Sec6}
In this section, we aim to prove \Cref{high reg theorem}. We begin by studying a suitable regularized equation.
\subsection{Well-posedness of a regularized equation}

 Since there is an apparent limit to the possible regularity of solutions to \eqref{gDNLS}, we construct $H_x^{2s}$ solutions as limits of smooth solutions to an appropriate regularized approximate equation. Like in the previous section $\eta$ will denote a time-dependent cutoff with $\eta=1$ on $[-1,1]$ with support in $(-2,2)$. To construct the requisite solutions, we need the following lemma: 
\begin{lemma}\label{reglem}
Let $2-\sigma<2s<4\sigma$. Let $2s\geq\alpha>\max\{2-\sigma,2s-1\}$. Then there is an $\epsilon>0$ such that for every $u_0\in H_x^{2s}$ with $\|u_0\|_{H_x^{\alpha}}\leq \epsilon$ and for all $j>0$, the regularized equation
\begin{equation}\label{regular}
\begin{cases}
&(i\partial_t+\partial_x^2)u=i \eta P_{<j}|u|^{2\sigma}\partial_xu,
\\
&u(0)=P_{<j}u_0,
\end{cases}
\end{equation}
 admits a global solution $u\in C^2(\mathbb{R};H^{\infty}_x)$. Moreover, we have the following bounds for $T=2$,
 \begin{equation}\label{Strichartzbound}
 \begin{split}
&\|u\|_{\mathcal{X}_T^{\alpha}\cap\mathcal{S}_T^{1+\delta}}\lesssim\epsilon,
\\
&\|(i\partial_t+\partial_x^2)u\|_{\mathcal{S}_T^{\delta}\cap Z_{\infty}^{s-1+\delta}}\lesssim \epsilon,
\end{split}
 \end{equation}
 where the implicit constant in the above inequality is independent of the parameter $j$ and where $0<\delta\ll 1$ is any small positive constant.
\end{lemma}
\begin{remark}
 The smallness assumption on the $H_x^{\alpha}$ norm of $u_0$ will turn out to be inconsequential (by $L_x^2$ subcriticality for \eqref{gDNLS}). This assumption is made for convenience to guarantee (\ref{Strichartzbound}). 
\end{remark}
Let us now construct solutions to (\ref{regular}). The first step is to construct solutions to an appropriate linear equation. For this, we have the following lemma.
\begin{lemma}\label{linearapproximation}
Let $\eta=\eta(t)$ be a smooth time-dependent cutoff with $\eta=1$ on $[-1,1]$ and with support in $(-2,2)$. Let $T>0$ and $v\in L_T^{2\sigma}L_x^{\infty}$. Let $u_0\in H_x^{2s}$. Then for each $j>0$, there exists a unique solution $w\in C([-T,T];H_x^{\infty})$ solving the equation
\begin{equation}\label{approx}
\begin{cases}
&\partial_tw=i\partial_x^2w+\eta P_{<j}|v|^{2\sigma}\partial_xw,
\\
&w(0)=P_{< j}u_0.
\end{cases}
\end{equation}
\end{lemma}
\begin{proof}
First, observe that for each $n>j$ a simple (iterated) application of the contraction mapping theorem in the closed subspace of $C([-T,T];L_x^2)$ consisting of functions whose spatial Fourier transform is supported on $[-2^n,2^n]$ gives rise to a solution $w^{(n)}\in C([-T,T];H_x^{\infty})$ to the following regularized linear equation,
\begin{equation}\label{approx3}
\begin{cases}
&\partial_tw^{(n)}=i\partial_x^2w^{(n)}+\eta P_{\leq n}(P_{<j}|v|^{2\sigma}\partial_xw^{(n)}),
\\
&w^{(n)}(0)=P_{<j}u_0.
\end{cases}
\end{equation}
We show that the sequence $w^{(n)}$ converges as $n\to\infty$ to some $w\in C([-T,T];H_x^{\infty})$ which solves (\ref{approx}). This follows in two stages, but is standard. First, for each integer $k\geq 0$, a standard energy estimate and Bernstein's inequality shows that $w^{(n)}$ satisfies the bound
\begin{equation}\label{uniformboundsapprox}
\|w^{(n)}\|_{C([-T,T];H_x^{k})}\lesssim \exp(2^{j(k+1)}\|v\|_{L_T^{2\sigma}L_x^{\infty}}^{2\sigma})\|P_{<j}u_0\|_{H_x^{k}}   
\end{equation}
where importantly, the bound is independent of $n$ (but can depend on $j$). Furthermore, a simple energy estimate in $L_x^2$ for differences of solutions $w^{(n)}-w^{(m)}$ to (\ref{approx3}) shows that the sequence $w^{(n)}$ is Cauchy in $C([-T,T];L_x^2)$ and thus converges to some $w\in C([-T,T];L_x^2)$. Interpolating against (\ref{uniformboundsapprox}) shows that in fact $w^{(n)}$ converges to some $w$ in $C([-T,T];H_x^{\infty})$ and that $w$ solves (\ref{approx}) in the sense of distributions, and furthermore that $w$ satisfies the bound (\ref{uniformboundsapprox}) for each $k\geq 0$.
\end{proof}
The next step in the proof of \Cref{reglem} is to construct the corresponding $C^2(\mathbb{R};H_x^{\infty})$ solution to (\ref{reglem}). For this purpose, consider the following iteration scheme,
\begin{equation}\label{iteration}
\begin{cases}
&(i\partial_t+\partial_x^2)u^{(n+1)}=i\eta P_{<j}|u^{(n)}|^{2\sigma}\partial_xu^{(n+1)},
\\
&u^{(n+1)}(0)=P_{<j}u_0,
\end{cases}
\end{equation}
with the initialization $u^{(0)}=0$. Thanks to \Cref{linearapproximation} it follows that for each $n$, there is a solution $u^{(n+1)}\in C([-2,2];H_x^{\infty})$ to the above equation. In particular, $u^{(n+1)}$ can be extended globally in time because for $|t|>2$, $u^{(n+1)}$ solves the linear Schr\"odinger equation.
\\
\\
Next, we have the following lemma  concerning the convergence of this iteration scheme, from which \Cref{reglem} is immediate. 
\begin{lemma}
Let $2-\sigma<2s<4\sigma$. Let $2s\geq\alpha>\max\{2-\sigma,2s-1\}$. Let $u_0\in H_x^{2s}$ and let $u^{(n+1)}$ be the corresponding $C(\mathbb{R};H_x^{\infty})$ solution to (\ref{iteration}). Then there is $\epsilon>0$ independent of $j$ such that if $\|u_0\|_{H_x^{\alpha}}\leq \epsilon$, then $u^{(n)}$ converges to some $u\in C(\mathbb{R};H_x^{\infty})$ solving (\ref{regular}). Furthermore, we have $u\in C^2(\mathbb{R};H_x^{\infty})$ and the bounds
\begin{equation}\label{strich2}
\begin{split}
&\|u\|_{\mathcal{X}_T^{\alpha}\cap\mathcal{S}_T^{1+\delta}}\lesssim\epsilon,
\\
&\|(i\partial_t+\partial_x^2)u\|_{\mathcal{S}_T^{\delta}\cap Z_{\infty}^{s-1+\delta}}\lesssim \epsilon.    
\end{split}    
\end{equation}
\end{lemma}
\begin{proof}
We begin by showing that $u^{(n+1)}$ satisfies the bounds
\begin{equation}\label{iteratebound}
\begin{split}
&\|u^{(n+1)}\|_{\mathcal{X}_T^{\alpha}\cap\mathcal{S}_T^{1+\delta}}\lesssim \epsilon,
\\
&\|(i\partial_t+\partial_x^2)u^{(n+1)}\|_{\mathcal{S}_T^{\delta}\cap Z_{\infty}^{s-1+\delta}}\lesssim \epsilon,
\end{split}
\end{equation}
for $T=2$ uniformly in $n$. Given the initialization $u^{(0)}=0$, we may make the inductive hypothesis that (\ref{iteratebound}) holds with $n+1$ replaced by $n$. Now, we prove the above two bounds for $u^{(n+1)}$.
\\
\\
We begin by showing $\|u^{(n+1)}\|_{\mathcal{X}_T^{\alpha}\cap\mathcal{S}_T^{1+\delta}}\lesssim\epsilon$. Indeed, it follows from the modification of the low regularity bounds outlined in \Cref{modifiedlemma} that for $2s\geq\alpha>2-\sigma$,
\begin{equation}\label{S2X}
\|u^{(n+1)}\|_{\mathcal{X}_T^{\alpha}\cap\mathcal{S}_T^{1+\delta}}\lesssim \|u^{(n+1)}\|_{\mathcal{X}_T^{\alpha}}.
\end{equation}
Then \Cref{keyestimate1} and the inductive hypothesis gives
\begin{equation}
\begin{split}
\|u^{(n+1)}\|_{\mathcal{X}_T^{\alpha}}^2&\lesssim \|u_0\|_{H_x^{\alpha}}^2+\epsilon^{2\sigma}(\|u^{(n+1)}\|_{\mathcal{X}_T^{\alpha}}^2+\|u^{(n+1)}\|_{\mathcal{S}_T^{1}}^2)+\epsilon^{2\sigma-1}\|u^{(n+1)}\|_{\mathcal{S}_T^1}\|u^{(n+1)}\|_{\mathcal{X}_T^{\alpha}}\|u^{(n)}\|_{\mathcal{X}_T^{\alpha}}
\\
&+\epsilon^{4\sigma-2}\|u^{(n+1)}\|_{\mathcal{S}_T^{1}}^2\|u^{(n)}\|_{\mathcal{X}_T^{\alpha}}^2, 
\end{split}
\end{equation}
and so,
\begin{equation}
\|u^{(n+1)}\|_{\mathcal{X}_T^{\alpha}}^2\lesssim \|u_0\|_{H_x^{\alpha}}^2+\epsilon^{2\sigma}\|u^{(n+1)}\|_{\mathcal{X}_T^{\alpha}}^2.
\end{equation}
From this, we deduce
\begin{equation}\label{Xbound}
\|u^{(n+1)}\|_{\mathcal{X}_T^{\alpha}}\lesssim\epsilon.    
\end{equation}
Next, we aim to verify the bound, 
\begin{equation}
\|(i\partial_t+\partial_x^2)u^{(n+1)}\|_{\mathcal{S}_T^{\delta}\cap Z_{\infty}^{s-1+\delta}}\lesssim \epsilon.  
\end{equation}
For this, we use the equation,
\begin{equation}
(i\partial_t+\partial_x^2)u^{(n+1)}=i\eta P_{<j}|u^{(n)}|^{2\sigma}\partial_xu^{(n+1)}.
\end{equation}
From \Cref{nonlinest} and (\ref{S2X}), we have
\begin{equation}
\|i\eta P_{<j}|u^{(n)}|^{2\sigma}\partial_xu^{(n+1)}\|_{\mathcal{S}_T^{\delta}\cap Z_{\infty}^{s-1+\delta}}\lesssim \epsilon^{2\sigma}\|u^{(n+1)}\|_{\mathcal{X}_T^{\alpha}}\lesssim \epsilon.
\end{equation}
This verifies the uniform in $n$ bound (\ref{iteratebound}).
\\
\\
Next, we show that that $u^{(n)}$ converges to $u\in C(\mathbb{R};L_x^2)$. Clearly it suffices to show (by the localization properties of $\eta$) that $u^{(n)}$ converges to $u\in C([-2,2];L_x^2)$.
\\
\\
We begin by estimating the $L_x^2$ norm of $u^{(n+1)}(t)-u^{(n)}(t)$ for $|t|\leq 2$. Indeed, we see that $u^{(n+1)}-u^{(n)}$ satisfies the equation,
\begin{equation}
\begin{cases}
&(i\partial_t+\partial_x^2)(u^{(n+1)}-u^{(n)})=i\eta P_{<j}|u^{(n)}|^{2\sigma}\partial_x(u^{(n+1)}-u^{(n)})+i\eta P_{<j}(|u^{(n)}|^{2\sigma}-|u^{(n-1)}|^{2\sigma})\partial_xu^{(n)},
\\
&(u^{(n+1)}-u^{(n)})(0)=0.
\end{cases}
\end{equation}
A simple energy estimate shows that for each $-2\leq T\leq 2$
\begin{equation}\label{diffestimates}
\begin{split}
\|u^{(n+1)}-u^{(n)}\|_{L_T^{\infty}L_x^2}^2&\lesssim \|u^{(n)}\|_{S_T^1}\|u^{(n)}\|_{L_T^{\infty}L_x^{\infty}}^{2\sigma-1}\|u^{(n+1)}-u^{(n)}\|_{L_T^{\infty}L_x^2}^2
\\
&+\|u^{(n)}\|_{S_T^1}(\|u^{(n)}\|_{L_T^{\infty}L_x^{\infty}}^{2\sigma-1}+\|u^{(n-1)}\|_{L_T^{\infty}L_x^{\infty}}^{2\sigma-1})\|u^{(n)}-u^{(n-1)}\|_{L_T^{\infty}L_x^2}\|u^{(n+1)}-u^{(n)}\|_{L_T^{\infty}L_x^2}   
\end{split}
\end{equation}
where all the implicit constants are independent of $j$. Using (\ref{iteratebound}) and Cauchy Schwarz, we obtain
\begin{equation}
\begin{split}
\|u^{(n+1)}-u^{(n)}\|_{L_T^{\infty}L_x^2}^2&\leq \frac{1}{4}\|u^{(n+1)}-u^{(n)}\|_{L_T^{\infty}L_x^2}^2+\frac{1}{4}\|u^{(n)}-u^{(n-1)}\|_{L_T^{\infty}L_x^2}^2. 
\end{split}
\end{equation}
From this, one obtains
\begin{equation}
\begin{split}
\|u^{(n+1)}-u^{(n)}\|_{L_T^{\infty}L_x^2}^2&\leq \frac{1}{2}\|u^{(n)}-u^{(n-1)}\|_{L_T^{\infty}L_x^2}^2. 
\end{split}
\end{equation}
Hence, we see that $u^{(n)}$ converges to $u$ in $C([-2,2];L_x^2)$. By a simple energy estimate, and Bernstein's inequality, it is straightforward to verify that for each integer $k\geq 0$, we have the uniform (in $n$) bound
\begin{equation}\label{unifenerg}
\|u^{(n+1)}\|_{C([-2,2];H_x^{k})}\lesssim \exp(2^{j(k+1)}\|u^{(n)}\|_{L_T^{2\sigma}L_x^{\infty}}^{2\sigma})\|P_{<j}u_0\|_{H_x^{k}}\lesssim_j\|u_0\|_{H_x^{2s}}.   
\end{equation}    
Hence, by interpolating against (\ref{unifenerg}), we see that $u^{(n)}$  converges to $u$ in $C([-2,2];H_x^{\infty})$. By differentiating the equation in time, we find $u\in C^2([-2,2];H_x^{\infty})$. 
\\
\\
 It remains to show \eqref{strich2}. Since $u^{(n)}\to u$ in $C([-2,2];H_x^{\infty})$, the $\mathcal{X}_T^{\alpha}\cap\mathcal{S}_T^{1+\delta}$ bound follows immediately from (\ref{iteratebound}). For the remaining estimate, we may clearly control
\begin{equation}
(i\partial_t+\partial_x^2)u=i\eta P_{<j}|u|^{2\sigma}\partial_xu    
\end{equation}
in $\mathcal{S}_T^{\delta}\cap Z_{\infty}^{s-1+\delta}$ by (after possibly slightly enlarging $\delta$)
\begin{equation}
\|u\|_{\mathcal{X}_T^{\alpha}\cap\mathcal{S}_T^{1+\delta}}^{2\sigma}+\|i\eta P_{<j}|u|^{2\sigma}\partial_xu\|_{Z_{\infty}^{s-1+\delta}}\lesssim \epsilon+\|iP_{<j}\eta|u|^{2\sigma}\partial_xu\|_{Z_{\infty}^{s-1+\delta}}.    
\end{equation}
From \Cref{nonlinest}, we have 
\begin{equation}\label{auxbound2}
\|i\eta P_{<j}|u|^{2\sigma}\partial_xu\|_{Z_{\infty}^{s-\frac{3}{2}+\delta}}\lesssim \epsilon.   
\end{equation}
Then applying \Cref{nonlinest} again, using (\ref{auxbound2}) then gives
\begin{equation}
\|i\eta P_{<j}|u|^{2\sigma}\partial_xu\|_{Z_{\infty}^{s-1+\delta}}\lesssim \epsilon.   
\end{equation}
This completes the proof. 
\end{proof}
\begin{remark}
Note that at this point, we haven't said anything about the behavior of (\ref{regular}) as $j\to \infty$. For this, we will again need the uniform bounds from \Cref{keyestimate1}.
\end{remark}
\subsection{Well-posedness for the full equation}
In this section, we prove the local well-posedness of \eqref{gDNLS} in $H_x^{2s}$ for $2-\sigma<2s<4\sigma$. 
\\
\\
Indeed, let $u_0\in H_x^{2s}$ and let $2-\sigma<\alpha\leq 2s$. By rescaling (recalling the problem is $L_x^2$ subcritical), we may assume without loss of generality that $\|u_0\|_{H_x^{\alpha}}\leq\epsilon$ for some $\epsilon>0$ sufficiently small, and construct the corresponding $H_x^{2s}$ solution on the time interval $[-1,1]$. For $2-\sigma<2s\leq \frac{3}{2}$, we construct the solution in the Strichartz type space $\mathcal{X}_T^{2s}\cap \mathcal{S}_T^{1+\delta}$, where $0<\delta \ll 1$ is any sufficiently small positive constant. When $s>\frac{3}{2}$, the extra $\mathcal{S}_T^{1+\delta}$ component is, of course, redundant, thanks to Sobolev embedding. 
\\
\\
We will realize $H_x^{2s}$ well-posed solutions as (restrictions to the interval $[-1,1]$ of) limits of smooth solutions to the regularized equation (\ref{regular}). To establish this, we have the following lemma.
\begin{lemma}\label{reglem1}
Let $2-\sigma<2s<4\sigma$. Let $2s\geq\alpha>\max\{2-\sigma,2s-1\}$. Then there is an $\epsilon>0$ such that for every $u_0\in H_x^{2s}$ with $\|u_0\|_{H_x^{\alpha}}\leq \epsilon$, the time-truncated equation,
\begin{equation}\label{regular1}
\begin{cases}
&(i\partial_t+\partial_x^2)u=i \eta |u|^{2\sigma}\partial_xu,
\\
&u(0)=u_0,
\end{cases}
\end{equation}
 admits a global solution $u\in C^2(\mathbb{R};H^{\infty}_x)$. Moreover, we have the following bounds for $T=2$,
 \begin{equation}\label{Strichartzbound1}
 \begin{split}
&\|u\|_{\mathcal{X}_T^{\alpha}\cap\mathcal{S}_T^{1+\delta}}\lesssim\epsilon,
\\
&\|(i\partial_t+\partial_x^2)u\|_{\mathcal{S}_T^{\delta}\cap Z_{\infty}^{s-1+\delta}}\lesssim \epsilon,
\end{split}
 \end{equation}
 and also
 \begin{equation}\label{highubound}
\|u\|_{\mathcal{X}_T^{2s}}^2\lesssim \frac{1}{1-C\epsilon^{2\sigma}}\|u_0\|_{H_x^{2s}}^2,   
\end{equation}
 where $C>0$ is some universal constant.
\end{lemma}
\begin{proof}
If $\epsilon$ is small enough, thanks to \Cref{reglem}, for each $j>0$, there is a smooth solution $u^{(j)}\in C^2(\mathbb{R};H_x^{\infty})$ to the equation,
\begin{equation}
\begin{cases}
&(i\partial_t+\partial_x^2)u^{(j)}=i\eta P_{<j}|u^{(j)}|^{2\sigma}u^{(j)}_x,
\\
&u^{(j)}(0)=P_{<j}u_0,
\end{cases}
\end{equation}
satisfying
\begin{equation}
\|u^{(j)}\|_{\mathcal{X}_T^{\alpha}\cap\mathcal{S}_T^{1+\delta}}+\|(i\partial_t+\partial_x^2)u^{(j)}\|_{\mathcal{S}_T^{\delta}\cap Z_{\infty}^{s-1+\delta}}\lesssim \epsilon
\end{equation}
uniformly in $j$.
Now, define for $k>j$, $v^{(k,j)}:=u^{(k)}-u^{(j)}$. Then $v^{(k,j)}$ satisfies the equation, 
\begin{equation}
\begin{cases}
&(i\partial_t+\partial_x^2)v^{(k,j)}=i\eta P_{<k}|u^{(k)}|^{2\sigma}\partial_xv^{(k,j)}+i\eta P_{<k}(|u^{(k)}|^{2\sigma}-|u^{(j)}|^{2\sigma})\partial_xu^{(j)}+i\eta P_{j\leq \cdot<k}|u^{(j)}|^{2\sigma}\partial_xu^{(j)},
\\
&v^{(k,j)}(0)=P_{j\leq \cdot<k}u_0.
\end{cases}
\end{equation}
Multiplying by $-i\overline{v^{(k,j)}}$ taking real part and integrating over $\mathbb{R}$ and from $0$ to $t$ with $|t|\leq T$ leads to the simple energy estimate 
\begin{equation}
\begin{split}
\|v^{(k,j)}\|_{L_T^{\infty}L_x^2}^2&\lesssim \|P_{j\leq \cdot<k}u_0\|_{L_x^2}^2+(\|u^{(j)}\|_{\mathcal{S}_T^1}^{2\sigma-1}+\|u^{(k)}\|_{\mathcal{S}_T^1}^{2\sigma-1})\|u^{(j)}\|_{\mathcal{S}_T^1}\|v^{(k,j)}\|_{L_T^{\infty}L_x^2}^2
\\
&+\|u^{(k)}\|_{\mathcal{S}_T^1}^{2\sigma}\|v^{(k,j)}\|_{L_T^{\infty}L_x^2}^2+\|P_{j\leq \cdot<k}|u^{(j)}|^{2\sigma}\|_{L_T^{\infty}L_x^{2}}\|u^{(j)}\|_{\mathcal{S}_T^1}\|v^{(k,j)}\|_{L_T^{\infty}L_x^2}.
\end{split}
\end{equation}
Using the uniform in $j$ bound
\begin{equation}\label{Strich2}
\|u^{(j)}\|_{\mathcal{S}_T^{1+\delta}}\lesssim \epsilon  
\end{equation}
from \Cref{reglem} and Cauchy Schwarz gives
\begin{equation}\label{diffest}
\|v^{(k,j)}\|_{L_T^{\infty}L_x^2}^2\lesssim \|P_{j\leq \cdot<k}u_0\|_{L_x^2}^2+\|P_{j\leq \cdot<k}|u^{(j)}|^{2\sigma}\|_{L_T^{\infty}L_x^{2}}^2\|u^{(j)}\|_{\mathcal{S}_T^1}^2.
\end{equation}
Furthermore,
\begin{equation}
\begin{split}
\|P_{j\leq\cdot< k}|u^{(j)}|^{2\sigma}\|_{L_T^{\infty}L_x^2}\lesssim 2^{-j}\|u^{(j)}\|_{\mathcal{S}_T^1}^{2\sigma}.
\end{split}
\end{equation}
Hence, the right hand side of (\ref{diffest}) goes to zero as $j,k\to\infty$. Therefore, $u^{(j)}$ converges to some $u$ in $C([-2,2];L_x^2)$. On the other hand, thanks to the uniform (in $k$) bounds from the energy estimate \Cref{keyestimate1}, we obtain
\begin{equation}\label{highenvbd}
\begin{split}
\|P_ju^{(k)}\|_{\mathcal{X}_T^{2s}}^2&\lesssim a_j^2\|u_0\|_{H_x^{2s}}^2+[b_j^{(k)}]^2\epsilon^{2\sigma}\|u^{(k)}\|_{\mathcal{X}_T^{2s}}^2,
\end{split}
\end{equation}
where $b_j^{(k)}$ is a $\mathcal{X}_T^{2s}$ frequency envelope for $u^{(k)}$. Using that $\|u^{(k)}\|_{S_T^{1+\delta}}\lesssim\epsilon$, an argument similar to the low regularity well-posedness shows that for $\epsilon$ small enough, $a_j$ is a $\mathcal{X}_T^{2s}$ frequency envelope for $u^{(k)}$. Analogously to the low regularity argument, this can be used to show that $u^{(k)}\to u$ in $\mathcal{X}_T^{2s}$ and that $a_j$ is a $\mathcal{X}_T^{2s}$ frequency envelope for $u$ and that $u$ solves the time truncated equation,
\begin{equation}\label{timetrunc}
\begin{cases}
&i\partial_tu+u_{xx}=i\eta|u|^{2\sigma}u_x,
\\
&u(0)=u_0,
\end{cases}    
\end{equation}
in the sense of distributions. 
Moreover, by square summing over $j$ and passing to the limit in (\ref{highenvbd}), we obtain the uniform bound 
\begin{equation}
\|u\|_{\mathcal{X}_T^{2s}}^2\lesssim \frac{1}{1-C\epsilon^{2\sigma}}\|u_0\|_{H_x^{2s}}^2.    
\end{equation}
\end{proof}
Next, we establish local well-posedness for the full equation \eqref{gDNLS}.
\\
\\
For existence, we may rescale (using the $L_x^2$ subcriticality of the equation) to assume $u_0\in H_x^{2s}$ has sufficiently small data. Then we may construct a $\mathcal{X}_T^{2s}$ solution to $\eqref{gDNLS}$ on the time interval $[-1,1]$ by applying \Cref{reglem1} and restricting to $|t|\leq 1$.
\\
\\
For uniqueness, we consider the difference of two $H_x^{2s}$ solutions $u_1,u_2$ to \eqref{gDNLS} and obtain, by a standard energy estimate, the weak Lipschitz bound,
\begin{equation}\label{lip}
\|u_1-u_2\|_{L_T^{\infty}L_x^2}\lesssim_{\|u_1\|_{\mathcal{S}_T^1},\|u_2\|_{\mathcal{S}_T^1}} \|u_1(0)-u_2(0)\|_{L_x^2}.    
\end{equation}
for $T>0$. Among other things, this shows uniqueness in $C([-1,1];H_x^{2s})\cap \mathcal{S}_T^{1}$. 
\\
\\
For continuous dependence, again assume without loss of generality that $u_0$ has sufficiently small $H_x^{2s}$ norm. To show continuous dependence for the full equation \eqref{gDNLS}, it clearly suffices (by restricting to $T\leq 1$) to show that the data to solution map $u_0\in H_x^{2s}\mapsto u\in \mathcal{X}_{T=2}^{2s}\cap\mathcal{S}_{T=2}^{1+\delta}$ for the time-truncated equation (\ref{timetrunc}) is continuous. For this, let $u_0^n\in H_x^{2s}$ be a sequence of initial data converging to some $u_0$ in $H_x^{2s}$. Let $u^n$ and $u$ denote the corresponding $\mathcal{X}_{T=2}^{2s}\cap\mathcal{S}_{T=2}^{1+\delta}$ solutions to the time-truncated equation (\ref{timetrunc}), respectively. From the frequency envelope bound (\ref{highenvbd}) and an argument almost identical to the proof of continuous dependence at low regularity, it follows that 
\begin{equation}
\lim_{n\to\infty}\|u^n-u\|_{\mathcal{X}_{T=2}^{2s}\cap\mathcal{S}_{T=2}^{1+\delta}}=0.
\end{equation}
We omit the details. This finally completes the proof of \Cref{high reg theorem}. 
\section{Global well-posedness} 
Here, we complete the proof of \Cref{low reg theorem}. That is, we show that for $\frac{\sqrt{3}}{2}<\sigma<1$ and $1\leq 2s<4\sigma$,  \eqref{gDNLS} is globally well-posed in $H_x^{2s}$. The proof of local well-posedness in $H_x^{2s}$ for $1\leq 2s\leq\frac{3}{2}$ and $\sigma>\frac{\sqrt{3}}{2}$ established in \Cref{Section4} relied on having global well-posedness when $\frac{3}{2}<2s<4\sigma$, so we establish this first. Ultimately, global well-posedness will follow from the conservation laws, which we use in the next lemma to establish uniform control of the $H^1_x$ norm of solutions:
		\begin{lemma}($H_x^1$ norm remains bounded)\label{lifespan1}
Let $u_0\in H_x^{2s}$, $1\leq 2s<4\sigma$ and $\frac{\sqrt{3}}{2}<\sigma<1$. Let $T>0$ be sufficiently small. If $2s\leq \frac{3}{2}$, suppose that there is a corresponding well-posed solution $u\in X_T^{2s}$ to \eqref{gDNLS}. Likewise, if $4\sigma>2s>\frac{3}{2}$, let $u\in \mathcal{X}_T^{2s}$ be the corresponding well-posed solution to \eqref{gDNLS}. Then for $0\leq |t|\leq T$, we have
\begin{equation}
\|u(t)\|_{H_x^1}\lesssim_{\|u_0\|_{H_x^1}} 1
\end{equation}
where the implied constant depends only on the size of $\|u_0\|_{H_x^1}$. In particular, the $H_x^1$ norm of $u$ cannot blow up in finite time.
\end{lemma}
\begin{remark}
There is one small technical caveat to be aware of. Namely, in \Cref{lifespan1}, it is assumed  for $1\leq 2s\leq \frac{3}{2}$  that the equation \eqref{gDNLS} is locally well-posed $X_T^{2s}$. As mentioned above, this will follow from the results proven in \Cref{Section4} once we have established global well-posedness in the range $\frac{3}{2}<2s<4\sigma$ (where we already have local well-posedness from Section 6).
\end{remark}
\begin{proof}
Recall that we have the conserved mass and energy, respectively
\begin{equation}
M(u):=\frac{1}{2}\int_{\mathbb{R}}|u|^2dx=M(u_0),
\end{equation}
\begin{equation}
E(u):=\frac{1}{2}\int_{\mathbb{R}}|u_x|^2dx+\frac{1}{2(1+\sigma)}\RE\int_{\mathbb{R}}i|u|^{2\sigma}\overline{u}u_xdx=E(u_0).
\end{equation}
 It is also straightforward to verify that any well-posed solution in $\mathcal{X}_T^{2s}$ (when $\frac{3}{2}<2s<4\sigma$) or $X_T^{2s}$ (when $1\leq 2s\leq\frac{3}{2})$ satisfies these conservation laws. By interpolation, we have the following lower bound for the energy (where $C$ is some constant that may change from line to line) 
\begin{equation}
\begin{split}
E(u)&\geq \frac{1}{2}\|u_x\|_{L_x^2}^2-C\|u\|_{L_x^{4\sigma+2}}^{2\sigma+1}\|u_x\|_{L_x^2}
\\
&\geq \frac{1}{4}\|u_x\|_{L_x^2}^2-C\|u\|_{L_x^2}^{\frac{1+\sigma}{1-\sigma}}
\\
&\geq \frac{1}{4}\|u_x\|_{L_x^2}^2-CM(u)^{\frac{1+\sigma}{2(1-\sigma)}}.
\end{split}
\end{equation}
Hence, for $0\leq |t|\leq T$, we have
\begin{equation}
\|u(t)\|_{H_x^1}^2\lesssim E(u_0)+M(u_0)+M(u_0)^{\frac{1+\sigma}{2(1-\sigma)}}\lesssim_{\|u_0\|_{H_x^1}} 1.
\end{equation}
\end{proof}
\begin{corollary}\label{lifespan2}
Let $u_0\in H_x^{2s}$, $0<T^*<\infty$, $\frac{3}{2}<2s<4\sigma$ and $\frac{\sqrt{3}}{2}<\sigma<1$. Suppose that for each $T<T^*$, there is a corresponding well-posed solution $u\in\mathcal{X}_T^{2s}$ with initial data $u_0$. Then for each $0<\delta\ll 1$, we have
\begin{equation}
\limsup_{T\nearrow T^*}\|u\|_{\mathcal{S}_T^{1+\delta}\cap X_T^{2-\sigma+2\delta}}<\infty.
\end{equation}
In particular, the $\mathcal{S}_T^{1+\delta}\cap X_T^{2-\sigma+2\delta}$ norm of a solution cannot blow up in finite time. 
\end{corollary}
\begin{proof}
\Cref{lifespan1} shows that for all $0<T<T^*$, the norm $\|u\|_{L_T^{\infty}H_x^1}$ is bounded by a constant depending on the initial data $\|u_0\|_{H_x^1}$. Therefore, iterating (after appropriately translating and rescaling the initial data) \Cref{Ubounds} shows that
\begin{equation}\label{X1life}
\limsup_{T\nearrow T^*}\|u\|_{X_T^1}\lesssim_{\|u_0\|_{H_x^1}} 1.  
\end{equation}
By virtue of (\ref{X1life}) and iterating \Cref{Ubounds}, we find that
\begin{equation}
\limsup_{T\nearrow T^*}\|u\|_{X_T^{2-\sigma+2\delta}}<\infty.   
\end{equation}
It follows that
\begin{equation}
\limsup_{T\nearrow T^*}\|u\|_{\mathcal{S}_T^{1+\delta}}\leq \limsup_{T\nearrow T^*}\|u\|_{X_T^{2-\sigma+2\delta}}<\infty.    
\end{equation}

\end{proof}
Next, we use \Cref{lifespan2} and \Cref{reglem1} to establish global well-posedness in the high regularity regime $\frac{3}{2}<2s<4\sigma$. Indeed, for $u_0\in H_x^{2s}$ let $T^*>0$ be the maximal time for which there is a corresponding well-posed solution $u\in\mathcal{X}_T^{2s}$ for each $T<T^*$. If $T^*=\infty$, then we are done. We can therefore assume for the sake of contradiction that $T^*<\infty$. Then we have
\begin{equation}
\limsup_{T\nearrow T^*}\|u\|_{\mathcal{X}_T^{2s}}=\infty. 
\end{equation}
 We show that this is impossible. By rescaling and translation, we may without loss of generality take $T^*=1$.
\\
\\
We begin with the case $\frac{3}{2}<2s<2$. Set $\alpha=2-\sigma+2\delta$ where $\delta$ is some small positive constant.
\\
\\
Let $0<\epsilon\ll 1$. Define now the rescaled solution $u_{\lambda}(t,x)=\lambda^{\frac{1}{2\sigma}}u(\lambda^2t,\lambda x)$ to \eqref{gDNLS}, where $\lambda$ satisfies $k:=\lambda^{-2}\in\mathbb{N}$ and where $\lambda$ is small enough so that for each $T<\lambda^{-2}$, 
\begin{equation}\label{strichartzcontrol}
\|u_{\lambda}\|_{L_{T<\lambda^{-2}}^{\infty}H_x^{\alpha}}\lesssim \lambda^{\frac{1}{2\sigma}-\frac{1}{2}}\|u\|_{{L_{T<1}^{\infty}H_x^{\alpha}}}\lesssim \epsilon. 
\end{equation}
 By assumption $u_{\lambda}$ is a $\mathcal{X}_T^{2s}$ solution to \eqref{gDNLS} for $T<\lambda^{-2}$ with
 \begin{equation}\label{rescaledblowup}
 \limsup_{T\nearrow \lambda^{-2}}\|u_{\lambda}\|_{\mathcal{X}_T^{2s}}=\infty.
 \end{equation}
 Now, we iterate \Cref{reglem1}. We consider the initial value problem for each natural number $n<k$, 
 \begin{equation}
\begin{cases}
&(i\partial_t+\partial_x^2)w_n=i\eta |w_n|^{2\sigma}\partial_xw_n,
\\
&w_n(0)=u_{\lambda}(n).
\end{cases}     
 \end{equation}
 By \Cref{reglem1} by taking $\alpha=2-\sigma+2\delta$, and (\ref{strichartzcontrol}) there is a global solution $w\in C(\mathbb{R};H_x^{2s})$ to the above equation satisfying 
 \begin{equation}
 \|w_n\|_{\mathcal{X}_{T=2}^{2s}}^2\lesssim \frac{1}{1-C\epsilon^{2\sigma}}\|u_{\lambda}(n)\|_{H_x^{2s}}^2   
 \end{equation}
 from which we deduce (by restricting $w$ to times in $[-1,1]$), 
 \begin{equation}
\|u_{\lambda}(n+\cdot)\|_{\mathcal{X}_{T=1}^{2s}}^2\lesssim \frac{1}{1-C\epsilon^{2\sigma}}\|u_{\lambda}(n)\|_{H_x^{2s}}^2.
 \end{equation}
 Iterating this $k$ times gives the bound
 \begin{equation}
\|u_{\lambda}\|_{\mathcal{X}_{T<\lambda^{-2}}^{2s}}^2\lesssim \left(\frac{1}{1-C\epsilon^{2\sigma}}\right)^{k}\|u_{\lambda}(0)\|_{H_x^{2s}}^2.         
 \end{equation}
This contradicts (\ref{rescaledblowup}). Therefore $T^*=\infty$  and the $\mathcal{X}_T^{2s}$ norm cannot blow up in finite time when $\frac{3}{2}<2s<2$. 
\\
\\
Next, we proceed with the case $2\leq 2s<4\sigma$. If $2\leq 2s<3$, then if we assume a maximal time of existence $T^*<\infty$ for a $\mathcal{X}_T^{2s}$ solution, then the previous case shows that for $\delta>0$ sufficiently small,
\begin{equation}
\limsup_{T\nearrow T*}\|u\|_{\mathcal{X}_T^{2s-1+\delta}}<\infty. 
\end{equation}
Replacing $\alpha$ in the previous case with $\max\{2s-1+\delta,2-\sigma+2\delta\}$ and repeating the proof verbatim shows once again that $T^*=\infty$. Iterating once more  shows that in the case $3\leq 2s<4\sigma$, we also  have the same conclusion. Thus, \eqref{gDNLS} is globally well-posed in $H_x^{2s}$ when $\frac{3}{2}<2s<4\sigma$.
\\
\\
We finally turn to the last case. Namely, we show that $\eqref{gDNLS}$ is globally well-posed when $1\leq 2s\leq \frac{3}{2}$. 
\\
\\
Indeed, at this point, we know from Section 4 and the previous two cases that we have a locally well-posed $X_T^{2s}$ solution. Iterating the low regularity bounds \Cref{Ubounds} and using \Cref{lifespan1} shows that such a solution can be continued for all time. This finally completes the proof of \Cref{low reg theorem}.
\bibliographystyle{plain}
\bibliography{refs.bib}

\begin{thebibliography}{10}

\bibitem{agrawal2000nonlinear}
Govind~P Agrawal.
\newblock Nonlinear fiber optics.
\newblock In {\em Nonlinear Science at the Dawn of the 21st Century}, pages
  195--211. Springer, 2000.

\bibitem{ambrose2015local}
David~M. Ambrose and Gideon Simpson.
\newblock Local existence theory for derivative nonlinear {S}chr\"{o}dinger
  equations with noninteger power nonlinearities.
\newblock {\em SIAM J. Math. Anal.}, 47(3):2241--2264, 2015.

\bibitem{GWPDNLS}
Hajer Bahouri and Galina Perelman.
\newblock Global well-posedness for the derivative nonlinear {S}chr\"{o}dinger
  equation.
\newblock {\em Invent. Math.}, 229(2):639--688, 2022.

\bibitem{MR1837253}
H.~A. Biagioni and F.~Linares.
\newblock Ill-posedness for the derivative {S}chr\"{o}dinger and generalized
  {B}enjamin-{O}no equations.
\newblock {\em Trans. Amer. Math. Soc.}, 353(9):3649--3659, 2001.

\bibitem{MR1209299}
J.~Bourgain.
\newblock Fourier transform restriction phenomena for certain lattice subsets
  and applications to nonlinear evolution equations. {I}. {S}chr\"{o}dinger
  equations.
\newblock {\em Geom. Funct. Anal.}, 3(2):107--156, 1993.

\bibitem{MR1215780}
J.~Bourgain.
\newblock Fourier transform restriction phenomena for certain lattice subsets
  and applications to nonlinear evolution equations. {II}. {T}he
  {K}d{V}-equation.
\newblock {\em Geom. Funct. Anal.}, 3(3):209--262, 1993.

\bibitem{cazenave2017non}
Thierry Cazenave, Fl\'{a}vio Dickstein, and Fred~B. Weissler.
\newblock Non-regularity in {H}\"{o}lder and {S}obolev spaces of solutions to
  the semilinear heat and {S}chr\"{o}dinger equations.
\newblock {\em Nagoya Math. J.}, 226:44--70, 2017.

\bibitem{champeaux1999remarks}
S~Champeaux, D~Laveder, T~Passot, and PL~Sulem.
\newblock Remarks on the parallel propagation of small-amplitude dispersive
  {A}lfv{\'e}nic waves.
\newblock {\em Nonlinear Processes in Geophysics}, 6(3/4):169--178, 1999.

\bibitem{chen2020modulational}
Jinbing Chen, Dmitry~E. Pelinovsky, and Jeremy Upsal.
\newblock Modulational instability of periodic standing waves in the derivative
  {NLS} equation.
\newblock {\em J. Nonlinear Sci.}, 31(3):Paper No. 58, 32, 2021.

\bibitem{MR1124294}
F.~M. Christ and M.~I. Weinstein.
\newblock Dispersion of small amplitude solutions of the generalized
  {K}orteweg-de {V}ries equation.
\newblock {\em J. Funct. Anal.}, 100(1):87--109, 1991.

\bibitem{christ2003illposedness}
Michael Christ.
\newblock Illposedness of a {S}chr{\"o}dinger equation with derivative
  nonlinearity.
\newblock {\em preprint}, 2003.

\bibitem{colin2006stability}
Mathieu Colin and Masahito Ohta.
\newblock Stability of solitary waves for derivative nonlinear
  {S}chr\"{o}dinger equation.
\newblock {\em Ann. Inst. H. Poincar\'{e} Anal. Non Lin\'{e}aire},
  23(5):753--764, 2006.

\bibitem{colliander2001global}
J.~Colliander, M.~Keel, G.~Staffilani, H.~Takaoka, and T.~Tao.
\newblock Global well-posedness for {S}chr\"{o}dinger equations with
  derivative.
\newblock {\em SIAM J. Math. Anal.}, 33(3):649--669, 2001.

\bibitem{colliander2002refined}
J.~Colliander, M.~Keel, G.~Staffilani, H.~Takaoka, and T.~Tao.
\newblock A refined global well-posedness result for {S}chr\"{o}dinger
  equations with derivative.
\newblock {\em SIAM J. Math. Anal.}, 34(1):64--86, 2002.

\bibitem{fukaya2017instability}
Noriyoshi Fukaya.
\newblock Instability of solitary waves for a generalized derivative nonlinear
  {S}chr\"{o}dinger equation in a borderline case.
\newblock {\em Kodai Math. J.}, 40(3):450--467, 2017.

\bibitem{fukaya2021instability}
Noriyoshi Fukaya and Masayuki Hayashi.
\newblock Instability of degenerate solitons for nonlinear {S}chr\"{o}dinger
  equations with derivative.
\newblock {\em Nonlinear Anal.}, 222:Paper No. 112954, 25, 2022.

\bibitem{ginibre1989scattering}
J.~Ginibre and G.~Velo.
\newblock Scattering theory in the energy space for a class of nonlinear wave
  equations.
\newblock {\em Comm. Math. Phys.}, 123(4):535--573, 1989.

\bibitem{grillakis1987stability}
Manoussos Grillakis, Jalal Shatah, and Walter Strauss.
\newblock Stability theory of solitary waves in the presence of symmetry. {I}.
\newblock {\em J. Funct. Anal.}, 74(1):160--197, 1987.

\bibitem{grillakis1990stability}
Manoussos Grillakis, Jalal Shatah, and Walter Strauss.
\newblock Stability theory of solitary waves in the presence of symmetry. {II}.
\newblock {\em J. Funct. Anal.}, 94(2):308--348, 1990.

\bibitem{guo1995orbital}
Bo~Ling Guo and Ya~Ping Wu.
\newblock Orbital stability of solitary waves for the nonlinear derivative
  {S}chr\"{o}dinger equation.
\newblock {\em J. Differential Equations}, 123(1):35--55, 1995.

\bibitem{guo2017orbital}
Qing Guo.
\newblock Orbital stability of solitary waves for generalized derivative
  nonlinear {S}chr\"{o}dinger equations in the endpoint case.
\newblock {\em Ann. Henri Poincar\'{e}}, 19(9):2701--2715, 2018.

\bibitem{guo2018instability}
Zihua Guo, Cui Ning, and Yifei Wu.
\newblock Instability of the solitary wave solutions for the generalized
  derivative nonlinear {S}chr\"{o}dinger equation in the critical frequency
  case.
\newblock {\em Math. Res. Lett.}, 27(2):339--375, 2020.

\bibitem{guo2017global}
Zihua Guo and Yifei Wu.
\newblock Global well-posedness for the derivative nonlinear {S}chr\"{o}dinger
  equation in {$H^{\frac{1}{2}}(\Bbb{R})$}.
\newblock {\em Discrete Contin. Dyn. Syst.}, 37(1):257--264, 2017.

\bibitem{hakkaev2020all}
Sevdzhan Hakkaev, Milena Stanislavova, and Atanas Stefanov.
\newblock All non-vanishing bell-shaped solitons for the cubic derivative {NLS}
  are stable.
\newblock {\em arXiv preprint arXiv:2006.13658}, 2020.

\bibitem{hao2008well}
Chengchun Hao.
\newblock Well-posedness for one-dimensional derivative nonlinear
  {S}chr\"{o}dinger equations.
\newblock {\em Commun. Pure Appl. Anal.}, 6(4):997--1021, 2007.

\bibitem{harrop2022global}
Benjamin Harrop-Griffiths, Rowan Killip, Maria Ntekoume, and Monica Visan.
\newblock Global well-posedness for the derivative nonlinear {S}chr{\"o}dinger
  equation in {$ L^2(\mathbb{R})$}.
\newblock {\em arXiv preprint arXiv:2204.12548}, 2022, to appear in
  J.~Eur.~Math.~Soc.

\bibitem{H16}
Benjamin Harrop-Griffiths, Rowan Killip, and Monica Visan.
\newblock Large-data equicontinuity for the derivative {NLS}.
\newblock {\em Int. Math. Res. Not. IMRN}, (6):4601--4642, 2023.

\bibitem{hayashi2019long}
Masayuki Hayashi.
\newblock Long-period limit of exact periodic traveling wave solutions for the
  derivative nonlinear {S}chr\"{o}dinger equation.
\newblock {\em Ann. Inst. H. Poincar\'{e} Anal. Non Lin\'{e}aire},
  36(5):1331--1360, 2019.

\bibitem{hayashi2019studies}
Masayuki Hayashi.
\newblock {\em Studies on nonlinear Schr{\"o}dinger equations with derivative
  coupling}.
\newblock PhD thesis, Waseda University, 2019.

\bibitem{hayashi2020potential}
Masayuki Hayashi.
\newblock Potential well theory for the derivative nonlinear {S}chr\"{o}dinger
  equation.
\newblock {\em Anal. PDE}, 14(3):909--944, 2021.

\bibitem{hayashi2020stability}
Masayuki Hayashi.
\newblock Stability of {A}lgebraic {S}olitons for {N}onlinear {S}chr\"{o}dinger
  {E}quations of {D}erivative {T}ype: {V}ariational {A}pproach.
\newblock {\em Ann. Henri Poincar\'{e}}, 23(12):4249--4277, 2022.

\bibitem{hayashi2016well}
Masayuki Hayashi and Tohru Ozawa.
\newblock Well-posedness for a generalized derivative nonlinear
  {S}chr\"{o}dinger equation.
\newblock {\em J. Differential Equations}, 261(10):5424--5445, 2016.

\bibitem{HOV2025}
Masayuki Hayashi, Tohru Ozawa, and Nicola Visciglia.
\newblock Global {$ H^2$}-solutions for the generalized derivative {NLS} on
  {$\mathbb{T}$}.
\newblock {\em arXiv preprint arXiv:2406.06229}, 2024.

\bibitem{hayashi1993initial}
Nakao Hayashi.
\newblock The initial value problem for the derivative nonlinear
  {S}chr\"{o}dinger equation in the energy space.
\newblock {\em Nonlinear Anal.}, 20(7):823--833, 1993.

\bibitem{hayashi1992derivative}
Nakao Hayashi and Tohru Ozawa.
\newblock On the derivative nonlinear {S}chr\"{o}dinger equation.
\newblock {\em Phys. D}, 55(1-2):14--36, 1992.

\bibitem{Euler}
Mihaela Ifrim, Ben Pineau, Daniel Tataru, and Mitchell~A. Taylor.
\newblock Sharp {H}adamard local well-posedness, enhanced uniqueness and
  pointwise continuation criterion for the incompressible free boundary {E}uler
  equations.
\newblock {\em arXiv preprint arXiv:2309.05625}, 2023.

\bibitem{MHD_sharp}
Mihaela Ifrim, Ben Pineau, Daniel Tataru, and Mitchell~A. Taylor.
\newblock Sharp well-posedness for the free boundary {MHD} equations.
\newblock {\em arXiv preprint arXiv:2412.15625}, 2024.

\bibitem{ifrim2017well}
Mihaela Ifrim and Daniel Tataru.
\newblock Well-posedness and dispersive decay of small data solutions for the
  {B}enjamin-{O}no equation.
\newblock {\em Ann. Sci. \'{E}c. Norm. Sup\'{e}r. (4)}, 52(2):297--335, 2019.

\bibitem{primer}
Mihaela Ifrim and Daniel Tataru.
\newblock Local well-posedness for quasi-linear problems: a primer.
\newblock {\em Bull. Amer. Math. Soc. (N.S.)}, 60(2):167--194, 2023.

\bibitem{ifrim2020compressible}
Mihaela Ifrim and Daniel Tataru.
\newblock The compressible {E}uler equations in a physical vacuum: {A}
  comprehensive {E}ulerian approach.
\newblock {\em Ann. Inst. H. Poincar\'{e} C Anal. Non Lin\'{e}aire},
  41(2):405--495, 2024.

\bibitem{jenkins2019derivative}
Robert Jenkins, Jiaqi Liu, Peter Perry, and Catherine Sulem.
\newblock The derivative nonlinear {S}chr\"{o}dinger equation: global
  well-posedness and soliton resolution.
\newblock {\em Quart. Appl. Math.}, 78(1):33--73, 2020.

\bibitem{jenkins2020global}
Robert Jenkins, Jiaqi Liu, Peter Perry, and Catherine Sulem.
\newblock Global existence for the derivative nonlinear {S}chr\"{o}dinger
  equation with arbitrary spectral singularities.
\newblock {\em Anal. PDE}, 13(5):1539--1578, 2020.

\bibitem{kaup1978exact}
David~J. Kaup and Alan~C. Newell.
\newblock An exact solution for a derivative nonlinear {S}chr\"{o}dinger
  equation.
\newblock {\em J. Mathematical Phys.}, 19(4):798--801, 1978.

\bibitem{kenig2003local}
Carlos~E. Kenig and Kenneth~D. Koenig.
\newblock On the local well-posedness of the {B}enjamin-{O}no and modified
  {B}enjamin-{O}no equations.
\newblock {\em Math. Res. Lett.}, 10(5-6):879--895, 2003.

\bibitem{kenig1993well}
Carlos~E. Kenig, Gustavo Ponce, and Luis Vega.
\newblock Well-posedness and scattering results for the generalized
  {K}orteweg-de {V}ries equation via the contraction principle.
\newblock {\em Comm. Pure Appl. Math.}, 46(4):527--620, 1993.

\bibitem{kenig2006global}
Carlos~E. Kenig and Hideo Takaoka.
\newblock Global wellposedness of the modified {B}enjamin-{O}no equation with
  initial data in {$H^{1/2}$}.
\newblock {\em Int. Math. Res. Not.}, pages Art. ID 95702, 44, 2006.

\bibitem{killip2021well}
Rowan Killip, Maria Ntekoume, and Monica Vi\c~san.
\newblock On the well-posedness problem for the derivative nonlinear
  {S}chr\"odinger equation.
\newblock {\em Anal. PDE}, 16(5):1245--1270, 2023.

\bibitem{kwon2018orbital}
Soonsik Kwon and Yifei Wu.
\newblock Orbital stability of solitary waves for derivative nonlinear
  {S}chr\"{o}dinger equation.
\newblock {\em J. Anal. Math.}, 135(2):473--486, 2018.

\bibitem{li2018instability}
Bing Li and Cui Ning.
\newblock Instability of the solitary wave solutions for the generalized
  derivative nonlinear {S}chr\"odinger equation in the endpoint case.
\newblock {\em Nonlinear Anal.}, 252:Paper No. 113713, 14, 2025.

\bibitem{linares2019classII}
F.~Linares, G.~Ponce, and G.~N. Santos.
\newblock On a class of solutions to the generalized derivative
  {S}chr\"{o}dinger equations {II}.
\newblock {\em J. Differential Equations}, 267(1):97--118, 2019.

\bibitem{linares2019class}
Felipe Linares, Hayato Miyazaki, and Gustavo Ponce.
\newblock On a class of solutions to the generalized {K}d{V} type equation.
\newblock {\em Commun. Contemp. Math.}, 21(7):1850056, 21, 2019.

\bibitem{MR3952703}
Felipe Linares, Gustavo Ponce, and Gleison~N. Santos.
\newblock On a class of solutions to the generalized derivative
  {S}chr\"{o}dinger equations.
\newblock {\em Acta Math. Sin. (Engl. Ser.)}, 35(6):1057--1073, 2019.

\bibitem{MR3079669}
Xiao Liu, Gideon Simpson, and Catherine Sulem.
\newblock Stability of solitary waves for a generalized derivative nonlinear
  {S}chr\"{o}dinger equation.
\newblock {\em J. Nonlinear Sci.}, 23(4):557--583, 2013.

\bibitem{MR2955206}
Jeremy~L. Marzuola, Jason Metcalfe, and Daniel Tataru.
\newblock Quasilinear {S}chr\"{o}dinger equations {I}: {S}mall data and
  quadratic interactions.
\newblock {\em Adv. Math.}, 231(2):1151--1172, 2012.

\bibitem{MR4331023}
Jeremy~L. Marzuola, Jason Metcalfe, and Daniel Tataru.
\newblock Quasilinear {S}chr\"odinger equations {III}: {L}arge data and short
  time.
\newblock {\em Arch. Ration. Mech. Anal.}, 242(2):1119--1175, 2021.

\bibitem{mio1976modified}
Koji Mio, Tatsuki Ogino, Kazuo Minami, and Susumu Takeda.
\newblock Modified nonlinear {S}chr{\"o}dinger equation for {A}lfv{\'e}n waves
  propagating along the magnetic field in cold plasmas.
\newblock {\em Journal of the Physical Society of Japan}, 41(1):265--271, 1976.

\bibitem{mjolhus1976modulational}
Einar Mj{\o}lhus.
\newblock On the modulational instability of hydromagnetic waves parallel to
  the magnetic field.
\newblock {\em Journal of Plasma Physics}, 16(3):321--334, 1976.

\bibitem{moses2007self}
Jeffrey Moses, Boris~A Malomed, and Frank~W Wise.
\newblock Self-steepening of ultrashort optical pulses without
  self-phase-modulation.
\newblock {\em Physical Review A}, 76(2):021802, 2007.

\bibitem{mosincat2018unconditional}
Razvan Mosincat and Haewon Yoon.
\newblock Unconditional uniqueness for the derivative nonlinear
  {S}chr\"{o}dinger equation on the real line.
\newblock {\em Discrete Contin. Dyn. Syst.}, 40(1):47--80, 2020.

\bibitem{nahmod2012invariant}
Andrea~R. Nahmod, Tadahiro Oh, Luc Rey-Bellet, and Gigliola Staffilani.
\newblock Invariant weighted {W}iener measures and almost sure global
  well-posedness for the periodic derivative {NLS}.
\newblock {\em J. Eur. Math. Soc. (JEMS)}, 14(4):1275--1330, 2012.

\bibitem{ning2020instability}
Cui Ning.
\newblock Instability of solitary wave solutions for the nonlinear
  {S}chr\"{o}dinger equation of derivative type in degenerate case.
\newblock {\em Nonlinear Anal.}, 192:111665, 23, 2020.

\bibitem{ning2017instability}
Cui Ning, Masahito Ohta, and Yifei Wu.
\newblock Instability of solitary wave solutions for derivative nonlinear
  {S}chr\"{o}dinger equation in endpoint case.
\newblock {\em J. Differential Equations}, 262(3):1671--1689, 2017.

\bibitem{ohta2014instability}
Masahito Ohta.
\newblock Instability of solitary waves for nonlinear {S}chr\"{o}dinger
  equations of derivative type.
\newblock {\em SUT J. Math.}, 50(2):399--415, 2014.

\bibitem{passot1993multidimensional}
Thierry Passot and Pierre-Louis Sulem.
\newblock Multidimensional modulation of {A}lfv{\'e}n waves.
\newblock {\em Physical Review E}, 48(4):2966, 1993.

\bibitem{shimabukuro2017derivative}
Dmitry~E. Pelinovsky, Aaron Saalmann, and Yusuke Shimabukuro.
\newblock The derivative {NLS} equation: global existence with solitons.
\newblock {\em Dyn. Partial Differ. Equ.}, 14(3):271--294, 2017.

\bibitem{pelinovsky2018existence}
Dmitry~E. Pelinovsky and Yusuke Shimabukuro.
\newblock Existence of global solutions to the derivative {NLS} equation with
  the inverse scattering transform method.
\newblock {\em Int. Math. Res. Not. IMRN}, (18):5663--5728, 2018.

\bibitem{MR4830552}
Ben Pineau and Mitchell~A. Taylor.
\newblock Low regularity solutions for the general quasilinear ultrahyperbolic
  {S}chr\"odinger equation.
\newblock {\em Arch. Ration. Mech. Anal.}, 248(6):Paper No. 122, 67, 2024.

\bibitem{MR4820290}
Benjamin~Royce Pineau.
\newblock {\em On {L}ow {R}egularity {D}ynamics for {Q}uasilinear {D}ispersive
  {E}quations and {F}ree {B}oundary {P}roblems}.
\newblock ProQuest LLC, Ann Arbor, MI, 2024.
\newblock Thesis (Ph.D.)--University of California, Berkeley.

\bibitem{sanchez2010quasicollapse}
G~S{\'a}nchez-Arriaga, D~Laveder, T~Passot, and PL~Sulem.
\newblock Quasicollapse of oblique solitons of the weakly dissipative
  derivative nonlinear {S}chr{\"o}dinger equation.
\newblock {\em Physical Review E}, 82(1):016406, 2010.

\bibitem{santos2015existence}
Gleison do~N. Santos.
\newblock Existence and uniqueness of solution for a generalized nonlinear
  derivative {S}chr\"{o}dinger equation.
\newblock {\em J. Differential Equations}, 259(5):2030--2060, 2015.

\bibitem{shatah1983stable}
Jalal Shatah.
\newblock Stable standing waves of nonlinear {K}lein-{G}ordon equations.
\newblock {\em Comm. Math. Phys.}, 91(3):313--327, 1983.

\bibitem{takaoka1999well}
Hideo Takaoka.
\newblock Well-posedness for the one-dimensional nonlinear {S}chr\"{o}dinger
  equation with the derivative nonlinearity.
\newblock {\em Adv. Differential Equations}, 4(4):561--580, 1999.

\bibitem{MR1836810}
Hideo Takaoka.
\newblock Global well-posedness for {S}chr\"{o}dinger equations with derivative
  in a nonlinear term and data in low-order {S}obolev spaces.
\newblock {\em Electron. J. Differential Equations}, pages No. 42, 23, 2001.

\bibitem{tan2004blow}
Shao~Bin Tan.
\newblock Blow-up solutions for mixed nonlinear {S}chr\"{o}dinger equations.
\newblock {\em Acta Math. Sin. (Engl. Ser.)}, 20(1):115--124, 2004.

\bibitem{tao2001global}
Terence Tao.
\newblock Global regularity of wave maps. {II}. {S}mall energy in two
  dimensions.
\newblock {\em Comm. Math. Phys.}, 224(2):443--544, 2001.

\bibitem{tao2004global}
Terence Tao.
\newblock Global well-posedness of the {B}enjamin-{O}no equation in {$H^1({\bf
  R})$}.
\newblock {\em J. Hyperbolic Differ. Equ.}, 1(1):27--49, 2004.

\bibitem{tao2006nonlinear}
Terence Tao.
\newblock {\em Nonlinear dispersive equations: local and global analysis}.
\newblock Number 106. American Mathematical Soc., 2006.

\bibitem{MR4675424}
Mitchell~A. Taylor.
\newblock {\em A {C}ollection of {R}esults on {N}onlinear {D}ispersive
  {E}quations, {B}anach {L}attices and {P}hase {R}etrieval}.
\newblock ProQuest LLC, Ann Arbor, MI, 2023.
\newblock Thesis (Ph.D.)--University of California, Berkeley.

\bibitem{tsutsumi1980solutions}
Masayoshi Tsutsumi and Isamu Fukuda.
\newblock On solutions of the derivative nonlinear {S}chr\"{o}dinger equation.
  {E}xistence and uniqueness theorem.
\newblock {\em Funkcial. Ekvac.}, 23(3):259--277, 1980.

\bibitem{tsutsumi1981solutions}
Masayoshi Tsutsumi and Isamu Fukuda.
\newblock On solutions of the derivative nonlinear {S}chr\"{o}dinger equation.
  {II}.
\newblock {\em Funkcial. Ekvac.}, 24(1):85--94, 1981.

\bibitem{uchizono2012well}
Harunori Uchizono and Takeshi Wada.
\newblock On well-posedness for nonlinear {S}chr\"{o}dinger equations with
  power nonlinearity in fractional order {S}obolev spaces.
\newblock {\em J. Math. Anal. Appl.}, 395(1):56--62, 2012.

\bibitem{wu2014global}
Yifei Wu.
\newblock Global well-posedness for the nonlinear {S}chr\"{o}dinger equation
  with derivative in energy space.
\newblock {\em Anal. PDE}, 6(8):1989--2002, 2013.

\bibitem{wu2015global}
Yifei Wu.
\newblock Global well-posedness on the derivative nonlinear {S}chr\"{o}dinger
  equation.
\newblock {\em Anal. PDE}, 8(5):1101--1112, 2015.

\end{thebibliography}

\end{document}